\documentclass[final]{siamltex}
\usepackage{latexsym, graphicx, epsfig, amsmath, amsfonts,amssymb}
\usepackage{multirow}
 \usepackage{diagbox}
\usepackage{color}
\usepackage{tikz}
 \usepackage[percent]{overpic}
 \usepackage{subfigure}
\usepackage{algorithm}
\usepackage{algorithmicx}
\usepackage{algpseudocode}
\floatname{algorithm}{Algorithm}
\usepackage{threeparttable}
\usepackage{booktabs}
\usepackage{epstopdf,mathtools,bbm}
\usepackage[left=1.25in,right=1.25in,top=1in]{geometry}

 \newtheorem{remark}{Remark}
 \newtheorem{assumption}[theorem]{Assumption}
\def\mb{\mathbf}

\newcounter{Rownumber}

\title{Adaptive operator learning for infinite-dimensional Bayesian inverse problems}
\author
{
Zhiwei Gao\thanks{School of Mathematics, Southeast University, Nanjing 210096, China.}
\and Liang Yan\thanks{School of Mathematics, Southeast University, Nanjing 210096, China. Email: yanliang@seu.edu.cn. LY's work was supported by the NSF of China (Nos. 92370126, 12171085) and the Jiangsu Provincial Scientific Research Center of Applied Mathematics (Grant No. BK20233002).}
\and Tao Zhou\thanks{Institute of Computational Mathematics, Academy of Mathematics and Systems Science, Chinese Academy of Sciences, Beijing 100190, China. Email: tzhou@lsec.cc.ac.cn. TZ's work was supported by the National Key R\&D Program of China (2020YFA0712000), NSF of China (No. 12288201), the Strategic Priority Research Program of Chinese Academy of Sciences (Grant No. XDA25010404), the youth innovation promotion association (CAS), and Henan Academy of Sciences.}
} 
\begin{document}

\maketitle

\begin{abstract}
The fundamental computational issues in Bayesian inverse problems (BIP) governed by partial differential equations (PDEs) stem from the requirement of repeated forward model evaluations. A popular strategy to reduce such costs is to replace expensive model simulations with computationally efficient approximations using operator learning, motivated by recent progress in deep learning. However, using the approximated model directly may introduce a modeling error, exacerbating the already ill-posedness of inverse problems. Thus, balancing between accuracy and efficiency is essential for the effective implementation of such approaches. To this end, we develop an adaptive operator learning framework that can reduce modeling error gradually by forcing the surrogate to be accurate in local areas. This is accomplished by adaptively fine-tuning the pre-trained approximate model with training points chosen by a greedy algorithm during the posterior evaluation process. To validate our approach, we use DeepOnet to construct the surrogate and unscented Kalman inversion (UKI) to approximate the BIP solution, respectively. Furthermore, we present a rigorous convergence guarantee in the linear case using the UKI framework.  The approach is tested on a number of benchmarks, including the Darcy flow, the heat source inversion problem, and the reaction-diffusion problem.  The numerical results show that our method can significantly reduce computational costs while maintaining inversion accuracy. 
\end{abstract}

\begin{keywords}
Operator learning, DeepOnet, Bayesian inverse problems, Unscented Kalman inversion.
\end{keywords}

\pagestyle{myheadings}
\thispagestyle{plain}
\maketitle


\section{Introduction}
Many realistic phenomenons are governed by partial differential equations (PDEs), where the states of the system are described by PDEs solutions. The properties of these systems are characterized by the model parameters, such as the permeability and thermal conductivity, which can not be directly determined. Instead, the parameters can be inferred from the discrete and noisy observations of the states, which are known as  \textit{inverse problems}. Because inverse problems are ill-posed in general, many methods for solving them are based primarily on either regularization theory or Bayesian inference.  By imposing a prior distribution on the parameters, the Bayesian approach can provide a more flexible framework. The solutions to the \textit{Bayesian inverse problems}, or the \textit{posterior distributions}, can then be obtained by conditioning the observations using Baye's formula. In some cases, the model parameters are required to be functions, leading to \textit{infinite-dimensional Bayesian inverse problems}. These cases possibly occur when the model parameters are spatially varying with uncertain spatial structures, which can be found in many realistic applications, including many areas of engineering and science \cite{cui2016scalable,zhu2016bayesian,alexanderian2016fast,bui2013computational,petra2014computational}.

The formulation of infinite-dimensional Bayesian inverse problems presents a number of challenges, including the well-posedness guaranteed by the proper prior selection, as well as the convergence of the solutions governed by the discretization scheme. Following that, dealing with the discrete finite-dimensional posterior distributions can be difficult due to the expensive-to-solve forward models and high-dimensional parameter spaces.   Common methods to deal with these issues include (i) model reduction methods \cite{cui2015data,cui2016scalable,lieberman2010parameter}, which exploit the intrinsic low dimensionality of the governing physical systems, (ii) direct posterior approximation methods, such as Laplace approximation and variational inference \cite{schillings2020convergence,bui2013computational}, and (iii) surrogate modeling \cite{Conrad2016JASA,li2014adaptive,marzouk2007stochastic,yan2017convergence,yan2020adaptive}, which approximates the computationally expensive model with a more efficient, lower-cost alternative. 

Surrogate modeling emerges as the most promising approach for efficiently accelerating the sampling of posterior distributions among the methods listed above. Deep learning methods, specifically deep neural networks (DNN), have recently become the most popular surrogate models in engineering and science due to their power in approximating high-dimensional problems\cite{Han+Jentzen+E2018PNAS,raissi2019physics,Schwab+Zech2019AA,Tripathy+Bilionis2018JCP,Zhu+Zabaras2018bayesian}. In general, DNN employs the power of machine learning to construct a quick-to-evaluate surrogate model to approximate the {\it parameter-to-observation} maps\cite{deveney2019deep,yan2020adaptive,yanRTO2021}. Numerical experiments, such as those described in \cite{yan2020adaptive}, demonstrated that with sufficiently large training datasets, highly accurate approximations can be trained. Traditional deep learning methods, on the other hand, frequently necessitate a lot of training points that are not always available. Furthermore, whenever the measurement operator changes, the surrogate should be retrained. Physical-informed neural networks(PINNs)\cite{raissi2019physics} can address this issue by incorporating the physical laws into the loss function and learning the {\it parameter-to-state} maps\cite{li2023surrogate,nabian2020adaptive}. As a result, they can be used as surrogates for a variety of Bayesian inverse problems involving models governed by the same PDEs but with different types of observations operators, which will further reduce the cost of surrogate construction.  However, PINNs have some limitations\cite{wang2022and,krishnapriyan2021characterizing}, such as hyperparameter sensitivity and the potential for training instability due to the hybrid nature of their loss function. Several solutions have been proposed to address these issues\cite{gao2023failure,gao2023rFINN,CSIAM-AM-5-636,mcclenny2020self,xiang2022self}.  Operator neural networks, such as FNO\cite{li2020fourier} and DeepOnet\cite{lu2021learning}, are able to model complex systems in high-dimensional spaces as infinite-dimensional approximations.  They are therefore promising surrogates, as described in\cite{cao2023residual,genzel2022solving}. However, using approximate models directly may introduce a discrepancy or modeling error, exacerbating an already ill-posed problem and leading to a worse solution. 

In order to reduce model errors, several approaches\cite{Conrad2016JASA,li2014adaptive,yan2020adaptive,yanRTO2021,cleary2021calibrate,yan2021stein} that incorporate local approximation techniques have been applied to Bayesian posterior sampling problems.  For example,  Conrad et al.\cite{Conrad2016JASA} described a framework for building and refining local approximation during an MCMC simulation using either local polynomial approximations or local Gaussian process regressors(GPR). To generate the sample sets used for local approximations, the authors employ a sequential experimental design procedure that interleaves infinite refinement of the approximation with  Markov chain posterior exploration.  Recently, Cleary et al. \cite{cleary2021calibrate} proposed a ``Calibrate-Emulate-Sample"(CES) framework for approximate Bayesian inversion. This approach consists of three main steps: To begin, the ensemble Kalman method\cite{Garbuno2020interacting,iglesias2013ensemble} is used in conjunction with a full-order model to extract sample points from the exact posterior distribution. Next, these sample points are used to create a GPR emulator for the {\it parameter-to-observation} map. Finally, the approximated posterior obtained from the emulator is sampled using direct sampling methods such as MCMC. A similar idea is proposed in \cite{yanRTO2021},  which use a goal-oriented DNN surrogate approach to significantly reduce the computational burden of posterior sampling.  Specifically, they begin by using a Laplace approximation to approximate the posterior distribution. After that, they select training points to build a DNN surrogate from this approximate distribution. In the last stage, direct sampling methods are used to sample from the approximate posterior.  However, these works primarily focus on building  local approximations of the {\it parameter-to-observation} maps.  One significant limitation of this approach is that the calibrate phase (i.e., generating the posterior approximation) necessitates extensive computations with full-order models, which can be computationally expensive. Furthermore, if the observation data changes, the entire process must be restarted from scratch, including recalculating the posterior approximation and reconstructing the local approximation. This can reduce the framework's efficiency and increase its resource requirements, especially in dynamic environments where observations change frequently.  Inspired by these works, we present a mutual learning framework that can reduce model error by forcing the approximate model to be locally accurate for posterior characterization during the inversion process. This is achieved by first using neural network representations of {\it parameter-to-state} maps between function spaces, and then fine-tuning this initial model with points chosen adaptively from the approximate posterior.  This procedure can be repeated multiple times as necessary until the stop criteria is met.    In contrast to the CES  framework, we choose training points from the prior distribution to 
 build the initial emulator offline. When measurement data is available, the surrogate model only needs to be fine-tuned during the sampling process, which involves  specific sampling methods.  The integration of prior knowledge into the initial model construction enhances the robustness and accuracy of the model, while the targeted fine-tuning during sampling improves computational efficiency, making the approach particularly suitable for scenarios requiring fast and reliable decision-making. For the detailed implementation, we use the DeepOnet\cite{lu2021learning} to approximate the {\it parameter-to-state} maps and the unscented Kalman inversion \cite{huang2022iterated} to estimate the posterior distribution. Moreover, we can show that under the linear case, convergence can be obtained if the surrogate is accurate throughout the space, which can also be extended to non-linear cases with locally accurate approximate models. To demonstrate the effectiveness of our method, we propose testing several benchmarks such as the Darcy flow, the heat source inversion problem, and a reaction-diffusion problem.  Our main contributions can be summarized as follows.
\begin{itemize}
    \item We propose a framework for adaptively reducing the surrogate's model error. To maintain local accuracy, the greedy algorithm is proposed for selecting adaptive samples for fine-tuning the pre-trained model. 
    \item We adopt DeepOnet to approximate the {\it parameter-to-state} and then combine the UKI to accelerate infinite-dimensional Bayesian inverse problems. We demonstrate that this approach not only maintains inversion accuracy but also saves a significant amount of computational cost.
    \item We show that in the linear case, the mean vector and the covariance matrix obtained by UKI with an approximate model can converge to those obtained with a full-order model. The results can also be verified in non-linear cases with locally accurate surrogates.
    \item We present several benchmark tests including the Darcy flow, a heat source inversion problem and a reaction-diffusion problem to verify the effectiveness of our approach.
\end{itemize}
The remainder of this paper is organized as follows.  Section \ref{section_bayesian} introduces  infinite-dimensional Bayesian inverse problems as well as the basic concepts of DeepOnet.  Our adaptive framework for model error reduction equipped with greedy algorithm and the unscented Kalman inversion, is presented in Section \ref{section_adaptive_model}.  To confirm the efficiency of our algorithm, several benchmarks are tested in Section \ref{numerical_experiments}. The conclusion is covered in Section \ref{conclusion}.

\section{Background}\label{section_bayesian}

In this section, we first give a brief review of the infinite-dimensional Bayesian inverse problems. Then we will introduce the basic concepts of DeepOnet. 

\subsection{Infinite-dimensional Bayesian inverse problems}

Consider a steady physical system described by the following PDEs:
\begin{equation}
    \label{Problem}
    \begin{cases}
        &\mathcal{H}(u(\mathbf{x});m(\mathbf{x})) = 0, \quad \mathbf{x}\in \Omega,\\ 
        &\mathcal{B}(u(\mb{x}); m(\mathbf{x})) = 0,\quad \mathbf{x}\in \partial \Omega,
    \end{cases}
\end{equation}
where $\mathcal{H}$ denotes the general partial differential operator defined in the domain $\Omega \in \mathbb{R}^{d}$, $\mathcal{B}$ is the boundary operator on the boundary $\partial \Omega$, $m\in \mathcal{M}$ represents the unknown parameter field and $u\in \mathcal{U}$ represents the state field of the system.

Let $y\in \mathbb{R}^{N_{y}}$ denotes a set of discrete and noisy observations at specific locations in $\Omega$. Suppose the state $u$ and $y$ are connected through an observation system $\mathcal{O}:\mathcal{U}\rightarrow \mathbb{R}^{N_{y}}$,
\begin{equation}\label{obs_sys}
    y = \mathcal{O}(u) + \eta, 
\end{equation}
where $\mb{\eta}\sim \mathcal{N}(0, \Sigma_{\eta})$ is a Gaussian with mean zero and covariance matrix $\Sigma_{\eta}$, which models the noise in the observations. Combining the PDE model \eqref{Problem} and the observation system \eqref{obs_sys} defines the \textit{parameter-to-observation} map $\mathcal{G} = \mathcal{O}\circ \mathcal{F}:\mathcal{M}\rightarrow \mathbb{R}^{N_{y}}$, i.e, 
\begin{equation*}
    \label{data_model}
    y = \mathcal{G}(m) +\mb{\eta}.
\end{equation*}
Here  $\mathcal{F}:\mathcal{M}\rightarrow \mathcal{U}$ is the solution operator, or the \textit{parameter-to-state} map, of the PDE model \eqref{Problem}. 

The following least squared functional plays an important role in such inverse problems:
\begin{equation}
    \label{least-sqaured-error}
    \Phi(m;y) = \frac{1}{2}\left\|y - \mathcal{G}(m)\right\|^{2}_{\Sigma_{\eta}},
\end{equation}
where $\|\cdot\|_{\Sigma_{\eta}} = \|\Sigma_{\eta}^{-\frac{1}{2}}\cdot\|$ denotes the weighted Euclidean norm in $\mathbb{R}^{N_{y}}$. In cases where the inverse problem is ill-posed, optimizing   $\Phi$ in $\mathcal{M}$ is not a well-behaved problem, and some type of regularization is necessary. Bayesian inference is another method to consider. In the Bayesian framework, $(m, y)$ is viewed as a jointly varying random variable in $\mathcal{M}\times \mathbb{R}^{N_{y}}$. Given the prior $\nu_0$ on $m$, the solution to the inverse problem is to determine the distribution of  $m$ conditioned on the data $y$, i.e., the posterior $\nu$  given by an infinite dimensional version of  \textit{Bayes' formula} as 
\begin{equation}\label{bayesian_formula}
   \nu(\textcolor{black}{\mathrm{d}m})=\frac{1}{Z(y)}\exp(-\Phi(m;y))\nu_0(\textcolor{black}{\mathrm{d}m}),
\end{equation}
where  $Z(y)$ is the model evidence defined as 
\begin{equation*}
    Z(y) := \int_{\mathcal{M}}\exp(-\Phi(m;y))\nu_{0}(\textcolor{black}{\mathrm{d}m}).
\end{equation*}

In general, the main challenge of \textit{infinite-dimensional Bayesian inverse problems} lies in well-posedness of the problem and numerical methodologies. To guarantee the well-posedness, the prior is frequently considered to be a Gaussian random field, which guarantees the existence of the posterior distribution\cite{bui2013computational,stuart2010inverse}.  To obtain the finite posterior distributions, one can use Karhunen-Loeve (KL) expansions or direct spatial discretization methods. The posterior distribution can then be approximated using numerical techniques like Markov Chain Monte Carlo (MCMC)\cite{Brooks2011} and variational inference (VI)\cite{Blei2017variational}. It should be emphasized that each likelihood function evaluation requires a forward model $\mathcal{G}$ (or $\mathcal{F}$) evaluation.  The computation of the forward model can be very complicated and expensive in some real-world scenarios, making the computation challenging. As a result, it is critical to replace the forward model with a low-cost surrogate model. In this paper, we apply deep operator learning to construct the surrogate in order to substantially reduce computational time.

\subsection{DeepOnet as surrogates}
In this section, we employ the \textit{neural operator} DeepOnet as the surrogate, which is fast to evaluate and can speed up in the \textcolor{black}{posterior evaluations}. The basic idea is to approximate the true forward operator $\mathcal{F}$ with a neural network $\mathcal{F}_{\theta}:\mathcal{M}\rightarrow \mathcal{U}$, where $\mathcal{M}, \mathcal{U}$ are spaces defined before and $\theta$ are the parameters in the neural network.
This neural operator can be interpreted as a combination of \textit{encoder} $\mathcal{E}$, \textit{approximator} $\mathcal{A}$ and  \textit{reconstructor} $\mathcal{R}$ \cite{lanthaler2022error} as depicted in Figs.\ref{encode} and \ref{DeepOnet_figure}, i.e., 
\begin{equation*}
    \label{DeepOnet}
    \mathcal{F}_{\theta}:= \mathcal{R}\circ \mathcal{A}\circ \mathcal{E}.
\end{equation*}
Here, the {\it encoder} $\mathcal{E}$ maps $m$ into discrete values $\{m(\mb{x}_{i})\}_{i=1}^{N_{m}}$ in $\mathbb{R}^{N_{m}}$ at a fixed set of sensors $\{\mb{x}_{i}\}_{i=1}^{N_{m}}\in \Omega$, i.e.,
$$
\mathcal{E}: \mathcal{M} \rightarrow \mathbb{R}^{N_{m}}, \,\, \mathcal{E}(m) =(m(\mb{x}_1),\cdots,m(\mb{x}_{N_m})).
$$
The encoded data is then approximated by the {\it approximator} $\mathcal{A}:\mathbb{R}^{N_{m}}\rightarrow \mathbb{R}^{p}$ through a deep neural network. Given the encoder and approximator, we can define the \textit{branch net} $\mb{\beta}: \mathcal{M} \rightarrow \mathbb{R}^p$ as the composition $\mb{\beta}(m)=\mathcal{A}\circ \mathcal{E}(m)$. The {\it decoder} $\mathcal{R}:\mathbb{R}^{p}\rightarrow \mathcal{U}$ maps the results $\{\beta_{i}\}_{i=1}^{p}$ to $\mathcal{U}$ with the form 
\begin{equation*}
    \label{reconstruction}
    \mathcal{R}(\beta) = \sum_{i=1}^{p}\beta_{i}t_{i}(\mb{x}),\quad \mb{x}\in \Omega, 
\end{equation*}
where $t_{i}$ are the outputs of the \textit{trunk net} as depicted in Fig.\ref{DeepOnet_figure}. Note that by this formula, the \textit{trunk} net approximates the basis of the solution space and the \textit{branch} net approximates the coefficients of the expansion, which together approximate the spectral expansion of the solution space. 

Combined with the branch net and trunk net, the operator network approximation $\mathcal{F}_{\theta}(m)(\mb{x})$ is obtained by finding the optimal $\theta$, which minimizes the following loss function:
\begin{equation}
    \label{loss_function}
\theta^{*} = \arg \min_{\theta \in \Theta} \mathcal{L}(\theta) =\int_{\mathcal{M}}\int_{\Omega}|\mathcal{F}(m)(\mb{x})-\mathcal{F}_{\theta}(m)(\mb{x})|^{2}\,\mathrm{d}\mb{x}\,\nu_{0}(\textcolor{black}{\mathrm{d}m}),
\end{equation}
where $\Theta$ is the parameter space.
It should be noted that the loss function cannot be computed exactly and is usually approximated by Monte Carlo simulation by sampling the space $\mathcal{M}$ and the input sample space $\Omega$. That is, we  take $N_{prior}$ i.i.d samples $m_{1}, m_{2},\cdots, m_{N_{prior}}\sim \nu_{0}$  at $N_{\mb{x}}$ points $\mb{x}_{j}^{1},\cdots, \mb{x}_{j}^{N_{\mb{x}}}$, leading to the following \textit{empirical loss} 
\begin{equation}
    \label{empirical_loss}
    \widehat{\mathcal{L}}_{N_{prior},N_{\mb{x}}}({\theta}):=\frac{1}{NN_{\mb{x}}}\sum_{j=1}^{N}\sum_{k=1}^{N_{\mb{x}}}\left|{\mathcal{F}}(m_{j})(\mb{x}_{j}^{k})-{\mathcal{F}_{\theta}}(m_{j})(\mb{x}_{j}^{k})\right|^{2}.
\end{equation} 
\begin{figure}
    \centering 
    \includegraphics[width = 0.25\textwidth]{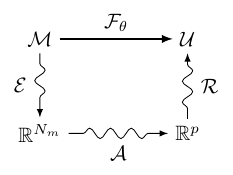}
    \caption{Diagram of the three components of DeepOnet.}
    \label{encode}
\end{figure}

\begin{figure}
    \centering 
    \includegraphics[width = 0.65\textwidth]{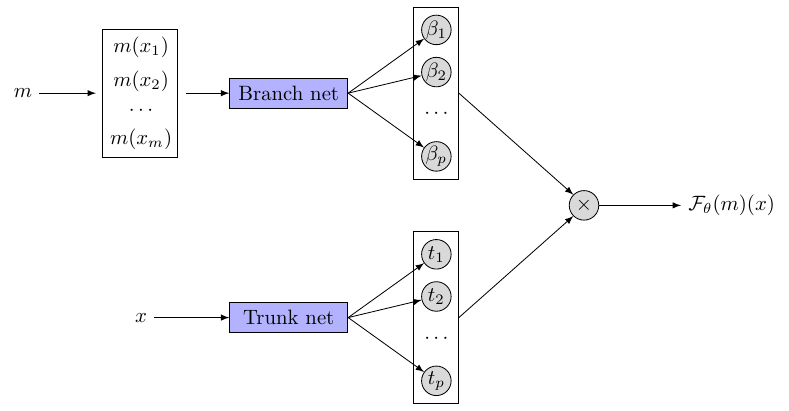}
    \caption{The framework of DeepOnet.}
    \label{DeepOnet_figure}
\end{figure}

After the operator network has been trained, an approximation of the forward model $\mathcal{G}$ can be constructed by adding the observation operator $\mathcal{O}$, i.e., $\widehat{\mathcal{G}} = \mathcal{O}\circ \mathcal{F}_{\theta}$. We then can obtain the surrogate posterior
\begin{equation}\label{apprpost}
\widehat{\nu}(\textcolor{black}{\mathrm{d}m})\propto \exp(-\widehat{\Phi}(m;y))\nu_0(\textcolor{black}{\mathrm{d}m}),
\end{equation}
where $\nu_0$ is again the prior of $m$ and $\widehat{\Phi}(m;y)$ is the approximate least-squares data misfit defined as 
\begin{equation*}
\widehat{\Phi}(m;y): =\frac{1}{2}\left\|y - \widehat{\mathcal{G}}(m)\right\|^{2}_{\Sigma_{\eta}}.
\end{equation*}

The main advantage of the surrogate method is that once an accurate approximation is obtained, it can be evaluated many times without resorting to additional simulations of the full-order forward model. However, using approximate models directly may introduce a discrepancy or modeling error, exacerbating an already ill-posed problem and leading to a worse solution\cite{yan2020adaptive}. Specifically, we can define an $\epsilon$-feasible set $\mathcal{M}(\epsilon):=\{m\in \mathcal{M} | \|\mathcal{G}(m)-\widehat{\mathcal{G}}(m)\|\leq \epsilon\},$ and the associated posterior measure  $ \nu(\mathcal{M}(\epsilon))$ as
 $$
\nu(\mathcal{M}(\epsilon)) = \int_{\mathcal{M}(\epsilon)}\nu(\textcolor{black}{\mathrm{d}m}).
$$
Then, the complement of the $\epsilon$-feasible set is given by $\mathcal{M}^{\bot}(\epsilon)=\mathcal{M}\setminus \mathcal{M}(\epsilon)$, which has posterior measure $\nu(\mathcal{M}^{\bot}(\epsilon))=1-\nu(\mathcal{M}(\epsilon))$.
We can obtain an error bound between $\widehat{\nu}$ and $\nu$ in the Kullback-Leibler  distance:
\begin{theorem}[\cite{yan2019adaptive1}]
    \label{Theorem2.1}
Suppose we have the full posterior distribution $\nu$ and its approximation  $\widehat{\nu}$ induced by the surrogate $\widehat{\mathcal{G}}$. For a given $\epsilon$, there exist constants $K_1>0$ and $K_2>0$ such that
\begin{equation*}\label{klbound}
D_{KL}(\nu\|\widehat{\nu}) \leq \Big(K_1 \epsilon +K_2 \nu(\mathcal{M}^{\bot}(\epsilon))\Big)^2.
\end{equation*}
\end{theorem}

It is important to note that in order for the approximate posterior $\widehat{\nu}$  to converge to the exact posterior $\nu$, the posterior measure $\nu(\mathcal{M}^{\bot}(\epsilon))$ must tend to zero.  One way to achieve this goal is to enable the surrogate model  $\widehat{\mathcal{G}}$  trained sufficiently  in the entire input space such that the model error is small enough.   However, a significant amount of data and training time are frequently required  to effectively train the surrogate model. Indeed, the surrogate model only needs to be accurate within the posterior distribution space, not the entire prior space\cite{cao2023residual,yan2020adaptive,yan2019adaptive1}.   To maintain accurate results while lowering the computational costs, an adaptive algorithm should be developed.  In the following section, we will look at how to design a framework for adaptively reducing the surrogate's modeling error.  

\section{Adaptive operator learning framework}
\label{section_adaptive_model}

\subsection{Adaptive model error reduction}

Developing stable and reliable adaptive surrogate modeling methods presents several challenges, especially when dealing with infinite-dimensional Bayesian inverse problems.  One major challenge lies in the need to maintain the accuracy of the surrogate model in the high-density regions of the posterior distribution, where the true posterior is most concentrated.  According to Theorem \ref{Theorem2.1}, the approximate posterior will be close to the true posterior when the surrogate is accurate in the high density region of the posterior distribution. On the other hand,  the accuracy of the surrogate model, especially in the context of an operator network like DeepOnet, depends heavily on the quality and distribution of the training set.  Ideally, this training set should be well-distributed across the posterior distribution in order to capture the key characteristics of the problem.  However, the high density region of the posterior distribution is unknown until the observations are given. One feasible approach, as discussed earlier, is the CES framework \cite{cleary2021calibrate}, which first involves a calibration step where posterior samples are obtained using the full-order model, and then these samples can be  used to train the surrogate. Nonetheless, this strategy  can be very computationally expensive, especially when new observational data emerges, as the entire process must be repeated. A natural question is how to create an effective local approximation that maintains a good balance between accuracy and computational cost.  To achieve this, the new adaptive framework should be  designed to automatically select new training points during the posterior exploration process. These selected points are then used to fine-tune the emulator generated by DeepOnet. Instead of relying on the full-order model to explore the posterior space initially, the training points for the emulator are sequentially updated by integrating samplers that utilize the emulator itself. Based on this purpose, we separate the adaptive algorithms into offline and online stages. Offline, a small amount of samples from the prior distribution are utilized to train the DeepOnet in an acceptable length of time. The purpose of offline computation is to obtain a rough but comprehensive approximation of the DeepOnet model.  The online stage consists of two major steps. First, decide whether to adaptively update the surrogate model with the posterior computation algorithm. Second, if the surrogate model needs to be updated, choose a new training set locally.  Our new approach is an online strategy that depends on a specific realization of the observed data and an associated posterior approximation, in contrast to the direct DeepOnet strategy which maintains the surrogate model $\widehat{\mathcal{G}}$ unchanged. We expect that the inversion results computed by adaptive DeepOnet will have better accuracy than those produced by the offline direct DeepOnet strategies because using the data centers attention on regions of high posterior probability and causes a localization effect in the construction of surrogate. We will explore this conjecture in numerical results below.
 
 \begin{figure}[t]
    \centering
    \includegraphics[width = 0.8\textwidth]{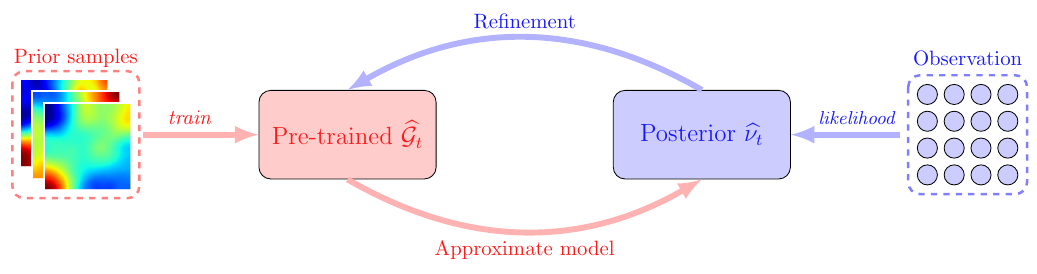}
    \caption{The flowchart for our mutual learning framework. The framework begins with an offline, pre-trained operator network. The operator network combines with observations to efficiently return likelihood information for samplers to explore the parameter space, while the samplers provide useful candidate points for surrogate refinement. They can learn from and improve each other over time. }
    \label{Flowchart_surroagte}
\end{figure}

 Our approach, as depicted in Fig.\ref{Flowchart_surroagte} is indeed the summary of the previous efforts. The procedure is broken down into the following steps, which are modular in nature and can be approached in various ways:

\begin{itemize}
    \item Initialization (offline): Build a surrogate $ \mathcal{F}_{\theta}$ using the initial training dataset $\mathcal{D}$ with a relatively small sample size. This model is used as the initial pre-trained model. 
    \item Posterior computation:  Using some numerical techniques to approximate, or draw samples from the approximate posterior $\widehat{\nu}$ induced by $\widehat{\mathcal{G}}=\mathcal{O}\circ  \mathcal{F}_{\theta}$. 
    \item Refinement (online): Choose a  criteria to determine whether the refinement is needed. If refinement is needed, then selecting new training points from $\widehat{\nu}$ to refine the training dataset $\mathcal{D}$ and the surrogate $\mathcal{F}_{\theta}$.
    \item Repeat the above procedure several times until the stop criteria is met.
\end{itemize}

The rest of this part turns this outline into a workable algorithm by explaining when and where to refine, as well as how to choose new training data to improve approximations.  To create an initialization surrogate, we can start by creating a small training dataset $\mathcal{D}$ from the prior distribution.  In detail, suppose we have a collection of model evaluations $\mathcal{D} = \{(m_{i}, \mathcal{F}(m_{i}))\}$. We then can use these points to train an operator network $\mathcal{F}_{\theta}$ and obtain a surrogate $\widehat{\mathcal{G}}=\mathcal{O}\circ \mathcal{F}_{\theta}$,  which can be evaluated repeatedly at negligible cost. Subsequently, we decide where and when to refine the surrogate by conducting exploration using the current surrogate. Specifically, assume that $\widehat{\nu}_{t}$ is the approximate posterior at $t$ step.  To assess the accuracy of the surrogate model in the approximate posterior space, we define the following ``local" model error assist with $\widehat{\nu}_{t}$:
\begin{equation}\label{localerr}
    e_{\mathcal{M}}(t) := \mathbb{E}\left[\mathcal{G} - \mathcal{\widehat{G}}_{t}\right] = \left(\int_{\mathcal{M}}|\mathcal{G}(m) - \mathcal{\widehat{G}}_{t}(m)|^{2}\widehat{\nu}_{t}(\textcolor{black}{\mathrm{d}m})\right)^\frac{1}{2}.
\end{equation}
A natural approach is to use this error as an indicator: if the error exceeds a predefined tolerance, the surrogate model should be refined. Unfortunately, the need for high-dimensional integration makes directly evaluating this error in high-dimensional spaces prohibitively expensive. Notice that our primary goal is to maintain the surrogate accurate along the posterior computation trajectories.  To achieve this, we can now obtain $T$ samples $\mb{M}^{(t)}:=\{m_k^{(t)}\}_{k=1}^{T}$ from  $\widehat{\nu}_{t}$ using a posterior computation method, such as particle methods or MCMC-based sampling methods. Importantly, this process only relies on information from the surrogate model  $\mathcal{\widehat{G}}_{t}$  and the data  $y$, without requiring any information from the full-order model  $\mathcal{G}$. To prevent the current samples from deviating too far from the true posterior trajectory, we define an anchor point  $r^t$  from the obtained samples $\mb{M}^{(t)}$ using the full-order model $\mathcal{G}$ as follows:
\begin{equation}
r^t = \arg\min_{m\in \mb{M}^{(t)}}\frac{1}{2}\|y - \mathcal{G}(m)\|_{\Sigma_{\eta}}^{2}.
\end{equation}
This allows us to define an error indicator  $e_{\mathcal{D}}(t)$  based on the data-fitting term at the anchor point  $r^t$  within the posterior sample set:
 \begin{equation}
        \label{data_fitting_error}
        e_{\mathcal{D}}(t):= \Phi(r^t;y) = \frac{1}{2}\|y - \mathcal{G}(r^t)\|_{\Sigma_{\eta}}^{2}. 
    \end{equation}
 This error indicator not only measures the model error but also ensures that the samples do not deviate significantly from the true posterior trajectory. When the relative error $|\frac{e_{\mathcal{D}}(t) - e_{\mathcal{D}}(t-1)}{e_{\mathcal{D}}(t)}|$   exceeds   a tolerance of $\epsilon$, the surrogate must be refined in $\widehat{\nu}_t$. Otherwise, the refinement process stops. 

The following task is to develop  sampling strategies for surrogate refinement. We only need to ensure that the surrogate is accurate along the \textcolor{black}{posterior evaluation} trajectories, so we can generate a small set of important (adaptive) points near the anchor point and then refine the surrogate.  This ensures that the refined surrogate $\widehat{\mathcal{G}}_{t}$ is locally accurate in $\widehat{\nu}_{t}$ while also reducing computational costs. This topic is relevant to optimal experimental design. A number of existing strategies can be used to accomplish this goal. In this work, we propose a greedy algorithm for efficiently selecting the most important samples from $\widehat{\nu}_{t}$.   Specifically, we first draw a large set of $K$ samples $\Gamma = \{m_{1},m_{2},\cdots,m_{K}\}$ from $\widehat{\nu}_t$.  A greedy algorithm is then used to select a subset $\gamma^{Q} = \{\widetilde{m}_{1},\widetilde{m}_{2},\cdots, \widetilde{m}_{Q}\}\subset \Gamma$ of ``important" points one by one from $\Gamma$. Assuming the current selected point sets are $\gamma^{j}(j<Q)$, then the newly selected point $\widetilde{m}_{j+1}$ must be near the anchor point $r^t$. Furthermore, the surrogate solution for this point should have the greatest distance with the set $\widehat{\mathcal{G}}_t^j:=\{\widehat{\mathcal{G}}_t(\widetilde{m}_{i})\}_{i=1}^{j}$, 
i.e., \begin{equation}\label{upeq}
   \widetilde{m}_{j+1} =  \arg \max_{m\in \Gamma/\gamma^{j}} \left(d\big(\widehat{\mathcal{G}}_t(m), \widehat{\mathcal{G}}_t^j\big)-\lambda \|m - r^{t}\|_{2}\right),
 \end{equation}
 where $d\big(\widehat{\mathcal{G}}_t(\cdot), \widehat{\mathcal{G}}_t^j\big) = \max_{\tilde{m}\in \gamma^{j}} \left\|\widehat{\mathcal{G}}_t(\cdot) - \hat{\mathcal{G}}_t(\tilde{m})\right\|_{2}$ is the distance between the value $\widehat{\mathcal{G}}_t(\cdot)$ and the set $\widehat{\mathcal{G}}_t^j$. Here $\lambda$ is the factor used to control the balance between these two distances and is chosen to be 1 for convenience. The process only calculates the distance between points and predicted values of the surrogate model on the preselected data set, resulting in a negligible calculation time.  Then, we have $\gamma^{j+1} = \gamma^{j}\cup \{\widetilde{m}_{j+1}\}$.  We can improve the DeepOnet model $\mathcal{F}_{\theta}$ by collecting the new training set $\gamma^{Q}$. Specifically, during the online training phase, we initialize the neural network parameter from the previously trained model $\mathcal{F}_{\theta}$.  We expect a significant speedup in solving Eq. (\ref{loss_function}) with this initialization, which can be viewed as an example of transfer learning.
 
The advantages of this method are obvious. Specifically, the selected adaptive points will have varying features in the surrogate solution spaces and will be close to the anchor point $r^t$. This is more of an adversarial process; in order to guarantee local accuracy, we want the new points to be close to the $r^t$. In order to enhance generalization capacity, we also want them to incorporate as much data as the current surrogate did. The only remaining question is how to approximate the approximate posterior distribution $\widehat{\nu}_{t}$. To this end, traditional MCMC-based sampling methods and particle-based approaches\cite{Chen2019projected,Detommaso2018stein,Garbuno2020interacting,Liu2016stein,yan2021stein} can be applied. In this paper, we have chosen to exclusively focus on the Unscented Kalman Inversion (UKI) method \cite{huang2022iterated} in order to conceptually validate that our proposed adaptive operator learning framework.  The details of the UKI algorithm will be presented in the following section. It is important to emphasize that UKI’s strategy relies on Gaussian approximations. While UKI may not provide highly accurate posterior approximations in many cases, such as non-Gaussian posteriors, we have still chosen to use UKI for sampling due to its computational efficiency. As a gradient-free, particle-based method, UKI requires only $2N_m+1$ full-order model evaluations per iteration and typically converges within $\mathcal{O}(10)$ iterations, making it computationally less expensive compared to MCMC-type methods. The numerical examples will demonstrate that even with the lower computational complexity of UKI, our newly designed framework still achieves significant improvements in computational efficiency. Additionally, our algorithm is easily integrated with MCMC or other particle-based methods\cite{Chada2021,carrillo2022consensus,ernst2015analysis,huang2022efficient,Garbuno2020interacting,iglesias2013ensemble,wang2023adaptive,Weissmann_2022,yan2019adaptive}, and when used with other samplers that require a larger number of samples, our framework yields even greater computational efficiency gains.

\subsection{Unscented Kalman Inversion}

In this section, we give a brief review of the UKI algorithm discussed in \cite{huang2022iterated}. 
The UKI is derived within the Bayesian framework and is considered to approximate the posterior distribution using Gaussian approximations on the random variable $m|y$ via its ensemble properties.  We consider  the following stochastic dynamical system:
 \begin{equation}
     \label{UKI}
     \begin{split}
         &\text{Evolution:}\quad\quad m_{n+1}=r_{0}+\alpha(m_{n}-r_{0})+ \omega_{n+1},\quad\quad \omega_{n+1}\sim\mathcal{N}(0,\Sigma_{\omega}),\\ 
         &\text{Observation:}\quad \quad y_{n+1}=\mathcal{G}(m_{n+1})+\eta_{n+1},\quad\quad \quad\quad\quad\eta_{n+1}\sim\mathcal{N}(0,\Sigma_{\eta}).
     \end{split}
 \end{equation}
 where $m_{n+1}$ is the unknown discrete parameter vector, and $y_{n+1}$ is the observation vector, the artificial evolution error $\omega_{n+1}$ and observation error $\eta_{n+1}$ are mutually independent, zero-mean Gaussian sequences with covariances $\Sigma_{\omega}$ and $\Sigma_{\eta}$, respectively. Here $\alpha\in(0,1]$ is the regularization parameter, $r_{0}$ is the initial arbitrary vector. 
 
Let $Y_{n} := \{y_{1}, y_{2},\cdots, y_{n}\}$ denote the observation set at time $n$. In order to approximate the conditional distribution $\nu_{n}$ of $m_{n}|Y_{n}$, the iterative algorithm starts from the prior $\nu_{0}$ and updates $\nu_{n}$ through the prediction and analysis steps: $\nu_{n}\rightarrow \widetilde{\nu}_{n+1}$, and then $\widetilde{\nu}_{n+1}\rightarrow \nu_{n+1}$, where $\widetilde{\nu}_{n+1}$ is the distribution of $m_{n+1}|Y_{n}$.  In the prediction step, we assume that $\nu_{n}= \mathcal{N}(r_{n}, C_{n})$, then under Eq. \eqref{UKI}, $\widetilde{\nu}_{n+1}$ is also Gaussian with mean and covariance:
 \begin{equation}
     \label{prediction}
     \begin{split}
         &\widehat{r}_{n+1}=\mathbb{E}[m_{n+1}|Y_{n}]=\alpha r_{n}+(1-\alpha)r_{0}, \\
         &\widehat{C}_{n+1}={\mathrm{Cov}}[m_{n+1}|Y_{n}]=\alpha^{2}C_{n}+\Sigma_{\omega}.
     \end{split}
 \end{equation}
 In the analysis step, we assume that joint distribution of $\{m_{n+1}, y_{n+1}\}|Y_{n}$ can be approximated by a Gaussian distribution 
 \begin{equation}
     \label{condition}
     \mathcal{N}\left(
         \left[
         \begin{array}{c}
             \widehat{r}_{n+1}\\ 
             \widehat{y}_{n+1}
         \end{array}
         \right],
         \left[
             \begin{array}{cc}
                 \widehat{C}_{n+1}&\widehat{C}^{my}_{n+1}\\ 
                 {\widehat{C}_{n+1}^{my\; T}}&\widehat{C}_{n+1}^{yy}
             \end{array}
         \right]
     \right),
 \end{equation}
 where 
 \begin{equation}
     \label{analysis}
 \begin{array}{l}{{\widehat{y}_{n+1}=\mathbb{E}[{\mathcal G}(m_{n+1})|Y_{n}],}}\\ {{\widehat{C}_{n+1}^{my}=\operatorname{Cov}[m_{n+1},{\mathcal G}(m_{n+1})|Y_{n}],}}\\ {{\widehat{C}_{n+1}^{yy}=\operatorname{Cov}[{\mathcal G}(m_{n+1})|Y_{n}]+\Sigma_{\eta}.}}
 \end{array}
 \end{equation}
 Conditioning the Gaussian in Eq.\eqref{condition} to find $m_{n+1}|\{Y_{n}, y_{n+1}\}=m_{n+1}|Y_{n+1}$ gives the following expressions for the mean $r_{n+1}$ and covariance $C_{n+1}$ of the approximation to $\nu_{n+1}$:
  \begin{equation}
     \label{update}
     \begin{array}{l l}
         {{r_{n+1}=\widehat{r}_{n+1}+\widehat{C}_{n+1}^{my}(\widehat{C}_{n+1}^{yy})^{-1}(y_{n+1}-\widehat{y}_{n+1}),}}\\ {{C_{n+1}=\widehat{C}_{n+1}-\widehat{C}_{n+1}^{my}(\widehat{C}_{n+1}^{yy})^{-1}\widehat{C}_{n+1}^{my}}}.
     \end{array}
  \end{equation}
  By assuming all observations are identical to $y$ (i.e., $Y_{n} = y$), Eqs.\eqref{prediction}-\eqref{update}  define a conceptual algorithm for using Gaussian approximation to solve BIPs.  To evaluate integrals appearing in Eq. \eqref{prediction}, UKI employs the unscented transform described below.
  
\begin{theorem}[Modified Unscented Transform \cite{wan2000unscented}]
     Let Gaussian random variable $m\sim\mathcal{N}(r, C)\in \mathbb{R}^{N_{m}}$, $2N_{m} + 1$ symmetric $\sigma$-points are chosen deterministically:
     \begin{equation*}
         \label{sigma_points}
         m^{0}=r,\quad m^{j}=m+c_{j}[\sqrt{C}]_{j}, \quad m^{j+N_{m}}=m-c_{j}[\sqrt{C}]_{j},\quad(1\leq j\leq N_{m}),
     \end{equation*}
 where $[\sqrt{C}]_{j}$ is the $jth$ column of the Cholesky factor of $C$. The quadrature rule approximates the mean and covariance of the transformed variable $\mathcal{G}_{i}(m)$ as follows 
 \begin{equation*}
     \label{unscented_transform}
     \begin{split}
     &\mathbb{E}[\mathcal{G}_{i}(\theta)]\approx \mathcal{G}_{i}(m_{0})=\mathcal{G}_{i}(r),\\
     &Cov[{\mathcal{G}_{1}}(m),\mathcal{G}_{2}(m)]\approx\sum_{i=1}^{2N_{m}}W_{j}^{c}\left(\mathcal{G}_{1}(m^{j})-\mathbb{E}{\mathcal{G}_{1}}(m)\right)\left(\mathcal{G}_{2}(m^{j})-\mathbb{E}{\mathcal{G}_{2}}(m)\right)^{T}.
     \end{split}
 \end{equation*}
 Here these constant weights are 
 \begin{equation*}
     \label{weights}
     \begin{split}
         &c_{1} =c_{2}\cdots =c_{N_{m}} = \sqrt{N_{m}+\lambda},\quad W_{1}^{c} = W_{2}^{c} = \cdots = W_{2N_{m}}^{c} = \frac{1}{2(N_{r}+\lambda)},\\ 
         &\lambda = a^{2}(N_{m} +\kappa) -N_{m},\quad \kappa = 0,\quad a = \min\{\sqrt{\frac{4}{N_{m}+\kappa}}, 1\}.
     \end{split}
 \end{equation*}
  \end{theorem}

We obtain the following UKI algorithm in Algorithm \ref{Algorithm_uki} by applying the aforementioned quadrature rules.  UKI is a derivative-free algorithm that applies a Gaussian approximation theorem iteratively to transport a set of particles to approximate given distributions. As a result, it only needs $2N_{m}+1$ model evaluations per iteration, making it simple to implement and inexpensive to compute. However, for highly-nonlinear problems, UKI may encounter the intrinsic difficulty of using Gaussian distributions to approximate the posterior and different samplers can be applied.

  \begin{algorithm}[th]
     \caption{Unscented Kalman Inversion (UKI)} 
     \label{Algorithm_uki}
     \begin{algorithmic}[1]
         \Procedure{RunUki}{$r_{0},C_{0}, y, \mathcal{G}, \Sigma_{\omega}, \Sigma_{\eta}, T$}
         \For{$n = 0,\ldots, T-1$}
         \State Prediction step:
         \begin{equation*}
             \widehat{r}_{n+1}=\alpha r_{n}+(1-\alpha)r_{0},\widehat{C}_{n+1}=\alpha^{2}C_{n}+\Sigma_{\omega}.
         \end{equation*}
         \State Generate $\sigma$-points 
         \begin{equation*}
             \begin{split}
                 &\widehat{m}^{0}_{n+1}=\widehat{r}_{n+1},\\ 
                 &\widehat{m}^{j}_{n+1}=\widehat{r}_{n+1}+c_{j}[\sqrt{\widehat{C}_{n+1}}]_{j},\quad (1\leq j\leq N_{m}), \\
                 &\widehat{m}^{j+N_{m}}_{n+1}=\widehat{r}_{n+1}-c_{j}[\sqrt{\widehat{C}_{n+1}}]_{j},\quad(1\leq j\leq N_{m}).
             \end{split}
         \end{equation*}
         \State Analysis step:
         \begin{equation*}
             \begin{split}
                 &\widehat{s}_{n+1}^{j}=\mathcal{G}(\widehat{m}_{n+1}^{j}),\qquad\widehat{y}_{n+1}=\widehat{y}_{n+1}^{0},\\ 
                 &\widehat C_{n+1}^{my}=\sum_{j=1}^{2N_{m}}W_{j}^{c}(\widehat m_{n+1}^{j}-\widehat r_{n+1})(\widehat y_{n+1}^{j}-\widehat y_{n+1})^{T},\\ 
                 &\widehat{C}_{n+1}^{yy}=\sum_{j=1}^{2N_{m}}W_{j}^{c}(\widehat{y}_{n+1}^{j}-\widehat{y}_{n+1})(\widehat{y}_{n+1}^{j}-\widehat{y}_{n+1})^{T}+\Sigma_{\eta},\\ 
                 &r_{n+1}=\widehat{r}_{n+1}+\widehat{C}_{n+1}^{my}(\widehat{C}_{n+1}^{yy})^{-1}(y-\widehat{y}_{n+1}),\\ 
                 &C_{n+1}=\widehat{C}_{n+1}-\widehat{C}_{n+1}^{my}(\widehat{C}_{n+1}^{yy})^{-1}\widehat{C}_{n+1}^{my}.
             \end{split}
         \end{equation*}
         \EndFor
         \EndProcedure
     \end{algorithmic}
  \end{algorithm}

\subsection{Algorithm Summary}

\begin{algorithm}[t]
    \caption{UKI with  DeepOnet approximation}
    \label{algorithm_summary}
    \begin{algorithmic}[1]
        \Procedure{RunLUki}{$\mathcal{D}, \mathcal{F}, \mathcal{F}_{\theta},y, r_{0}, C_{0}, \Sigma_{\omega}, \Sigma_{\eta}, T, I_{\max},\epsilon$}
        \State Set $\widehat{\mathcal{G}}_{0}=\mathcal{O}\circ  \mathcal{F}_{\theta}$, and compute $e_0$ in Eq.\eqref{data_fitting_error} with respect to $r_0$ and $y$.
        \For{$t = 0,\ldots,I_{max} - 1$}       
        \State $\{(r_{t+1}^{n}, C_{t+1}^{n})\}_{n=1}^{T}\leftarrow$ \Call{RunUki}{$r_{t}, C_{t}, y, \widehat{\mathcal{G}}_{t}, \Sigma_{\omega}, \Sigma_{\eta}, T$} 
        \State $J \leftarrow \arg\min_{n} \{e^n_{t+1}=\frac{1}{2}\|y - \mathcal{G}(r^n_{t+1})\|_{\Sigma_{\eta}}^{2}\}$
        \State $(r_{t+1}, C_{t+1}, e_{t+1})\leftarrow (r_{t+1}^{J}, C_{t+1}^{J}, e_{t+1}^{J})$
        \If{$|(e_{t}  -e_{t+1})/e_{t+1}|>\epsilon$}
        \State ($\mathcal{F}_{\theta}, \mathcal{D})\leftarrow$\Call{RefineApprox}{$r_{t+1}, C_{t+1}, \mathcal{D},\mathcal{F}, \mathcal{F}_{\theta}$}
        \State Set $\widehat{\mathcal{G}}_{t+1} = \mathcal{O}\circ \mathcal{F}_{\theta}$
        \Else
        \State Break;
        \EndIf
        \EndFor
        \State \textbf{Return} the final inversion result $r_{t+1}, C_{t+1}$ 
        \EndProcedure
        \State
        \Procedure{RefineApprox}{$r, C, \mathcal{D},\mathcal{F}, \mathcal{F}_{\theta}$}
        \State Sample $\Gamma = \{m_{j}\}_{j=1}^{K}\sim \mathcal{N}(r, C)$.
        \State Choose new $Q$ samples $\widetilde{m}_j$ using Eq.\eqref{upeq}, and replace $\mathcal{D} \leftarrow \mathcal{D} \cup\{\widetilde{m}_j, \mathcal{F}(\widetilde{m}_j)\}$
        \State Update the operator network $\mathcal{F}_{\theta}$ using the training set $\mathcal{D}$
        \State \textbf{return} the surrogate $  \mathcal{F}_{\theta}$ and the training set $\mathcal{D}$
        \EndProcedure
    \end{algorithmic}
 \end{algorithm} 
 
The overview of our UKI-based adaptive operator learning strategy is provided by Algorithm \ref{algorithm_summary}. To summarize, we begin with an offline pre-trained DeepOnet and refine the surrogate until the stop criteria is satisfied using local training data from the current approximate posterior distribution obtained by UKI. In other words, the UKI provides the DeepOnet with valuable candidate training points to refine,  while the DeepOnet efficiently returns approximate data misfit information for the UKI to explore the parameter space further.  Together, they establish a mutual learning system that enables them to grow and learn from one another over time. In particular, in order to approximate the induced posterior distribution $\widehat{\nu}_{t}$ and produce a series of intermediate Gaussian approximations $\mathcal{N}(r_{t+1}^n, C_{t+1}^n),\,\, n= 1,\cdots, T$, we first perform UKI with $T$ steps during the refinement process. The least squared error $e_{t+1}^{n}:=\Phi(r_{t+1}^{n};y)$ with $r_{t+1}^{n}$ will be calculated using Eq.\eqref{least-sqaured-error} to form $U_{t+1} = \{(r_{t+1}^{n}, C_{t+1}^{n}, e_{t+1}^{n})\}_{n=1}^{T}$. It should be noted that not every step in this process is valid because the model error will eventually blow up during the iteration process.  As a result, if the refinement process is terminated, we will select the last valid inversion result from this set as:
\begin{equation*}
    \label{data-fitting-error}
    (r_{t+1}, C_{t+1}, e_{t+1}) := (r_{t+1}^{J}, C_{t+1}^{J}, e_{t+1}^{J}) = \arg\min_{(r, C, e)\in U_{t+1}}\frac{1}{2}\|y - \mathcal{G}(r)\|_{\Sigma_{\eta}}^{2}.
\end{equation*} 
If the surrogate requires more refinement, the pair $(r_{t+1}, C_{t+1})$ is used to select new training points and serve as the initial vector for the subsequent UKI iteration. 

We review the computational efficiency of our method.  Since the pre-trained operator network can be applied as surrogates for a class of various BIPs with models that are governed by the same PDEs but have various types of observations and noise modes. Thus, for a given inversion task, the main computational cost centers on the online forward evaluations and the online fine-tuning. On the other hand, the online retraining only takes a few seconds each time, so it can be ignored in comparison to the forward evaluation time. In these situations, the forward evaluations account for the majority of the computational cost. For our algorithm, the maximum number of online high-fidelity model evaluations is $N_{DeepOnet} = (Q+T)I_{max}$, where $Q$ is the number of adaptive samples for refinement, $T$ is the maximum number of iterations for UKI using our approach, and $I_{max}$ is the number of adaptive refinement. While $N_{\tiny{FEM}} = (2N_{m}+1)T_{\tiny{FEM}}$ represents the total evaluations for UKI using the FEM solver, $N_{m}$ represents the discrete dimension of the parameter field, and $T_{FEM}$ represents the maximum iteration number for UKI. Consequently, the asymptotic speeds increase can be computed as 
 \begin{equation*}
    SpeedUp = \frac{(2N_{m}+1)T_{\tiny{FEM}}}{(Q+T)I_{max}}.
 \end{equation*}
Note that the efficiency of our method basically depends on $Q, I_{max}$ since $T$ is typically small. First of all, the number of adaptive samples $Q$ will be sufficiently small (e.g. $\mathcal{O}(10)$) compared to the discrete dimension $N_{m}$ (e.g. $\mathcal{O}(10^{2})-\mathcal{O}(10^{3})$), resulting in a significant reduction in computational cost as it determines the number of forward evaluations.  Second, the total number of adaptive retraining $I_{max}$ is determined by the inverse tasks, which are further  subdivided into \textit{in-distribution} data (IDD)  and \textit{out-of-distribution} (OOD) cases.  The IDD  typically refers to ground truth that is located in the high density region of the prior distribution. Alternatively, the OOD refers to the ground truth that is located far from the high density area of the prior. 
The original pre-trained surrogate  can be accurate in nearly all of the high probability areas for the IDD case. As a result, our framework will converge quicker. In the case of OOD, our framework requires a significantly higher number of retraining cycles in order to reach the high density area of the posterior distribution. However, $I_{max}$ for both inversion tasks can be small. Consequently, our method can simultaneously balance accuracy and efficiency and has the potential to be applied to dynamical inversion tasks. In other words, once the initial surrogate is trained, we can use our adaptive framework to modify the estimate at a much lower computational cost.

 \subsection{Convergence analysis under the linear case} 
 It is important to note that the UKI's ensemble properties are used to approximate the posterior distribution using Gaussian approximations. Specifically, the sequence in Eq.\eqref{update} obtained by the full-order model $\mathcal{G}$ will converge to the equilibrium points of the following equations under certain mild conditions in the linear case \cite{huang2022iterated}:
\begin{subequations}
    \label{converge}
    \begin{align}
        &C_{\infty}^{-1}=\mathcal{G}^{T}\Sigma_{\eta}^{-1}\mathcal{G}+(\alpha^{2}C_{\infty}+\Sigma_{\omega})^{-1},\label{converge_a}\\ 
        &C_{\infty}^{-1}r_{\infty}=\mathcal{G}^{T}\Sigma_{\eta}^{-1}y+(\alpha^{2}C_{\infty}+\Sigma_{\omega})^{-1}\alpha r_{\infty}\label{converge_b}.
    \end{align}
\end{subequations}
We can actually demonstrate that in the linear case, the mean vector and covariance matrix obtained by our approach will be close to those obtained by the true forward if the surrogate $\widehat{\mathcal{G}}$ is close to the true forward $\mathcal{G}$.  
Consider the following: $\Range(\widehat{\mathcal{G}}) = \mathbb{R}^{N_{m}}$, and $\widehat{\mathcal{G}}$ is linear. Using $\widehat{\mathcal{G}}$ as a surrogate, the corresponding sequence of $r_{n}, C_{n}$ in Eq.\eqref{update} then converges to the following equations
    \begin{subequations}
        \label{surrogate_converge}
        \begin{align}
            &\widehat{C}_{\infty}^{-1}=\widehat{\mathcal{G}}^{T}\Sigma_{\eta}^{-1}\widehat{\mathcal{G}}+(\alpha^{2}\widehat{C}_{\infty}+\Sigma_{\omega})^{-1},\label{surrogate_converge_a}\\ 
            &\widehat{C}_{\infty}^{-1}\widehat{r}_{\infty}=\widehat{\mathcal{G}}^{T}\Sigma_{\eta}^{-1}y+(\alpha^{2}\widehat{C}_{\infty}+\Sigma_{\omega})^{-1}\alpha\widehat{r}_{\infty}\label{surrogate_converge_b}
        \end{align}
    \end{subequations}

In the following, we will demonstrate that if the surrogate $\widehat{\mathcal{G}}$ is near the true forward model $\mathcal{G}$, then the $\widehat{r}_{\infty}, \widehat{C}_{\infty}$ ought to be near the true ones as well. We shall need the following assumptions.

\begin{assumption}
        \label{assumption1}
    Suppose for any $\epsilon$, the linear neural operator $\widehat{\mathcal{G}}:\mathbb{R}^{N_{m}}\rightarrow \mathbb{R}^{N_{y}}$ can be trained sufficiently to satisfy 
    \begin{equation}
        \|\widehat{\mathcal{G}} - \mathcal{G}\|_{2} < \epsilon.
    \end{equation}
    \end{assumption}
\begin{assumption}
    \label{assumpotion2}
    Suppose the forward map $\mathcal{G}$ is bounded, that is 
    \begin{equation}
        \|\mathcal{G}\|_{2} < H,
    \end{equation}
    where $H$ is a constant.
\end{assumption}
\begin{assumption}
    \label{Assumption3}
    Suppose the matrix $\mathcal{G}^{T}\Sigma_{\eta}^{-1}\mathcal{G}\succ 0$\footnote{We use the notation $\succ$ here to demonstrate that the matrix is symmetric and positive definite} and can be bounded from below as 
    \begin{equation}
        \|\mathcal{G}^{T}\Sigma_{\eta}^{-1}\mathcal{G}\|_{2}>C_{1},
    \end{equation}
    where $C_{1}$ is a positive constant.
\end{assumption}

We can obtain the following lemma.
\begin{lemma}
    \label{lemma1}
    Under Assumptions \ref{assumption1}-\ref{Assumption3}, $\widehat{\mathcal{G}}^{T}\Sigma_{\eta}^{-1}\widehat{\mathcal{G}}$ is also bounded from below as 
    \begin{equation}
        \label{lemma_C_2}
        \|\widehat{\mathcal{G}}^{T}\Sigma_{\eta}^{-1}\widehat{\mathcal{G}}\|_{2}>C_{2}.
    \end{equation}
    where $C_{2}$ is constant dependent on $C_{1}$.
\end{lemma}
\begin{proof}
    We first consider 
    \begin{equation}
            \begin{split}
            \mathcal{G}^{T}\Sigma_{\eta}^{-1}\mathcal{G} - \widehat{\mathcal{G}}^{T}\Sigma_{\eta}^{-1} \widehat{\mathcal{G}}
            & = \mathcal{G}^{T}\Sigma_{\eta}^{-1}\mathcal{G} - \mathcal{G_{\theta}}^{T}\Sigma_{\eta}^{-1}\mathcal{G} +\mathcal{G_{\theta}}^{T}\Sigma_{\eta}^{-1}\mathcal{G} - \widehat{\mathcal{G}}^{T}\Sigma_{\eta}^{-1} \widehat{\mathcal{G}}\\ 
            & = \left(\mathcal{G} - \widehat{\mathcal{G}}\right)^{T}\Sigma_{\eta}^{-1}\mathcal{G} +\mathcal{G_{\theta}}^{T}\Sigma_{\eta}^{-1}\left(\mathcal{G} -\widehat{\mathcal{G}}\right).
            \end{split}
    \end{equation}
    Combining Assumptions \ref{assumption1} and \ref{assumpotion2}, this leads to 
    \begin{equation}
        \label{divide}
        \begin{split}
        \|\mathcal{G}^{T}\Sigma_{\eta}^{-1}\mathcal{G} - \widehat{\mathcal{G}}^{T}\Sigma_{\eta}^{-1} \widehat{\mathcal{G}}\|_{2} &\leq \|\mathcal{G} - \widehat{\mathcal{G}}\|_{2}\|\Sigma_{\eta}^{-1}\|_{2}(\|\mathcal{G}\|_{2} + \|\widehat{\mathcal{G}}\|_{2})\\
        &\leq 2\epsilon H\|\Sigma_{\eta}^{-1}\|_{2}.
        \end{split}
    \end{equation}
    Then, we have 
    \begin{equation}
        \begin{split}
        \|\widehat{\mathcal{G}}^{T}\Sigma_{\eta}^{-1}\widehat{\mathcal{G}}\|_{2} &= \|\widehat{\mathcal{G}}^{T}\Sigma_{\eta}^{-1}\widehat{\mathcal{G}} - \mathcal{G}^{T}\Sigma_{\eta}^{-1}\mathcal{G} + \mathcal{G}^{T}\Sigma_{\eta}^{-1}\mathcal{G}\|_{2}\\ 
        & \geq \|\mathcal{G}^{T}\Sigma_{\eta}^{-1}\mathcal{G}\|_{2} - \|\widehat{\mathcal{G}}^{T}\Sigma_{\eta}^{-1}\widehat{\mathcal{G}} - \mathcal{G}^{T}\Sigma_{\eta}^{-1}\mathcal{G}\|_{2}\\ 
        & \geq C_{1} - 2\epsilon H\|\Sigma_{\eta}^{-1}\|_{2}\geq C_{2}.
        \end{split}
    \end{equation}
\end{proof}

Note that these assumptions are reasonable and can be found in many references \cite{yan2017convergence,cao2023residual}. We will then supply the main theorem based on these  assumptions.
\begin{theorem}
    \label{main_theorem}
    Under Assumptions \ref{assumption1}-\ref{Assumption3}, suppose $Range(\mathcal{G}^{T}) = Range(\widehat{\mathcal{G}}^{T}) = \mathbb{R}^{N_{m}}$ and $\Sigma_{\omega}\succ 0, \Sigma_{\eta}\succ 0$. Then the sequence $\widehat{r}_{\infty}, \widehat{C}_{\infty}^{-1}$ in Eq.\eqref{surrogate_converge} obtained by using the surrogate model will converge to the $r_{\infty}, C_{\infty}^{-1}$ in Eq.\eqref{converge} and we have the following error estimate
    \begin{equation}
        \label{error_estimate}
        \begin{split}
        &\|C_{\infty}^{-1} - \widehat{C}_{\infty}^{-1}\|_{2} \leq \frac{2\epsilon HH_{\eta}}{1-\beta},\\ 
        &\|r_{\infty} - \widehat{r}_{\infty}\|_{2} \leq \frac{K_{1}H_{\eta}H_{y}}{C_{1}}\left(1 + \frac{2(1 + \alpha\beta) K_{2}H_{\eta} H^{2}}{(1-\beta)C_{2}}\right)\epsilon,
        \end{split}
    \end{equation}
    where $\beta, C_{1}, C_{2}, K_{1}, K_{2}, H_{\eta}, H_{y}, H$ are positive bounded constants.
\end{theorem}
\begin{proof} The proof can be found in Appendix A.
\begin{remark}
 In order to meet the requirements of Theorem \ref{main_theorem}, it is possible to make the neural operator $\mathcal{G}_{\theta}$ linear by dropping the nonlinear activation functions in the branch net and keeping the activation functions in the trunk net.
\end{remark}

\section{Numerical experiments}
\label{numerical_experiments}
In this section, we provide several numerical examples to demonstrate the effectiveness and precision of the adaptive operator learning approach for solving inverse problems. To better present the results, we will compare DeepOnet based UKI inversion results (referred to as DeepOnet-UKI) with those of conventional FEM solvers (referred to as FEM-UKI). In addition, depending on whether adaptive refinement is used, the DeepOnet-UKI method has two variants: DeepOnet-UKI-Direct and DeepOnet-UKI-Adaptive. In particular, for DeepOnet-UKI-Direct, we will leave the surrogate model unchanged during the UKI iteration process.

In all of our numerical tests, the branch and trunk nets for DeepOnet are fully connected neural networks with five hidden layers and one hundred neurons in each layer, with the \textit{tanh} function as the activation function.  DeepOnet is trained offline with $1\times 10^5$ iterations and $N_{prior} = 1000$ prior samples from the Gaussian random field.  Unless otherwise specified, we set the maximum retraining number to $I_{max} = 10$ and the tolerance to $\epsilon = 0.01$.  For all examples investigated in this paper, the synthetic noisy data are generated as:
\begin{equation}
    y_{obs} = y_{ref} + \mbox{max}\{|y_{ref}|\}\delta \xi,
 \end{equation}
where $y_{ref} = \mathcal{G}(m_{ref})$ are the exact data,  $\delta$ dictates the relative noise level and $\xi$ is a Gaussian random variable with zero mean and unit standard deviation. In UKI, the regularization parameters are $\alpha=0.5$ for noise levels 0.05 and 0.1 and $\alpha=1$ for noise levels 0.01. The starting vector for the UKI is chosen at random from $\mathcal{N}(0,I)$. The selection of other hyperparameters is based on \cite{huang2022iterated}.  For numerical examples, we set $\Omega =[0,1]^2$. The maximum number of UKI iterations per cycle is 20 for FEM-UKI and DeepOnet-UKI-Direct, and 10 for DeepOnet-UKI-Adaptive. Following that, we will use the greedy algorithm to select $Q= 50$ adaptive samples for noise level 0.01 and $Q = 20$ adaptive samples for noise levels 0.05 and 0.1 from $K= 2000$ samples.  

To measure the accuracy of the numerical approximation with respect to the exact solution, we  use the following relative inversion error $e_{\mathcal{I}}$ defined as
 \begin{equation}
    \label{measure_function_relative}
    e_{\mathcal{I}} = \frac{\|\widehat{m} - m_{ref}\|_{2}}{\|m_{ref}\|_{2}},
 \end{equation}
 where $\widehat{m}$ and $m_{ref}$ are the numerical and exact solutions, respectively. Additionally, we will create $M = 100$ samples during the UKI iteration process to calculate the local model error in Eq.\eqref{localerr} as
 \begin{equation}
    \label{dlocal_model_error}
    e_{\mathcal{M}} = \frac{1}{M}\sum_{i=1}^{M}\left\|\widehat{\mathcal{G}}_{t}(m_{i}) - \mathcal{G}(m_{i})\right\|_{2},
 \end{equation}
to demonstrate that our adaptive framework can actually reduce the local model error. Moreover, we also calculate the data fitting error via Eq.\eqref{least-sqaured-error} using the true model during the UKI iteration process. 
   
 
 \subsection{Example 1: Darcy flow}
 \label{Darcy_flow}
In the first example, we consider the following Darcy flow problem:
 \begin{equation}
    \label{darcy_equation}
    \begin{split}
    -\nabla\cdot (\exp(\mb{m}(\mb{x}))\nabla{u(\mb{x})}) &= f(\mb{x}) ,\quad \mb{x}\in \Omega, \\ 
    u(\mb{x}) &= 0,\quad\quad \mb{x}\in\partial \Omega. 
    \end{split}
 \end{equation}
 Here, the source function $f(\mb{x})$ is defined as 
 \begin{equation}
    f(x_{1}, x_{2})=\begin{cases}
        1000\quad 0\leq x_{2}\leq \frac{4}{6},\\ 
        2000\quad \frac{4}{6}< x_{2}\leq \frac{5}{6},\\  
        3000 \quad \frac{5}{6}<x_{2}\leq 1.
    \end{cases}
 \end{equation}
 
 \begin{figure}[t]
\begin{center} 
\begin{overpic}[width=0.3\textwidth]{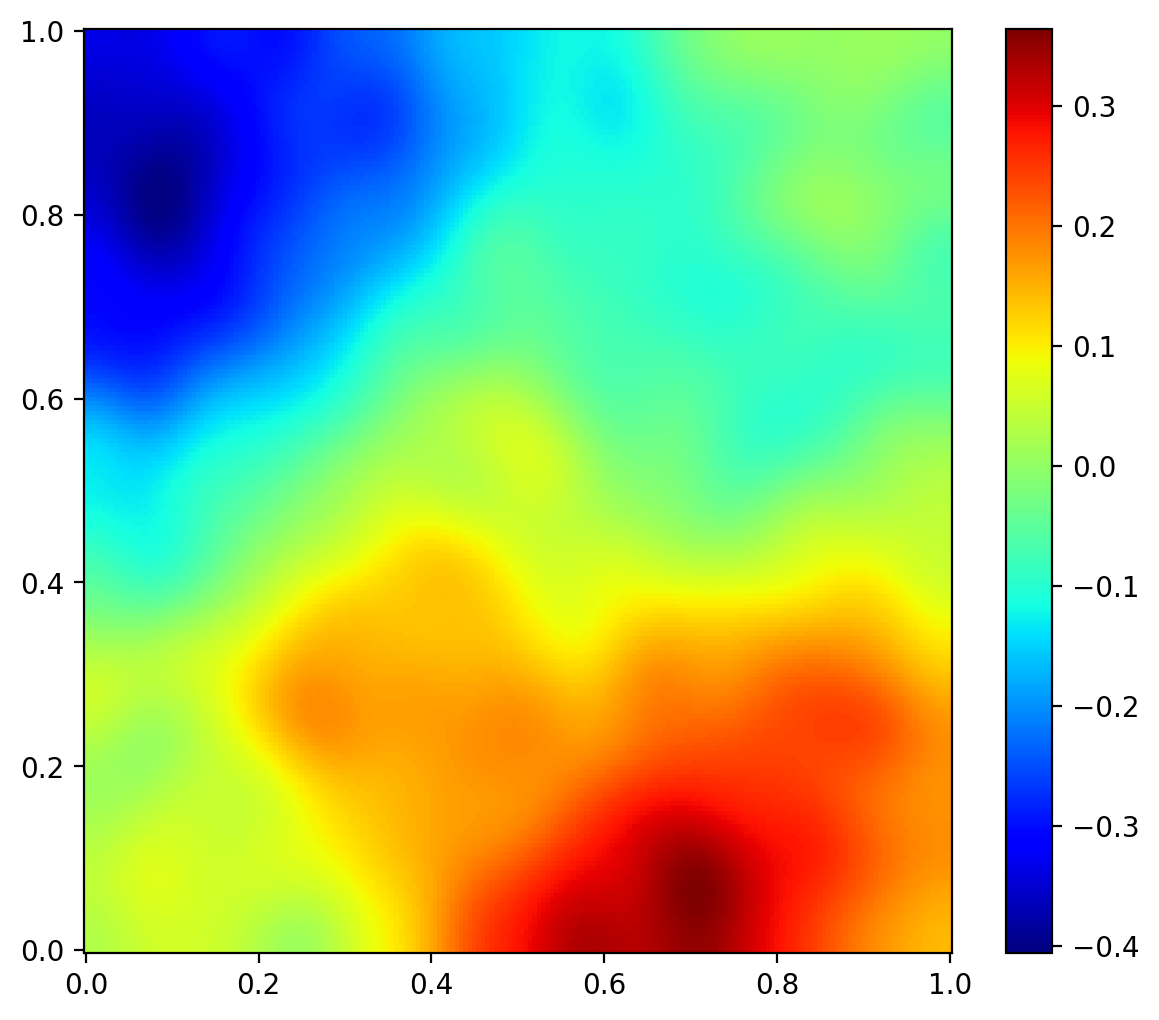}
\end{overpic}
\begin{overpic}[width=0.308\textwidth]{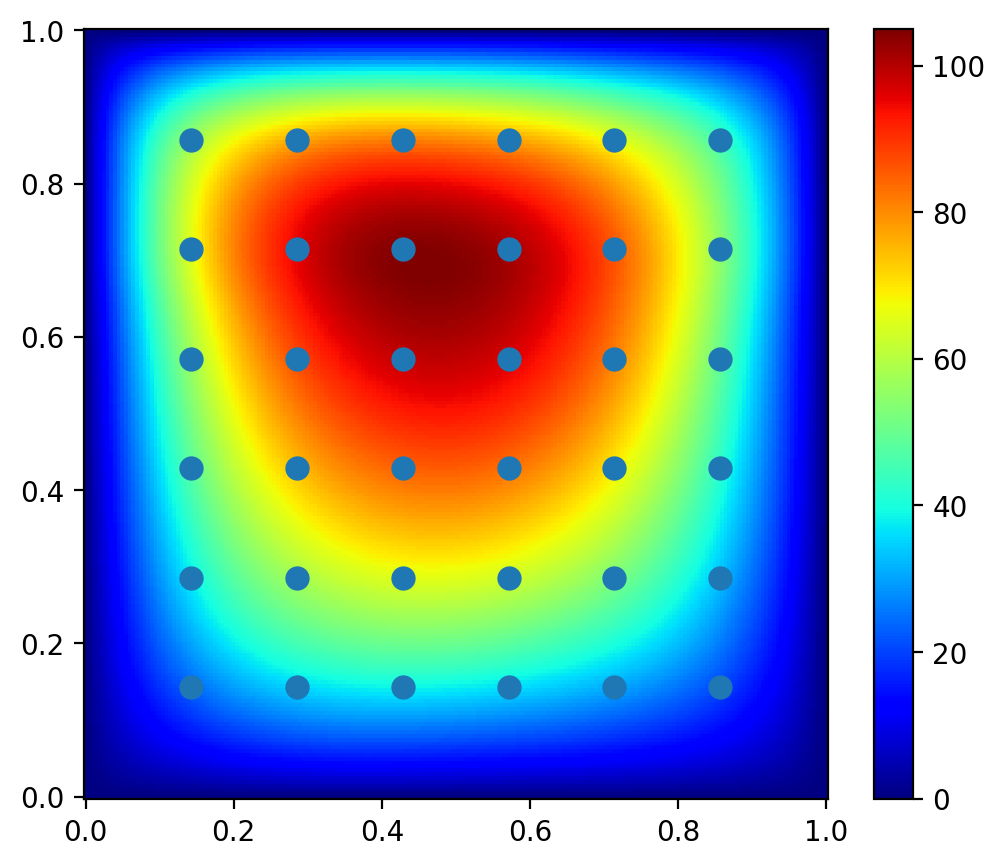}
\end{overpic}
\end{center}
\begin{center} 
    \begin{overpic}[width=0.295\textwidth]{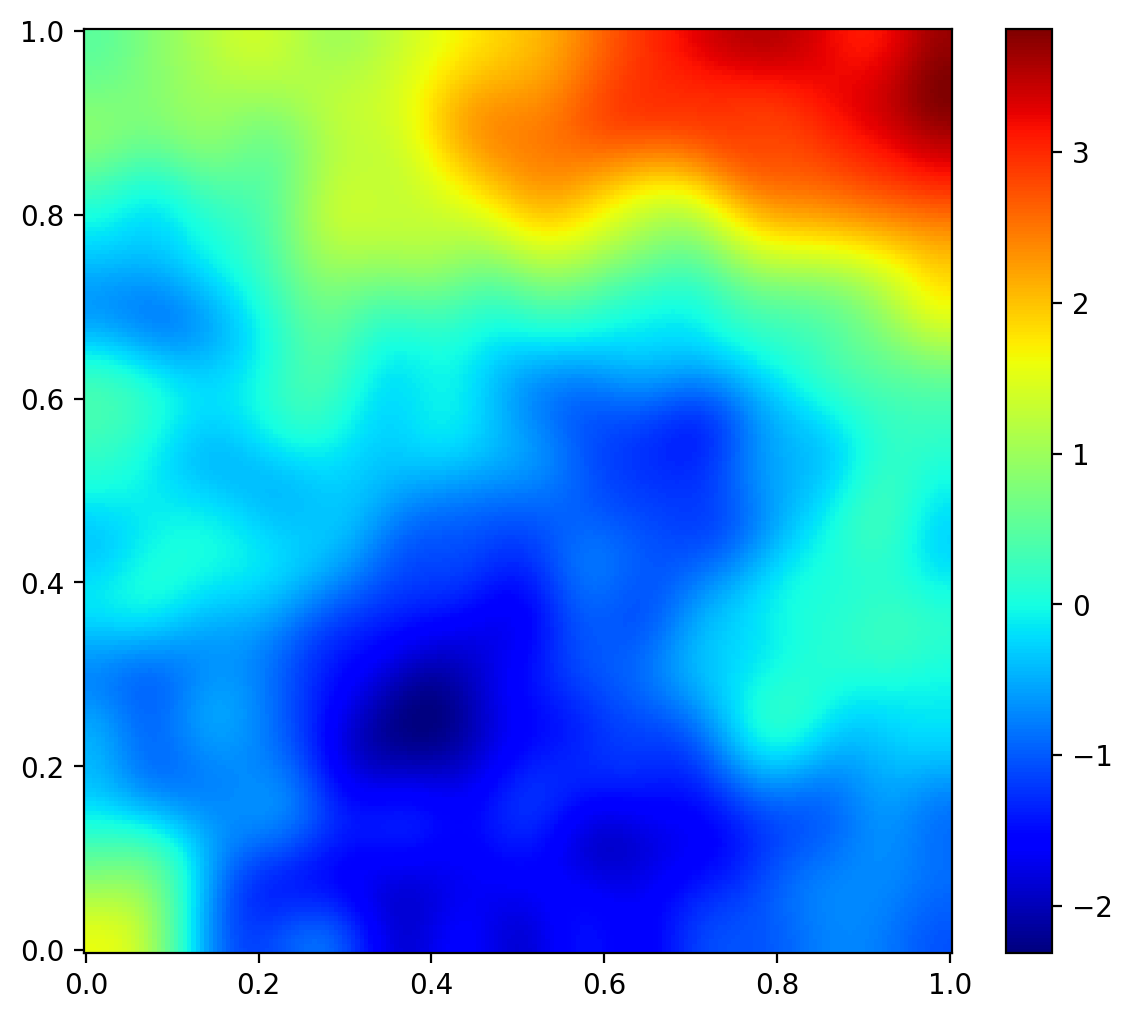}
    \end{overpic}
    \begin{overpic}[width=0.31\textwidth]{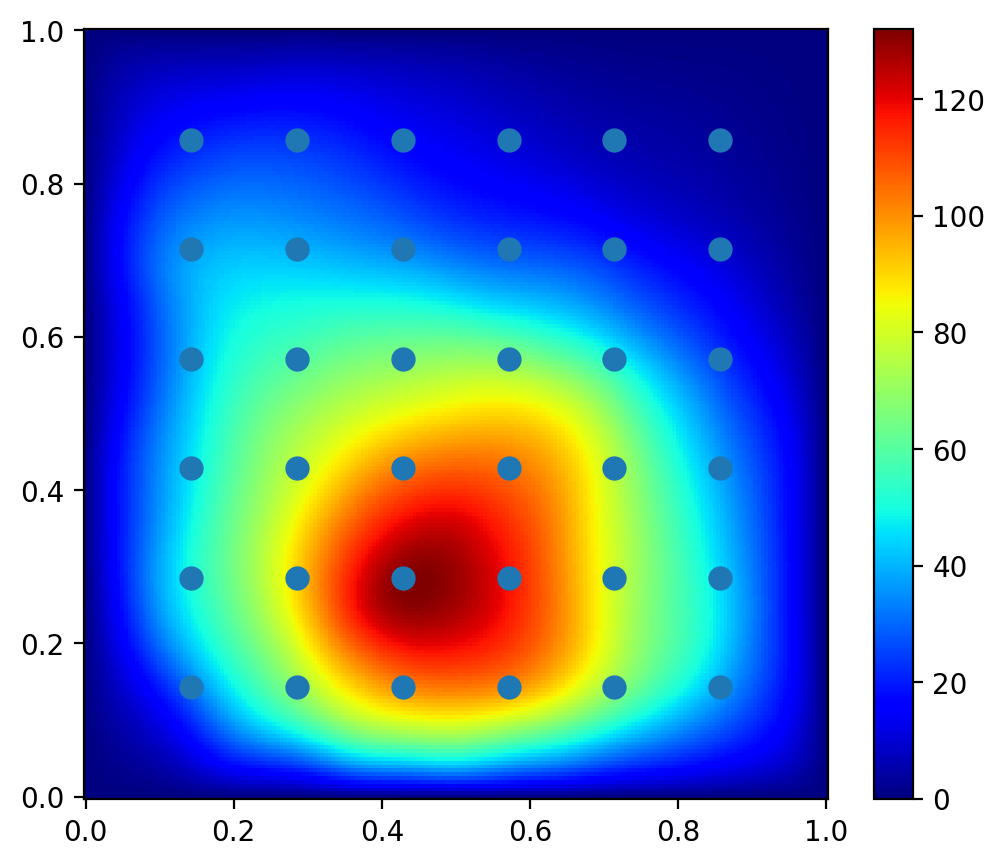}
    \end{overpic}
    \end{center}
\caption{The ground truth for \textit{in-distribution} data(IDD) and \textit{out-of-distribution} data (OOD) from above to below. Left: the true permeability fields $\mb{m}(\mb{x})$. Right: the pressure fields $u(\mb{x})$ and the corresponding $36$ equidistant observations with noise level 0.01. }
\label{observations}  
 \end{figure}
 
 The aim is to determine the  permeability $\mb{m}(\mb{x})$ from noisy measurements of the $u$-field at a finite set of locations.  To ensure the existence of the posterior distribution, we typically selected the prior distribution $\nu_{0}$ as a Gaussian measure $\mathcal{N}(r_{pr}, \mathcal{C}_{pr})$.  In particular, we focus on the covariance operator with the following form: 
 \begin{equation}
     \label{prior_cov}
     \mathcal{C}_{pr} = \sigma^{2}(-\Delta + \tau^{2})^{-d},
 \end{equation}
 where $\Delta$ denotes the Laplacian operator in $\Omega$ subject to homogeneous Neumann boundary conditions, $\tau$ denotes the inverse length scale of the random field and $d > 0$ determines its regularity. For the numerical experiments presented in this section, we take  the same values for these parameters as in\cite{huang2022iterated}: $\tau = 3, d = 2, \sigma = 1$. To sample from the prior distribution, we can use the Karhunen-Loeve (KL) expansion, which has the form 
 \begin{equation}
     \label{KL}
     \mb{m}(\mb{x}) = \sum_{k\in \mathbb{Z}^+} \textcolor{black}{\zeta_{k}}\sqrt{\lambda_{k}}\psi_{k}(\mb{x}),
 \end{equation}
where $\lambda_k$ and $\psi_k$ are the eigenvalues and eigenfunctions,  and $\textcolor{black}{\zeta_{k}} \sim \mathcal{N}(0, 1)$ are independent random variables. In practice, we truncate the sum \eqref{KL} to $n_d$ terms, based on the largest $n_d$ eigenvalues, and hence $\textcolor{black}{\zeta} \in \mathbb{R}^{n_d}$. The forward problem is solved by FEM method on a $70\times 70$ grid.  

We will create the observation data for the inverse problem using the \textit{in-distribution} data (IDD) and \textit{out-of-distribution} data (OOD), respectively, as shown in Fig.\ref{observations}. The IDD field $\mb{m}_{ref}(\mb{x})$ is calculated using Eq.\eqref{KL} with $n_d=256$ and $\textcolor{black}{\zeta_{k}}\sim \mathcal{N}(0,1)$. The OOD field $\mb{m}_{ref}(\mb{x})$ is generated for convenience  by sampling $\textcolor{black}{\zeta_{k}}\sim \mathcal{U}[-20, 20], k=1,\ldots, 256$.  To avoid the inverse crime, we will try to inverse the first $N_m = 128$ KL modes using these observation data.

 We plot the data fitting error, model error, and inversion error in Fig.\ref{loss_flow} to demonstrate the effectiveness of our framework. The performance of IDD and OOD data is different. For IDD data, if we directly apply the initial trained surrogate to run UKI, known as DeepOnet-UKI-Direct, we can see that the model error is consistently small. Even without refinements, the inversion error would be similar to that obtained by FEM-UKI, as shown in the right display of Fig.\ref{loss_flow}.  However, if we use an adaptive dataset to refine the initial model, we can still see a significant decrease in the local model error, resulting in a better estimate after running UKI for several steps.
This suggests that refinements can improve the inversion accuracy of IDD data. The situation with OOD data is not the same. Because the ground truth is far from the prior distribution, the model error will first decrease and then explode suddenly as expected if we directly apply the initial model, as shown in the middle display of Fig.\ref{loss_flow}. In such cases, DeepOnet-UKI-Direct will produce an incorrect estimate, which requires the refinement of the surrogate to improve inversion accuracy. The refinement process is typically divided into two stages: exploration and exploitation. During the exploration stage, we will run UKI with the current surrogate for $T$ steps.  Then we will select the {\it anchor point} with smallest data fitting error computed with true model.  In the exploitation stage, we generate an adaptive training dataset near this {\it anchor point} using the greedy algorithm to refine the surrogate, which leads to much smaller model error as demonstrated in the middle display of Fig.\ref{loss_flow}. Then, using this refined surrogate, we continue the UKI iteration, starting with the {\it anchor point} from the previous refinement.  This significantly improves the inversion results, as shown right display of Fig.\ref{loss_flow}.  Additionally, this figure shows that the model error will dramatically increase during the inversion process with just one refinement.  Refinements will no longer clearly affect the accuracy of the inversion after five iterations, at which point the model error is usually negligible. Note that this strategy also works for IDD data; with refinements, our method can reduce model error and achieve comparable performance to FEM-UKI. The difference is that for OOD data, convergence is slower, leading to more refinements. While this agrees with our formal analysis. It is worth noting that, in the IDD situation, the adaptive method produces better results than the FEM-UKI. One probable explanation is that the model error is lower than the noise level of observation data, resulting in a minor random fluctuation of data $y_{obs}$ that benefits the UKI method.

To summarize, DeepOnet-UKI-Adaptive can perform well for both IDD and OOD data. The black star in Fig.\ref{loss_flow} denotes the final solution for DeepOnet-UKI-Adaptive selected with the minimum data fitting error. We can observe that this value is typically the same as that acquired by FEM-UKI. We show the final estimated permeability fields generated by three different approaches in Figs.\ref{kappa_flow_in_distribution} and \ref{kappa_flow}. The estimated permeability fields obtained by FEM-UKI and DeepOnet-UKI-Adaptive are very similar to the true permeability field, however DeepOnet-UKI-Direct's result differs dramatically, illustrating the usefulness of our framework.
 \begin{figure}[t]
    \begin{center}
        \begin{overpic}[width=0.31\textwidth]{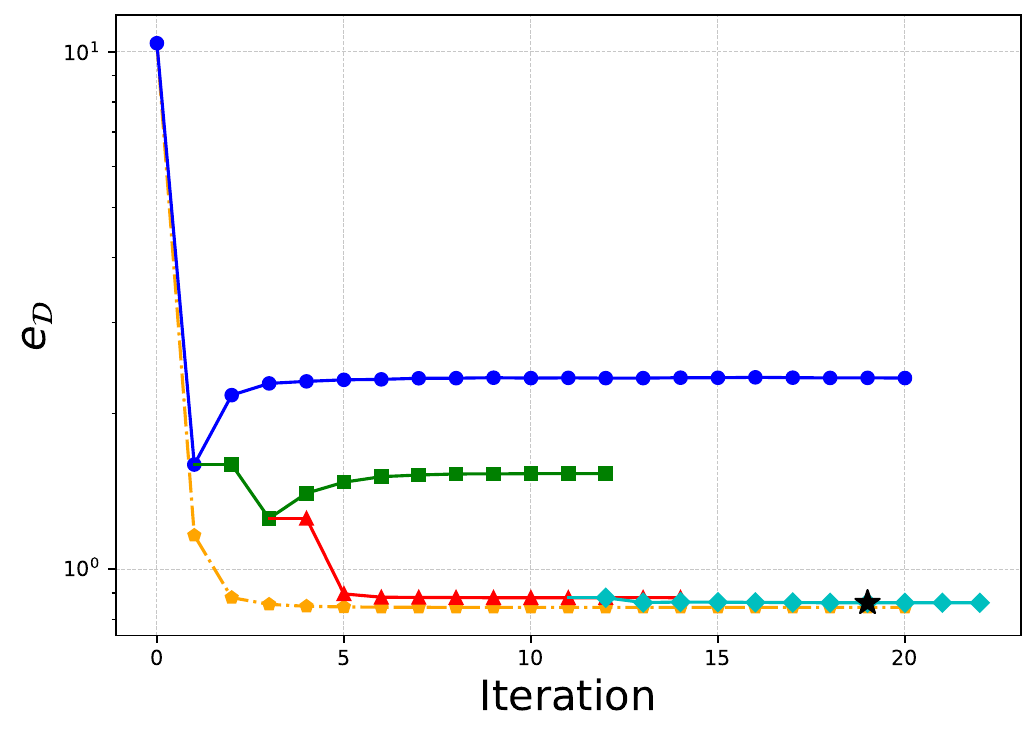}
                 \put (35,72) {\scriptsize {\bf fitting error}}
          \put (16,50) {\footnotesize \bf IDD case}
        \end{overpic}
        \begin{overpic}[width=0.31\textwidth]{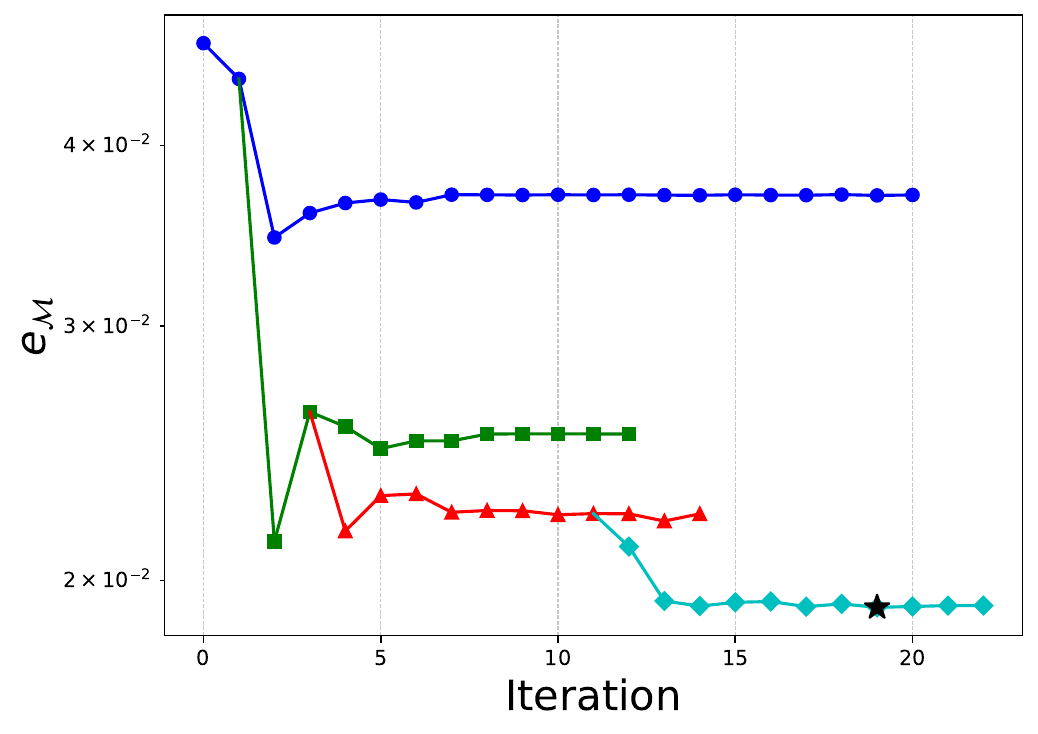}
           \put (38,72) {\scriptsize {\bf model error}}
        \end{overpic}
        \begin{overpic}[width=0.32\textwidth]{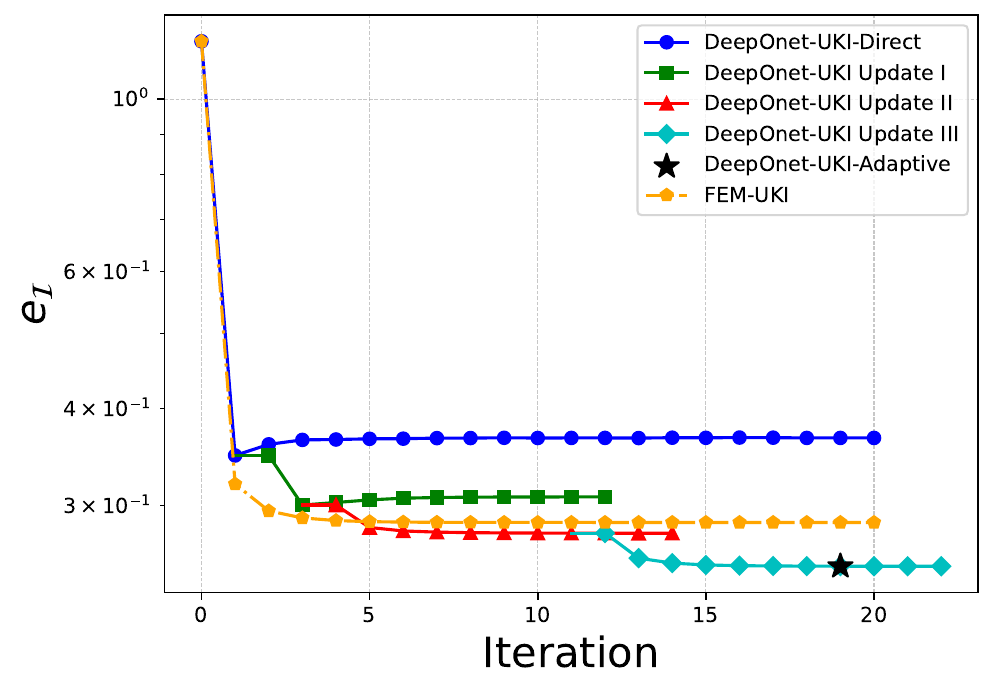}
           \put (25,69) {\scriptsize {\bf relative inversion error}}
        \end{overpic}
    \end{center}
    \begin{center}
        \begin{overpic}[width=0.31\textwidth]{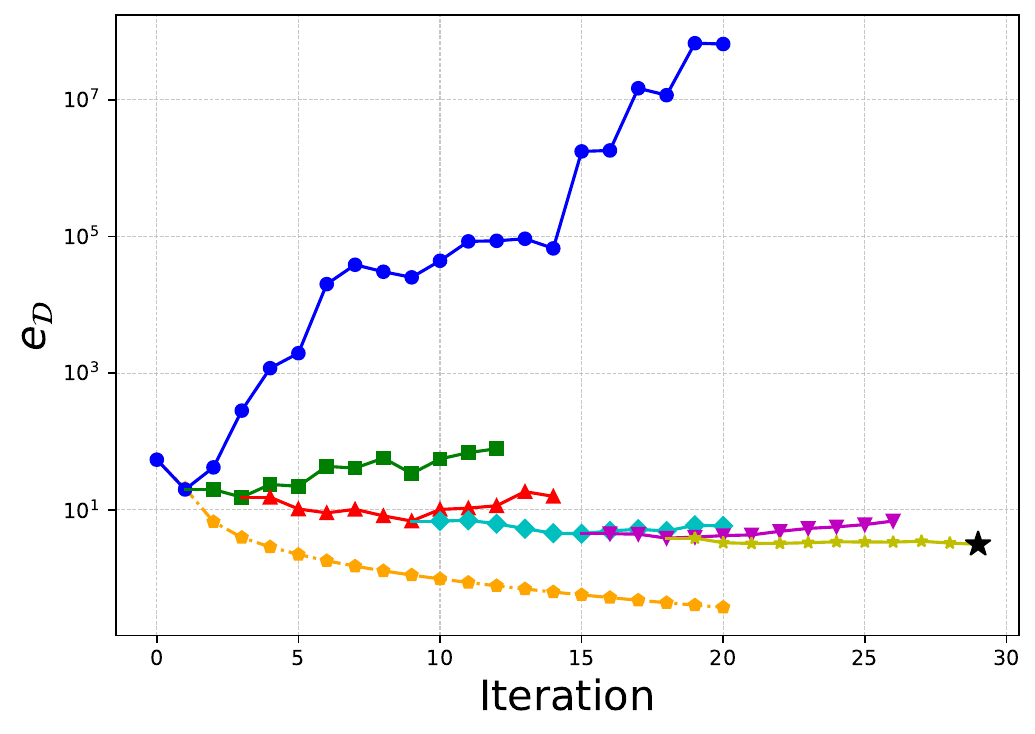}
            \put (14,50) {\footnotesize \bf OOD case}
        \end{overpic}
        \hspace{0.1cm}
        \begin{overpic}[width=0.31\textwidth]{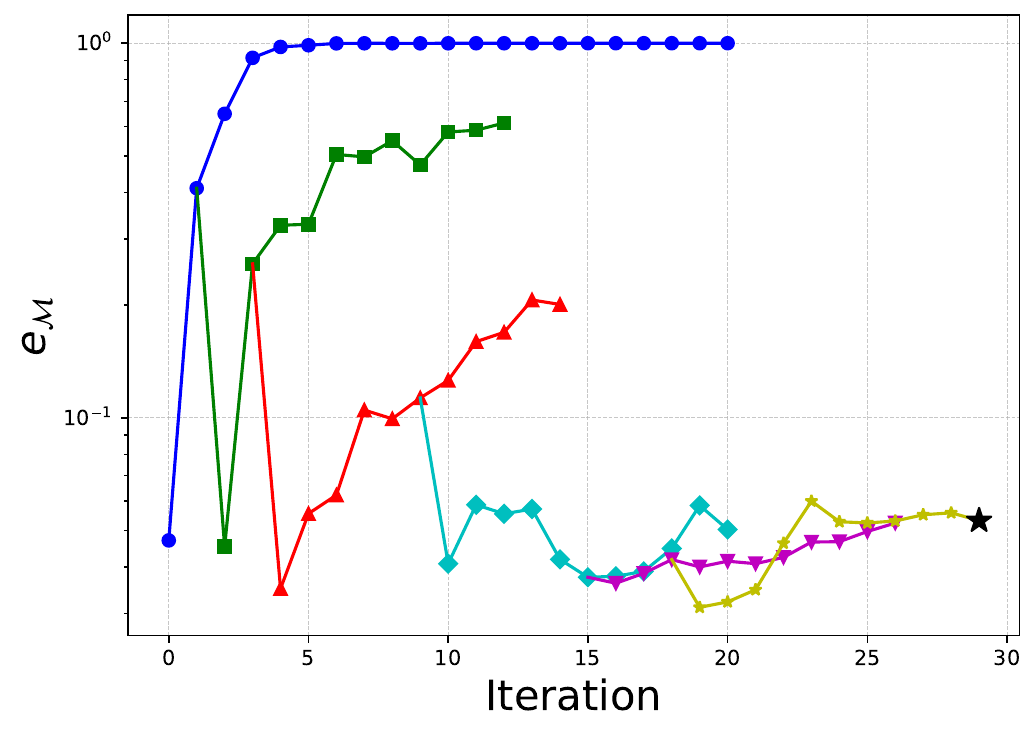}
        \end{overpic}
        \begin{overpic}[width=0.305\textwidth]{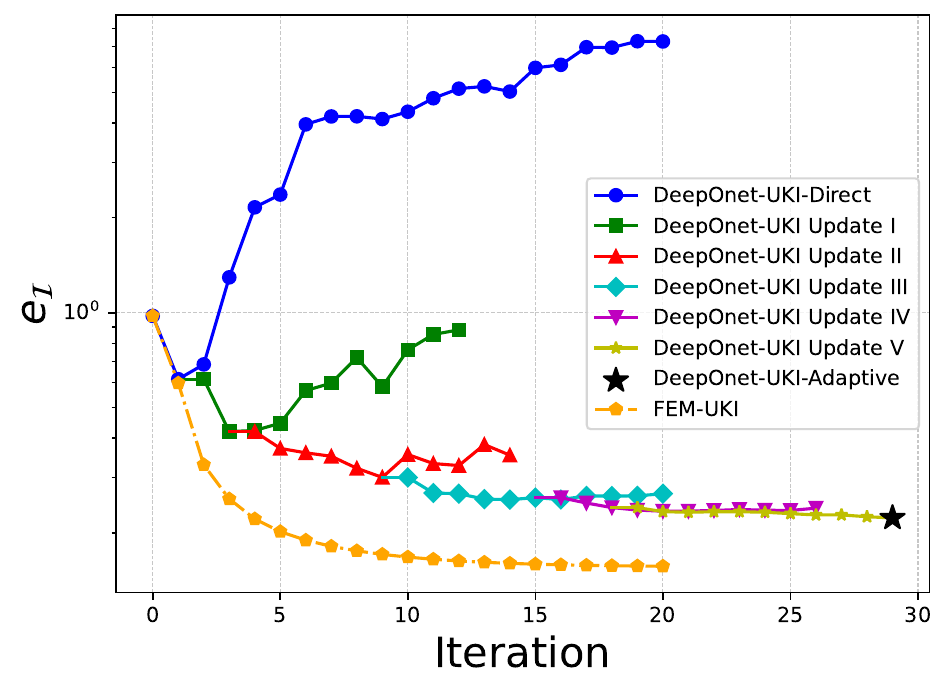}
        \end{overpic}
    \end{center}
    \caption{The data fitting error (left),  model error (middle) and relative inversion error (right) for Example 1.  Above:  IDD case. Below:  OOD case.}
    \label{loss_flow}
\end{figure}

\begin{figure}[t]
    \begin{center}
        \begin{overpic}[width=0.3\textwidth]{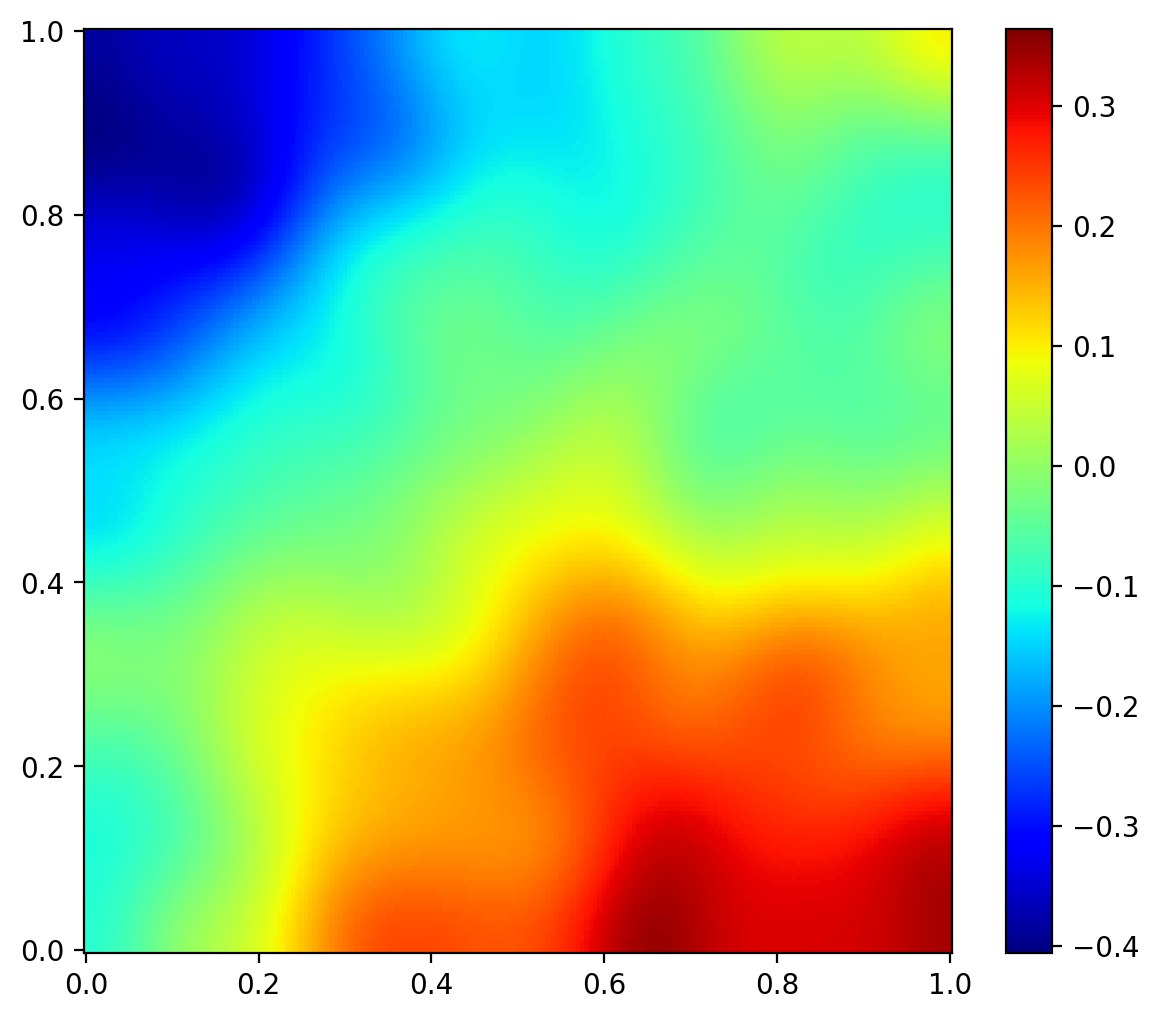}
        \put (30,88) {\footnotesize \bf FEM-UKI}
        \end{overpic}
        \begin{overpic}[width=0.3\textwidth]{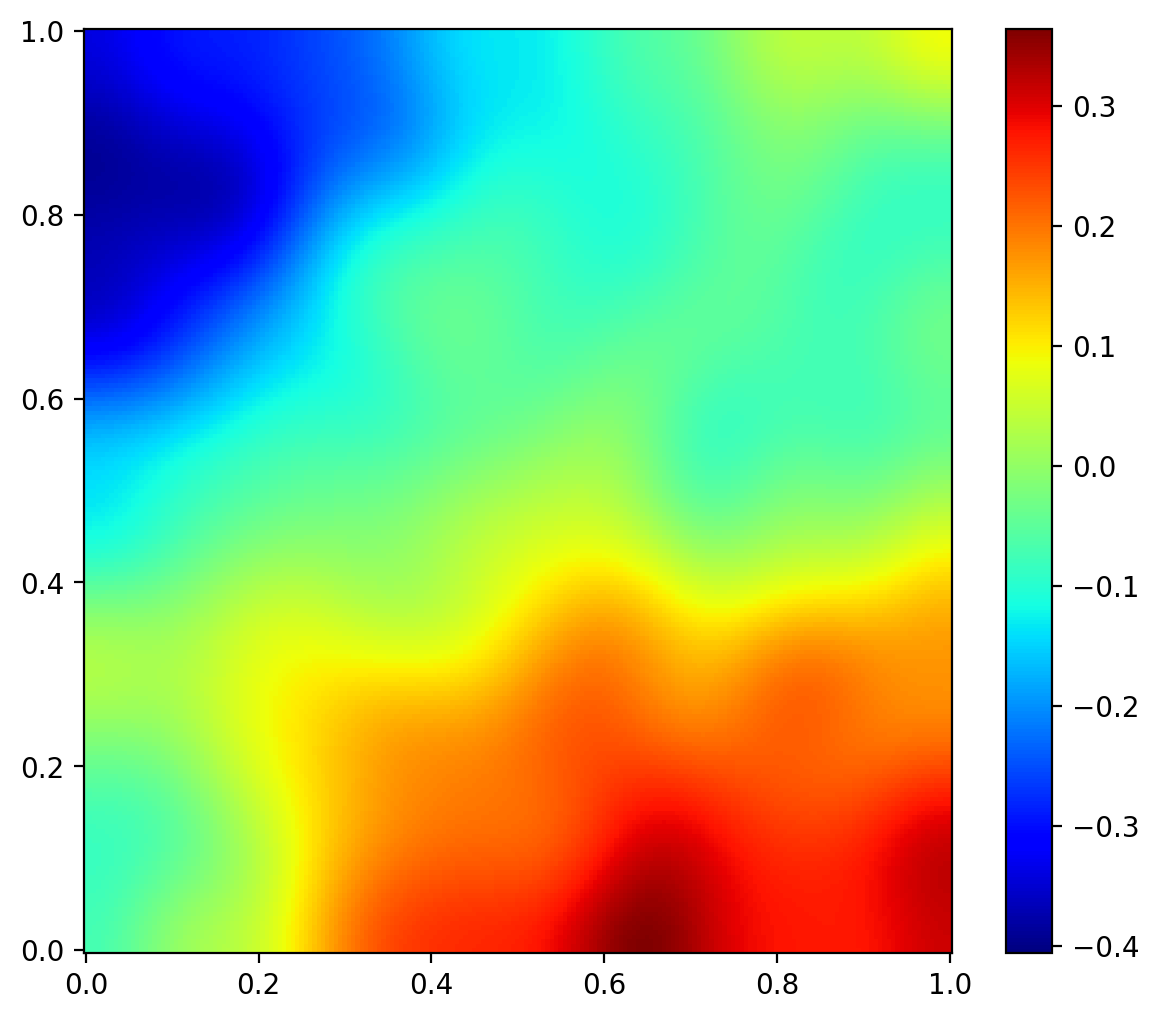}
            \put (8,88) {\footnotesize \bf DeepOnet-UKI-Adaptive}
        \end{overpic}
                \begin{overpic}[width=0.30\textwidth]{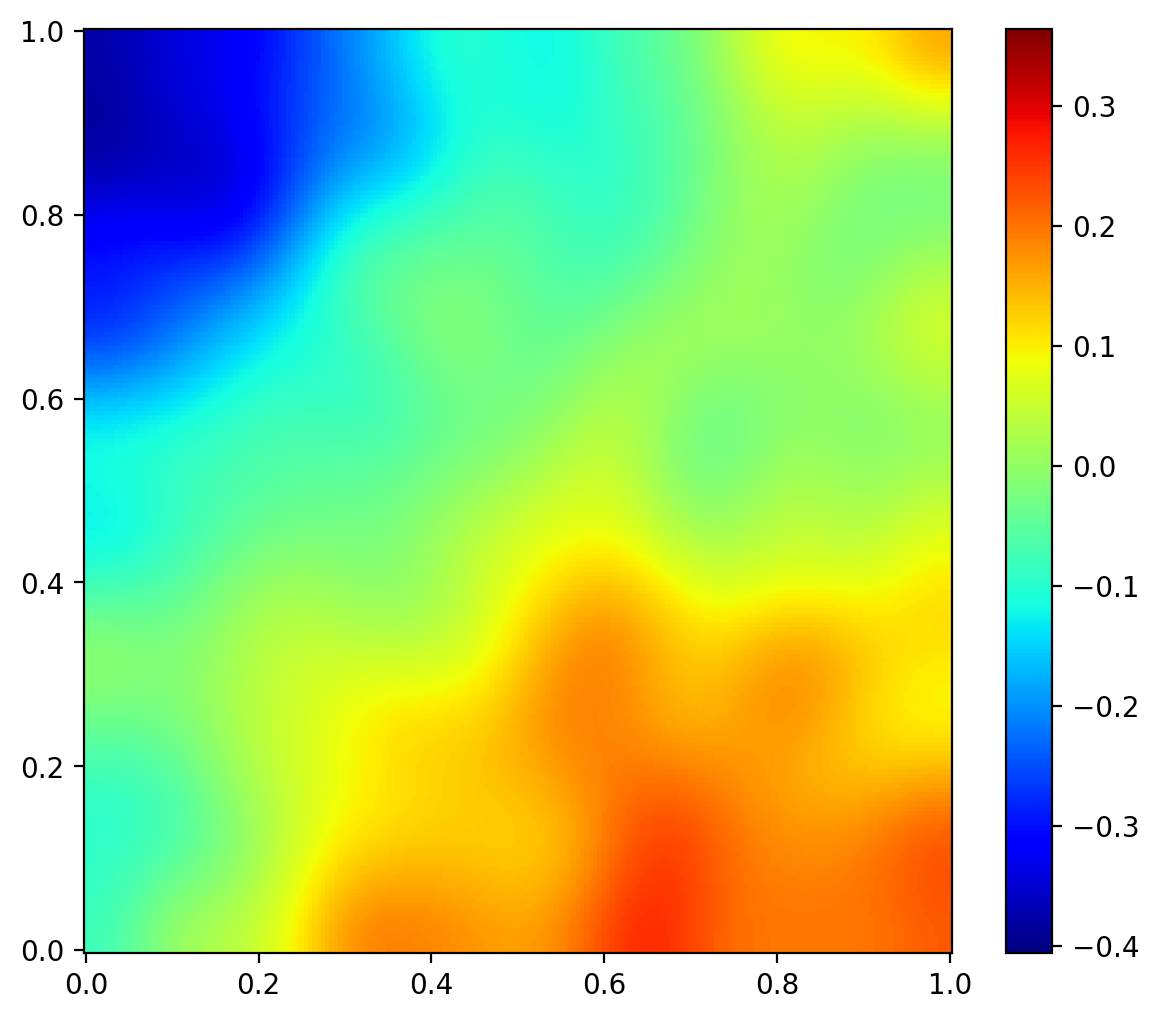}
            \put (10,88) {\footnotesize \bf DeepOnet-UKI-Direct}
        \end{overpic}
    \end{center}
    \begin{center}
        \begin{overpic}[width=0.305\textwidth]{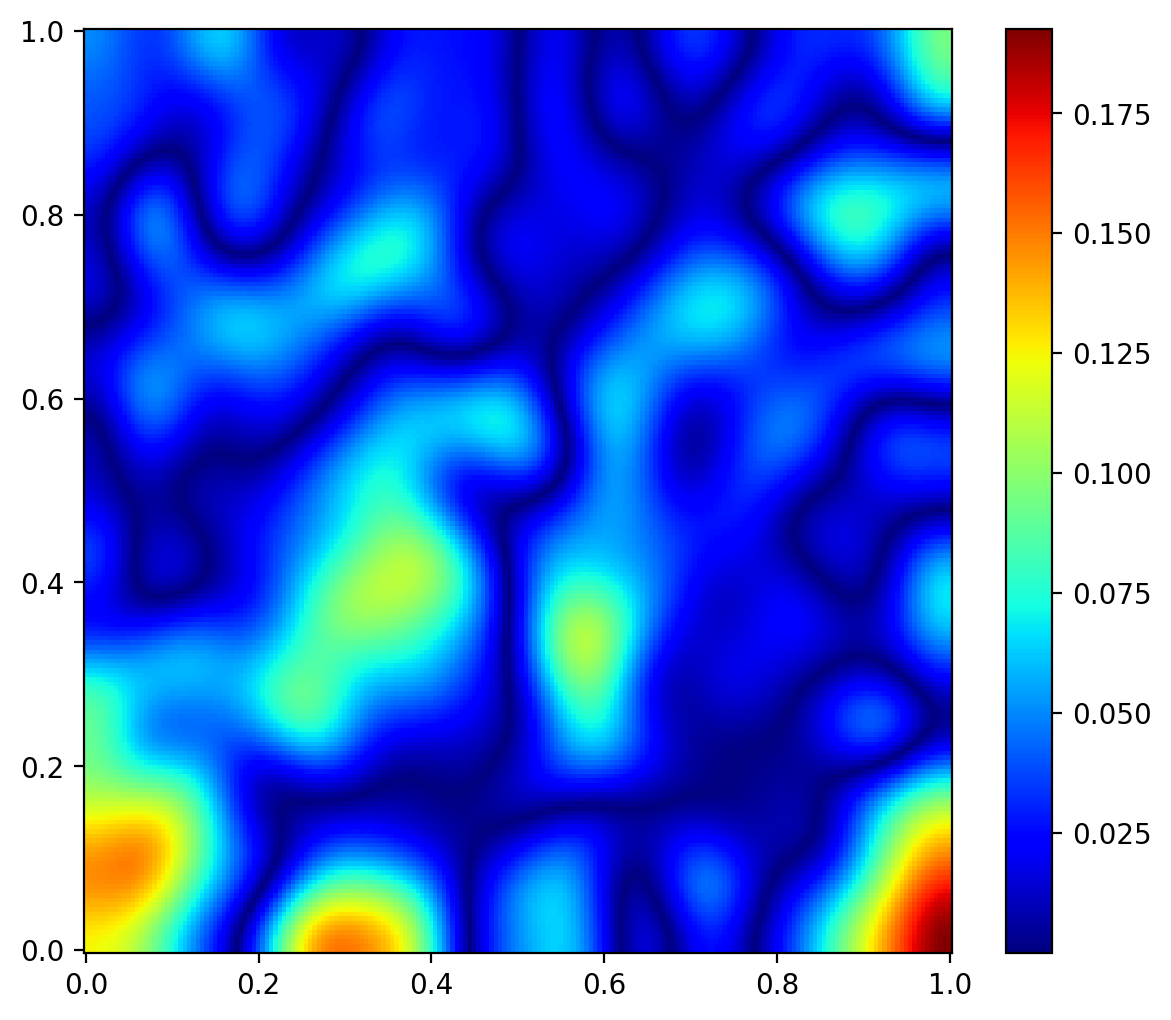}
        \end{overpic}
        \begin{overpic}[width=0.3\textwidth]{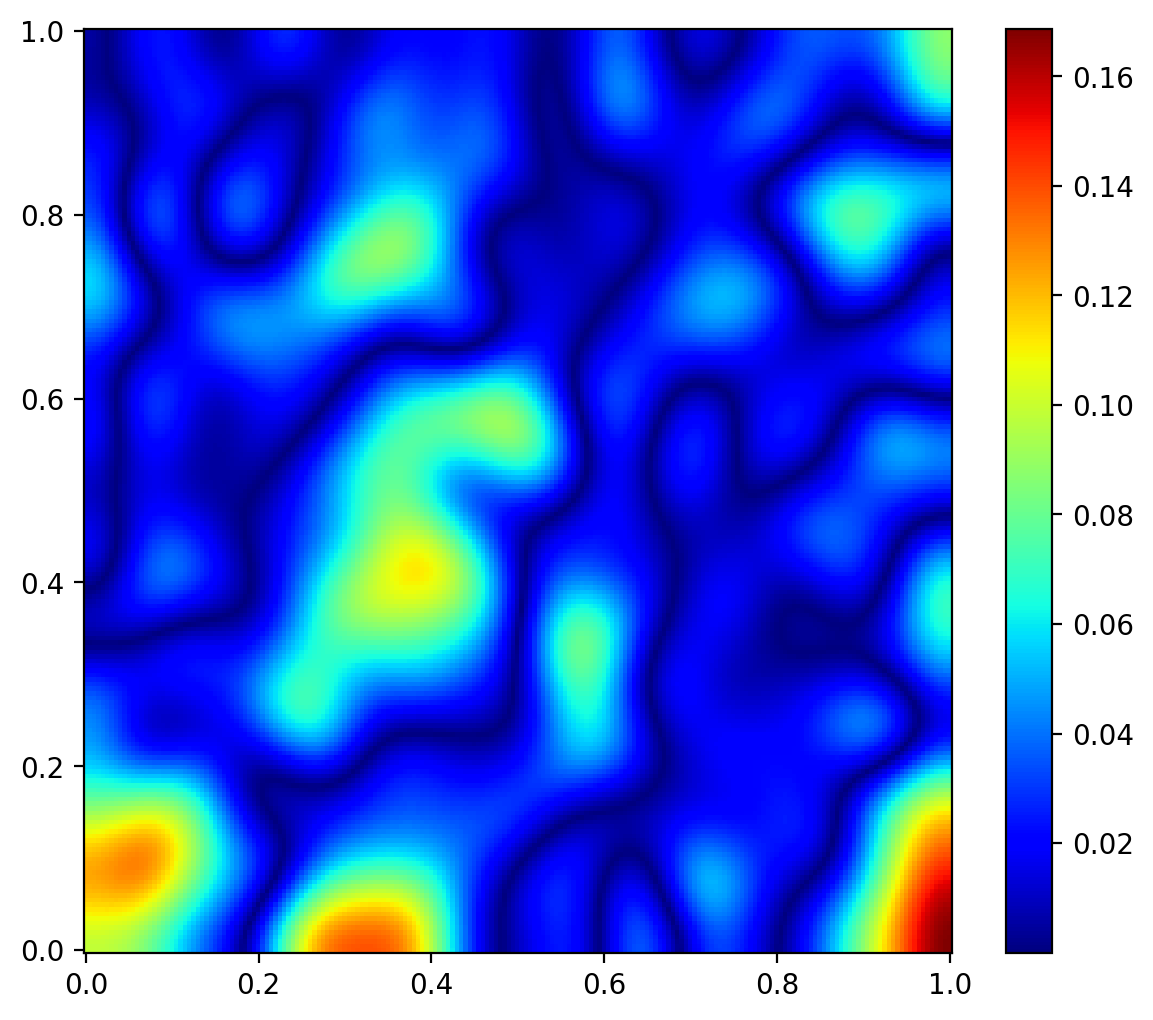}
        \end{overpic}
                \begin{overpic}[width=0.30\textwidth]{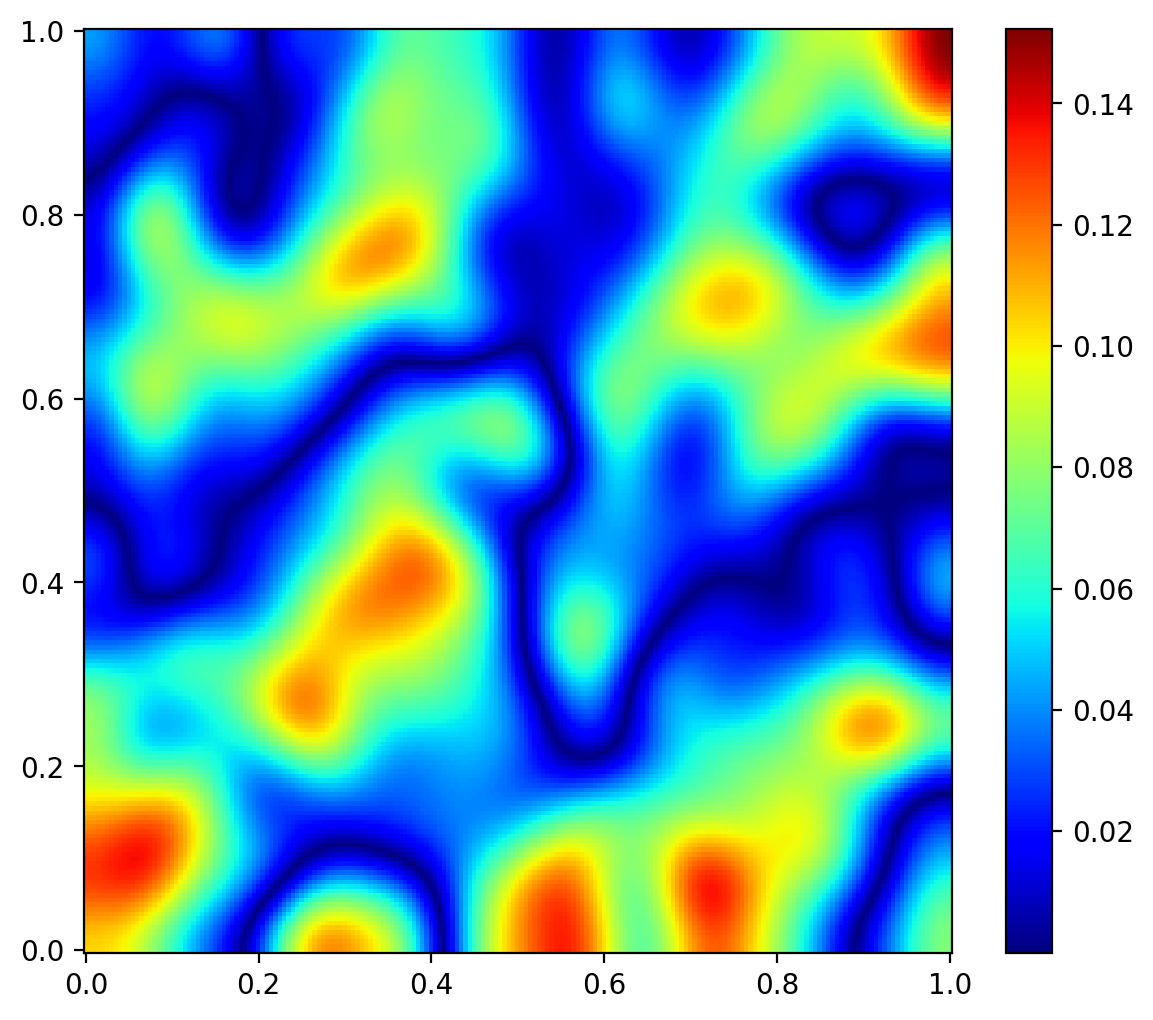}
        \end{overpic}
    \end{center}
    \vspace{-0.3cm}
    \caption{ {\bf IDD  case.} Above: the estimated permeability field obtained by different methods. Below: the absolute errors with respect to the true ones.}
    \label{kappa_flow_in_distribution}
\end{figure}
\begin{figure}[htbp]
    \begin{center}
        \vspace{0.3cm}
        \begin{overpic}[width=0.3\textwidth]{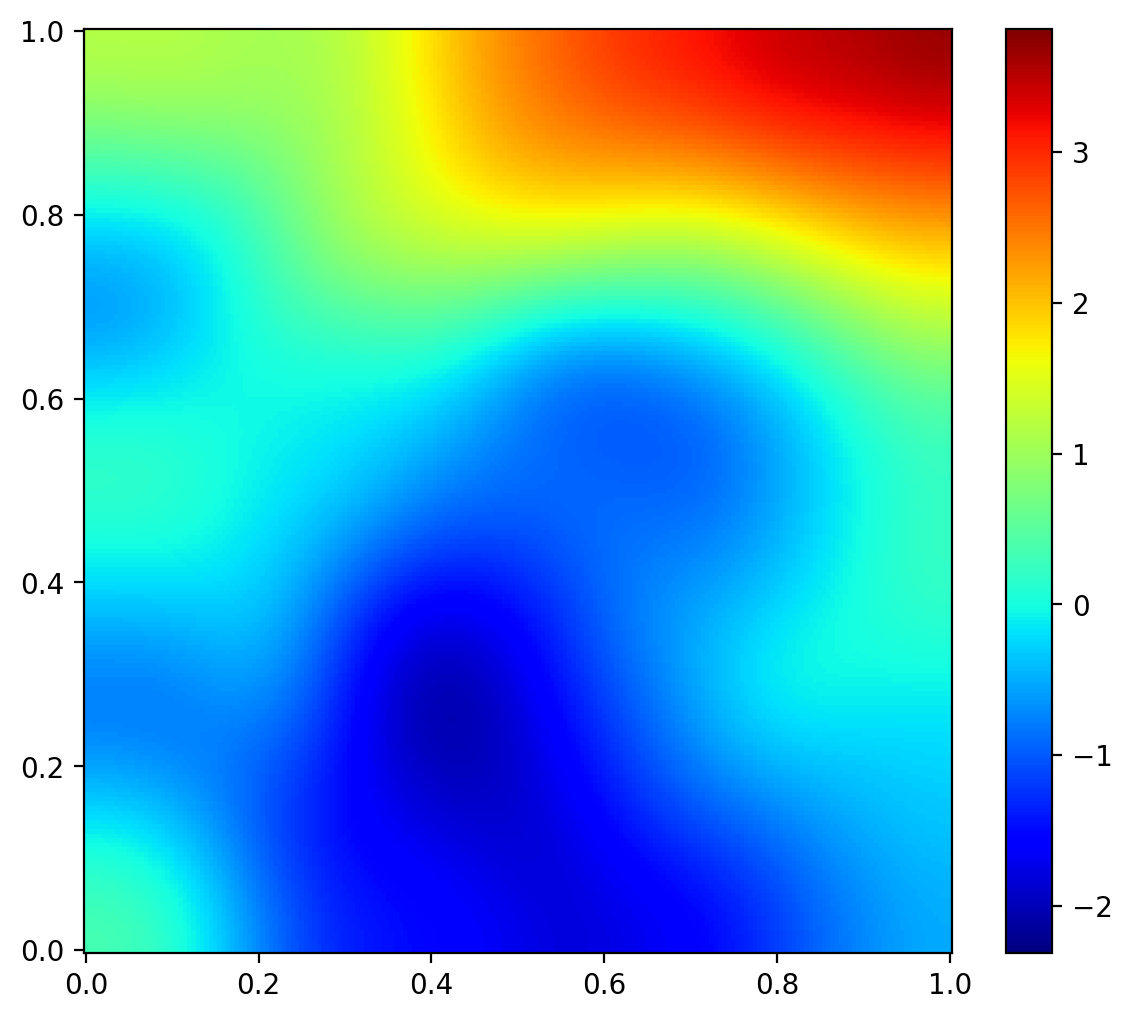}
        \put (30,90) {\footnotesize \bf FEM-UKI}
        \end{overpic}
        \begin{overpic}[width=0.3\textwidth]{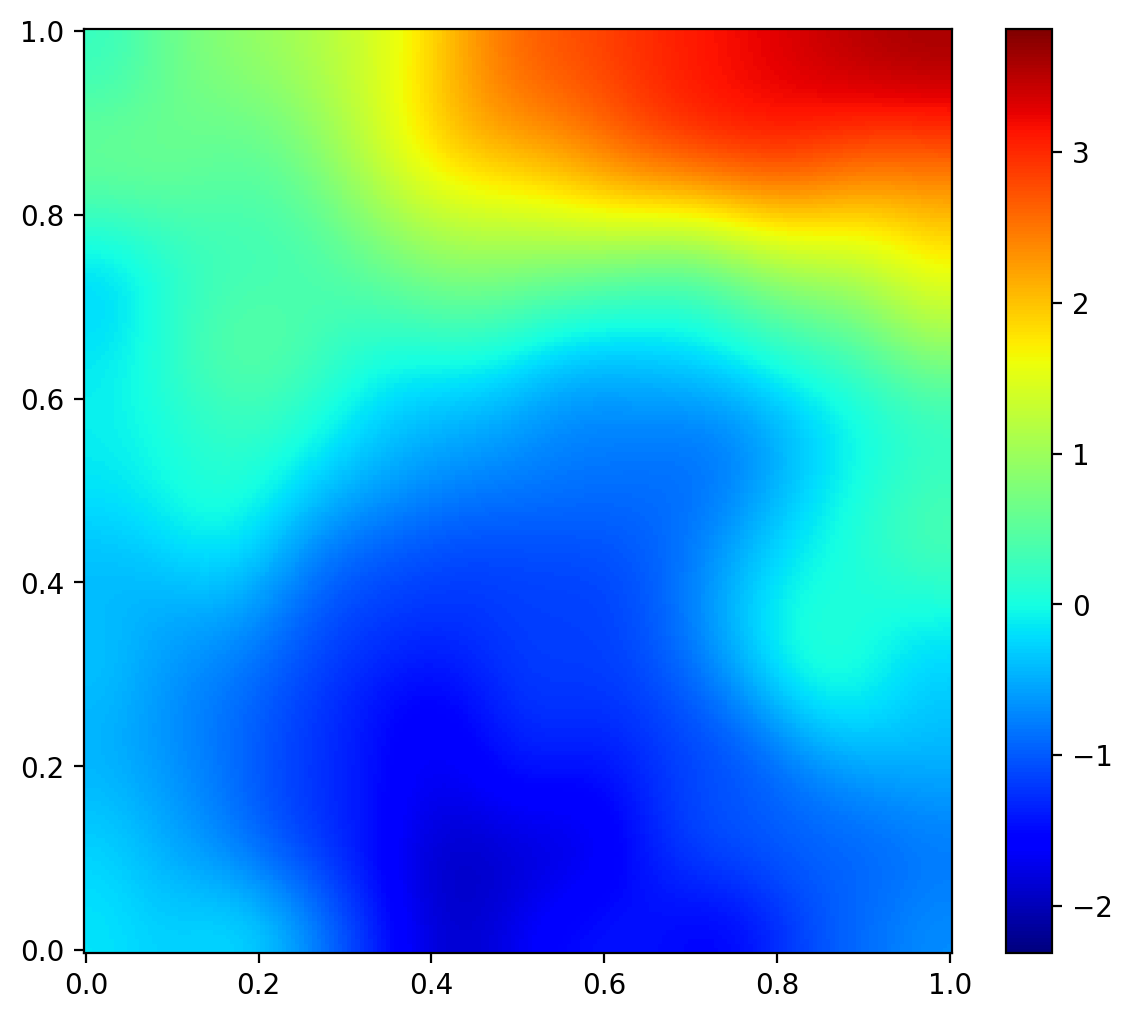}
            \put (8,90) {\footnotesize \bf DeepOnet-UKI-Adaptive}
        \end{overpic}
                \begin{overpic}[width=0.30\textwidth]{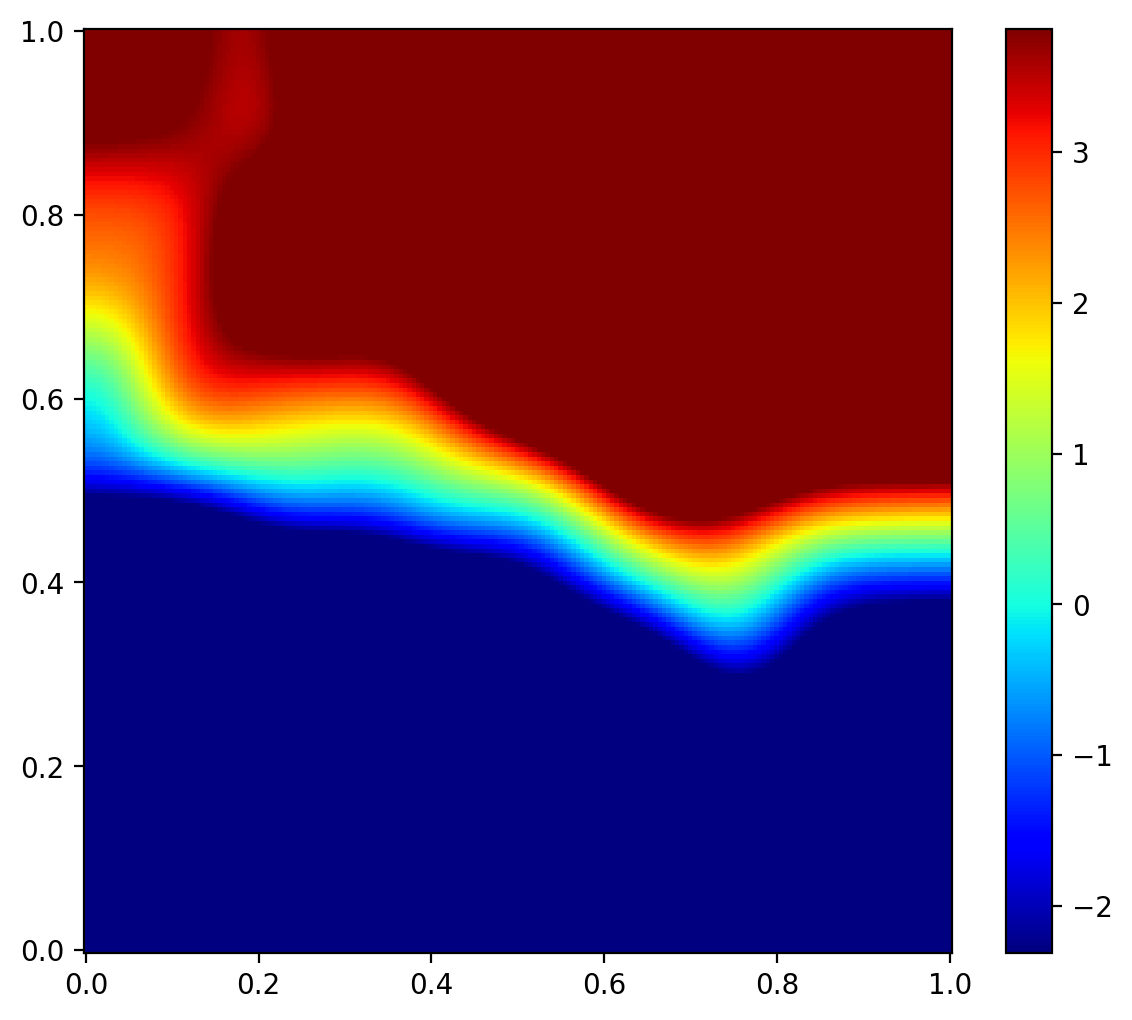}
            \put (10,90) {\footnotesize \bf DeepOnet-UKI-Direct}
        \end{overpic}
    \end{center}
    \begin{center}
        \begin{overpic}[width=0.31\textwidth]{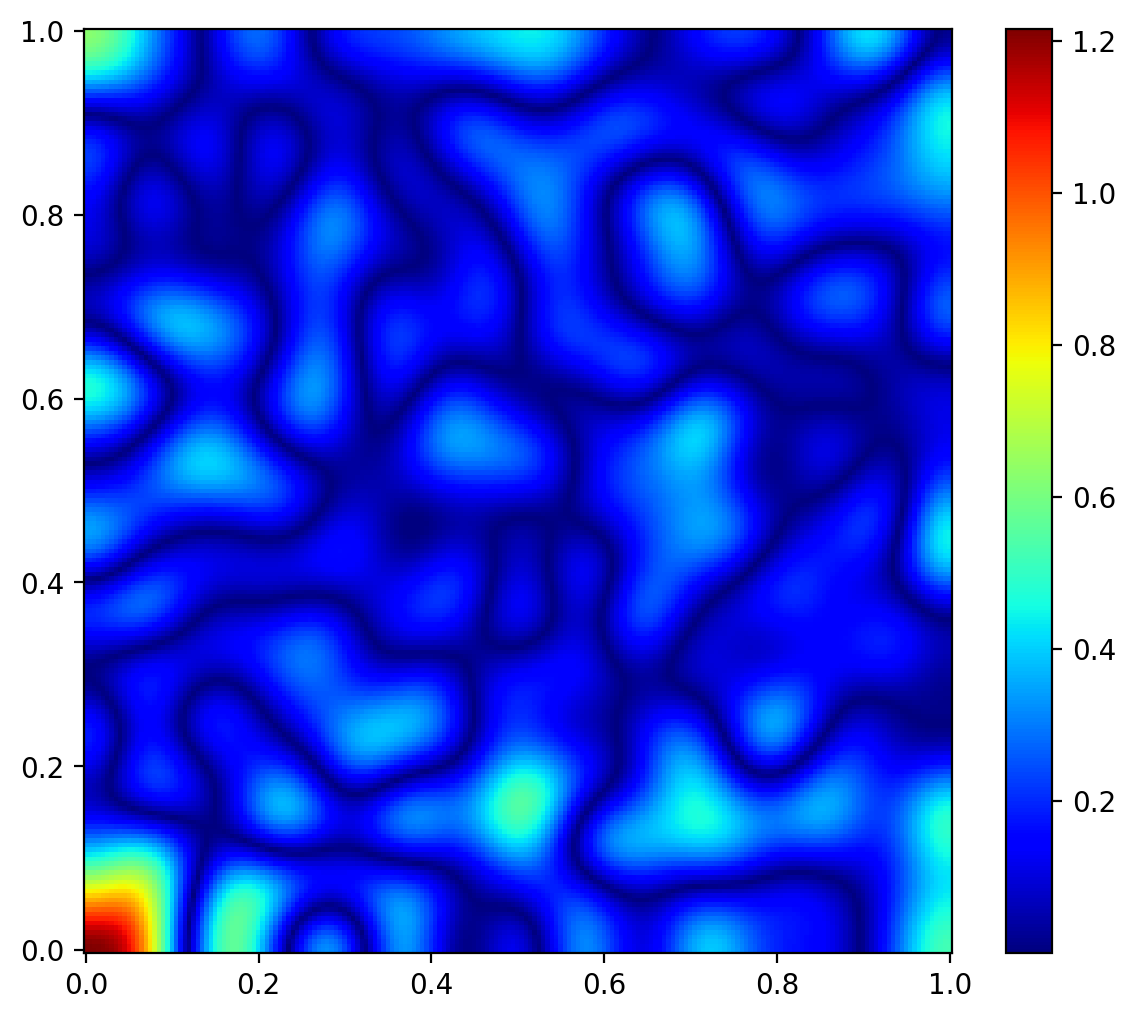}
        \end{overpic}
        \begin{overpic}[width=0.3\textwidth]{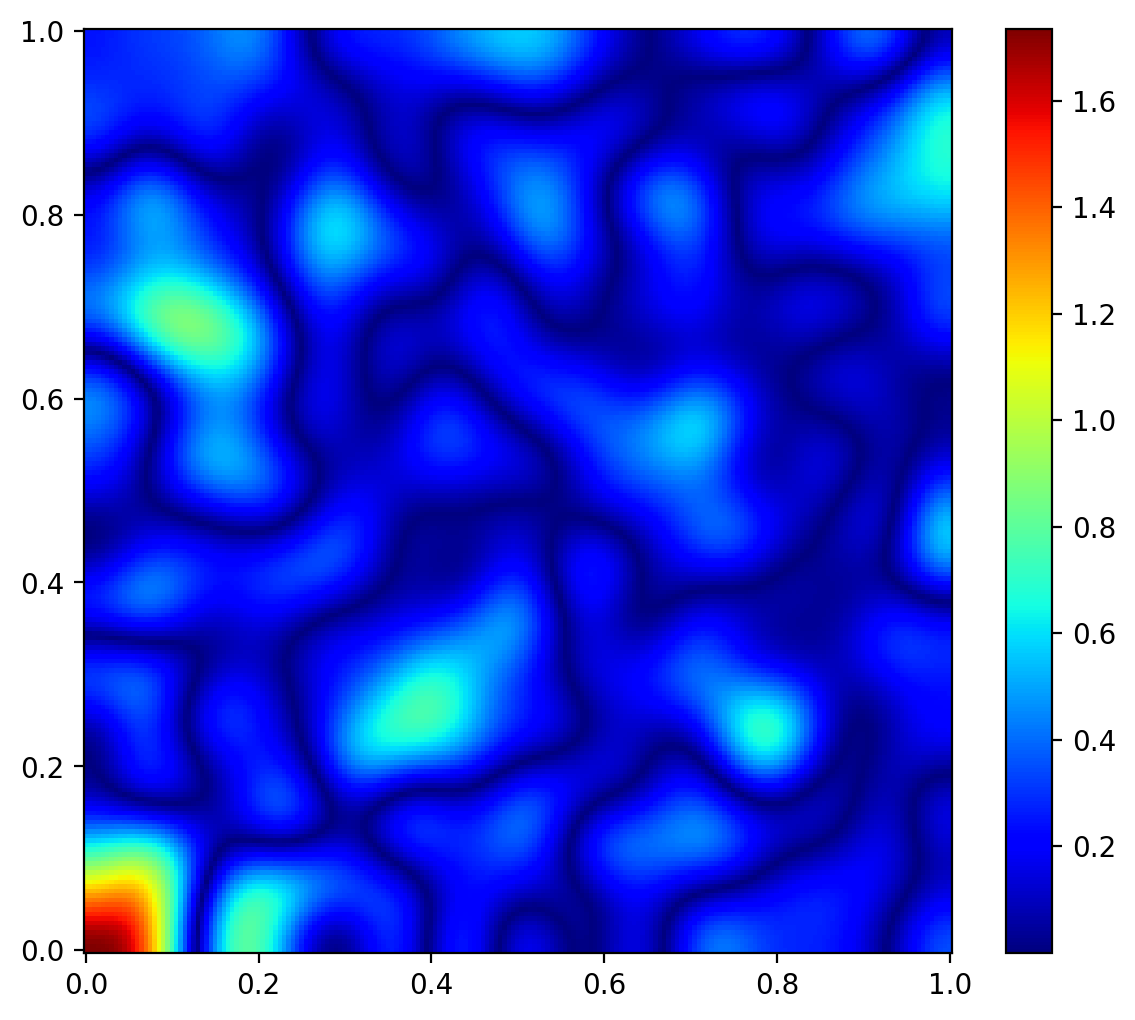}
        \end{overpic}
         \begin{overpic}[width=0.295\textwidth]{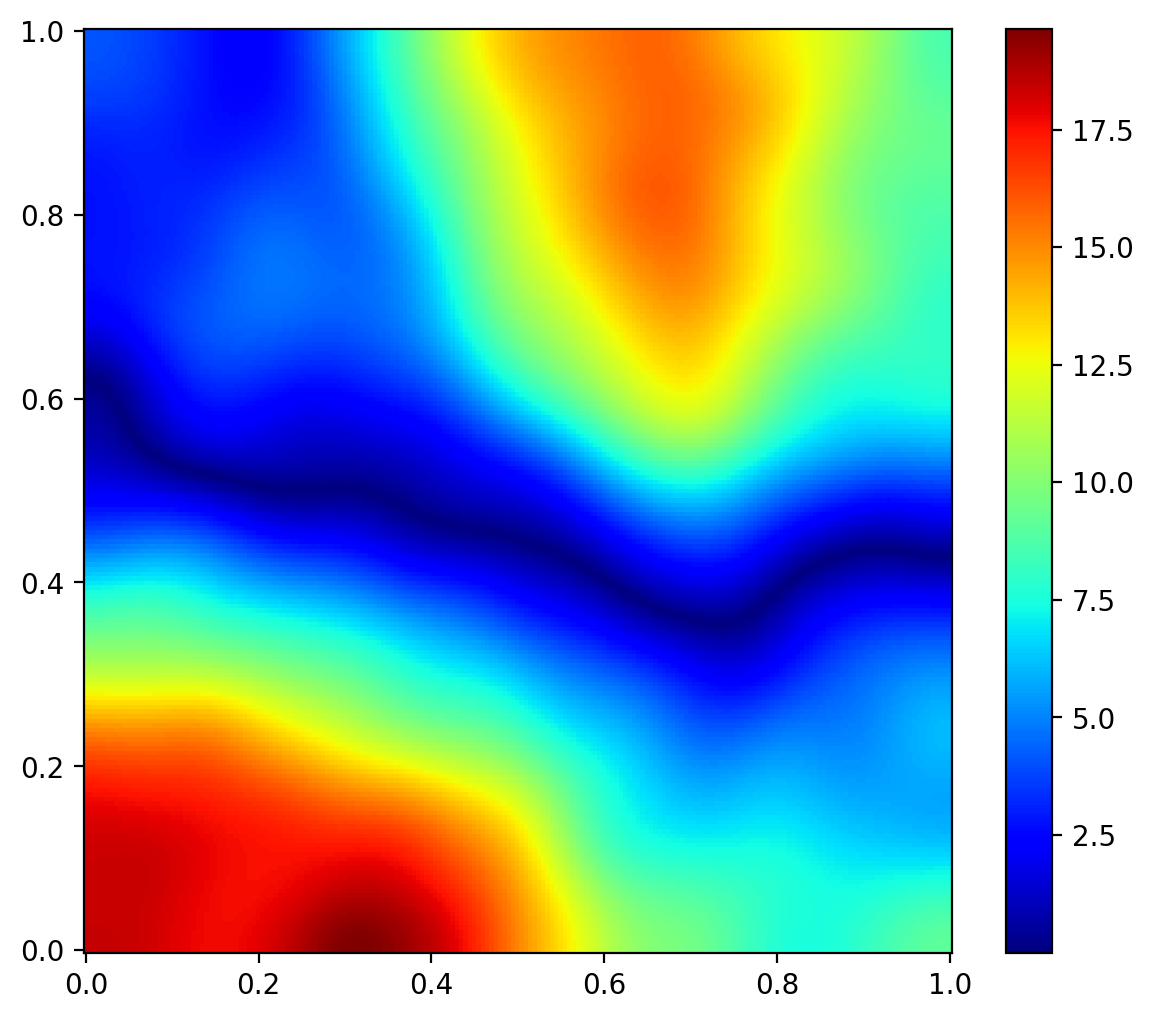}
        \end{overpic}
    \end{center}
    \caption{{\bf  OOD case.}  Above: the estimated permeability field obtained by different methods. Below: the absolute errors with respect to the true ones.}
    \label{kappa_flow}
\end{figure}

\begin{figure}[htbp]
    \centering 
    \begin{overpic}[width = 0.406\textwidth]{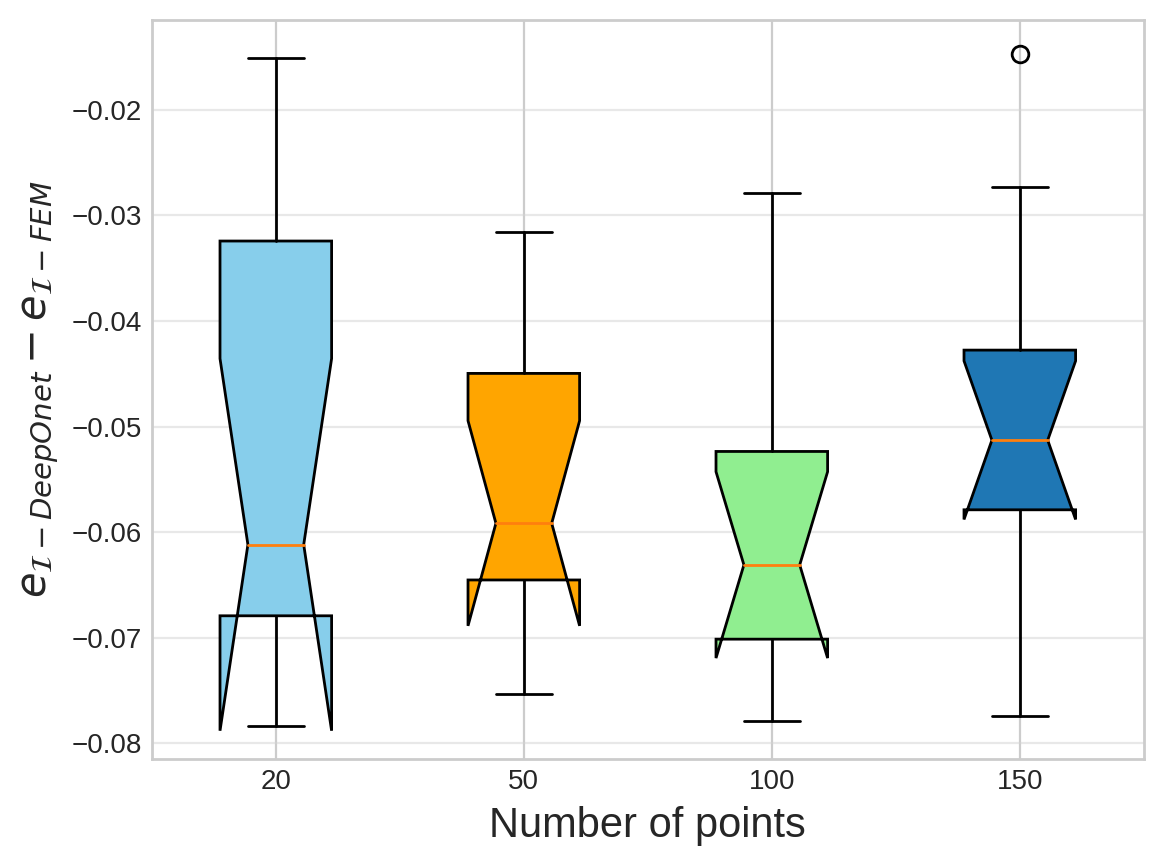}
    \put (40,76) {\footnotesize \bf IDD case}
    \end{overpic}
    \begin{overpic}[width = 0.4\textwidth]{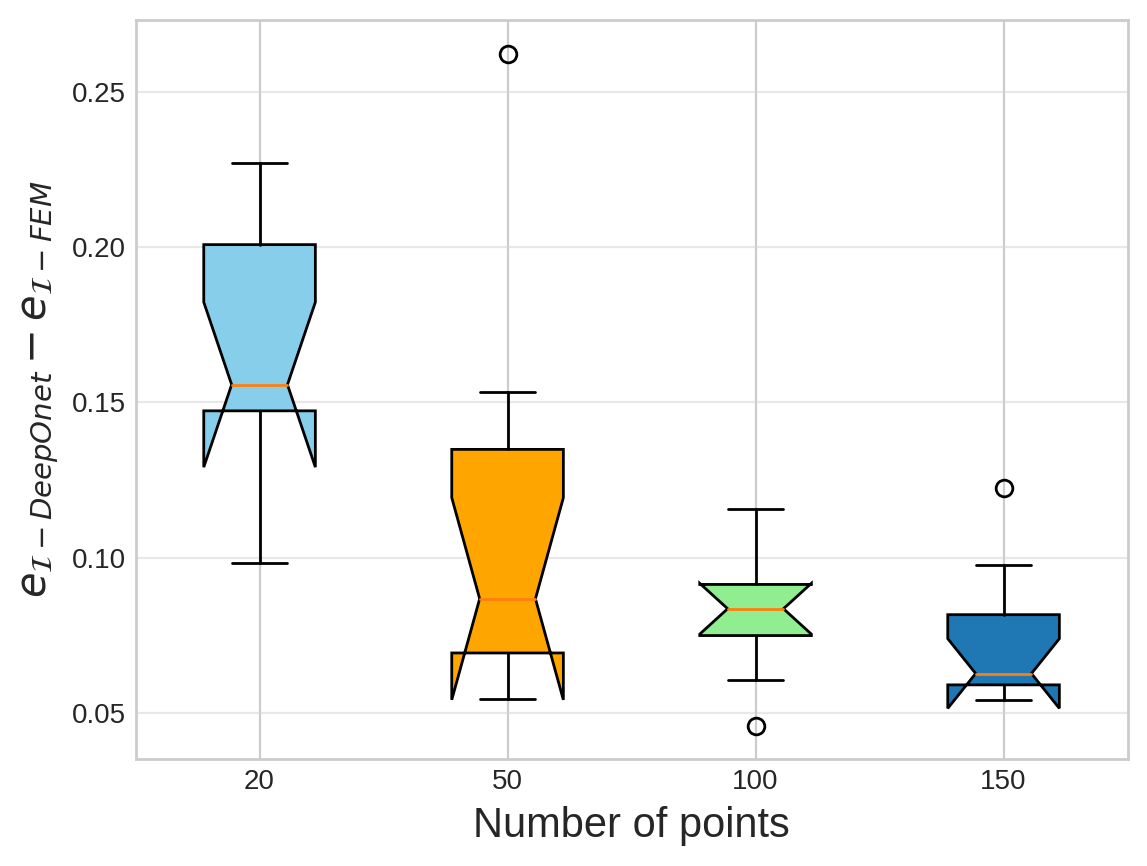}
    \put (40,77) {\footnotesize \bf OOD case}
    \end{overpic}
    \begin{overpic}[width = 0.4\textwidth]{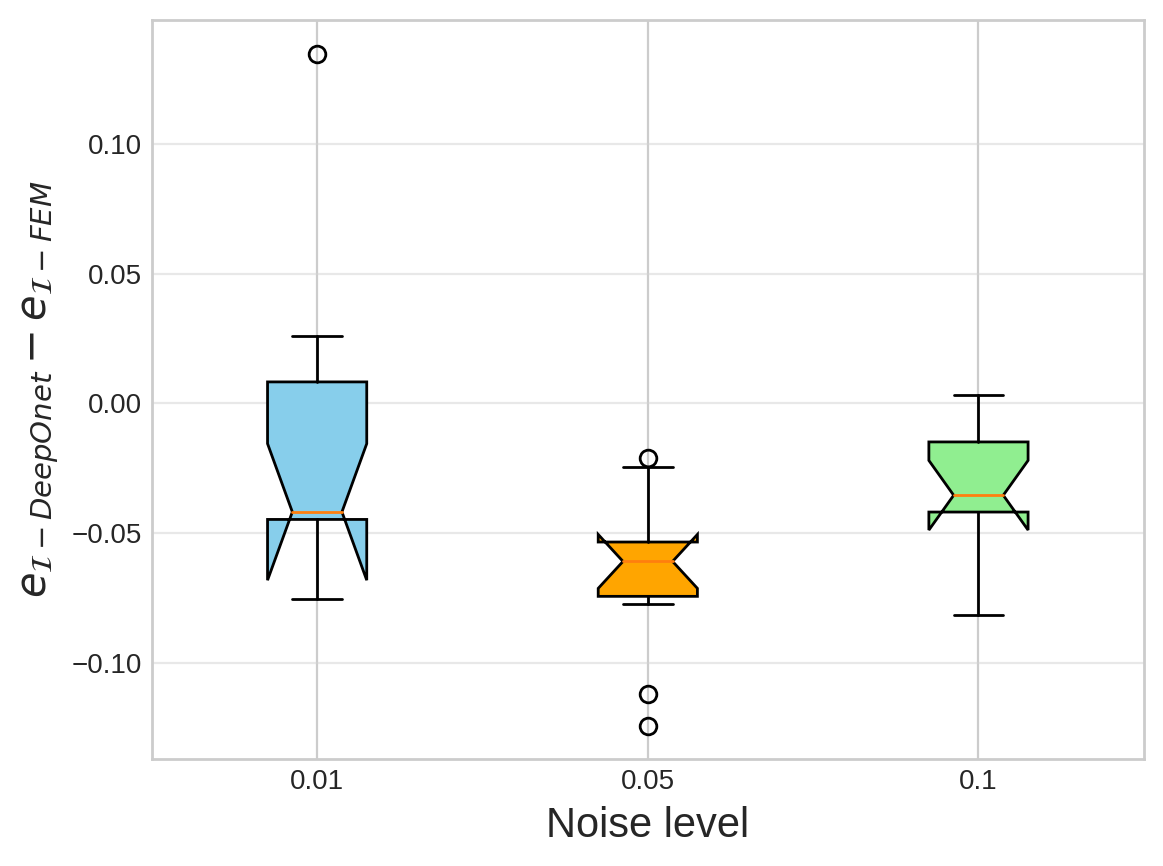}
    \end{overpic}
    \begin{overpic}[width = 0.405\textwidth]{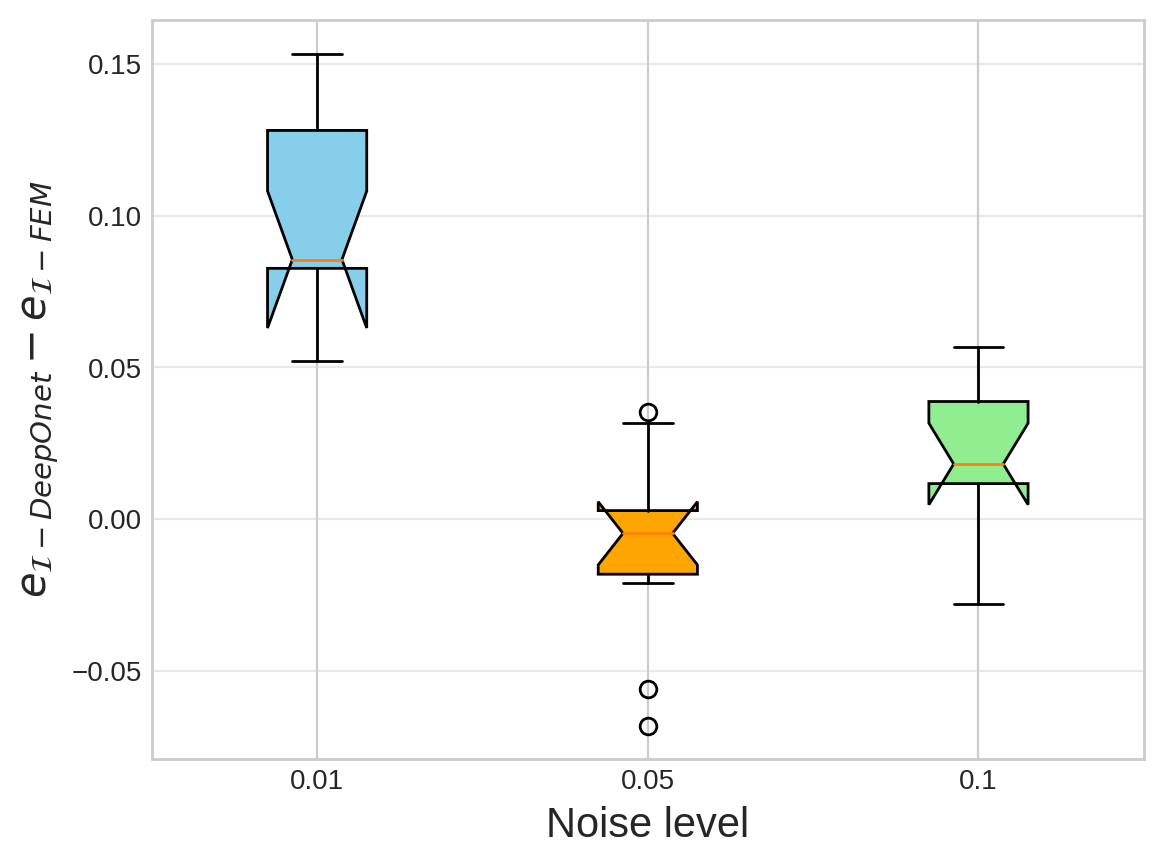}
    \end{overpic}
    \caption{The error box plot of $e_{\mathcal{I}-DeepOnet} - e_{\mathcal{I}-FEM}$ with different numbers  of  points $Q$(above) and noise levels (below). Left:  IDD case. Right: OOD case.}
    \label{num_flow}
\end{figure}
\begin{figure}[htbp]
    \begin{center} 
    \begin{overpic}[width = 0.4\textwidth]{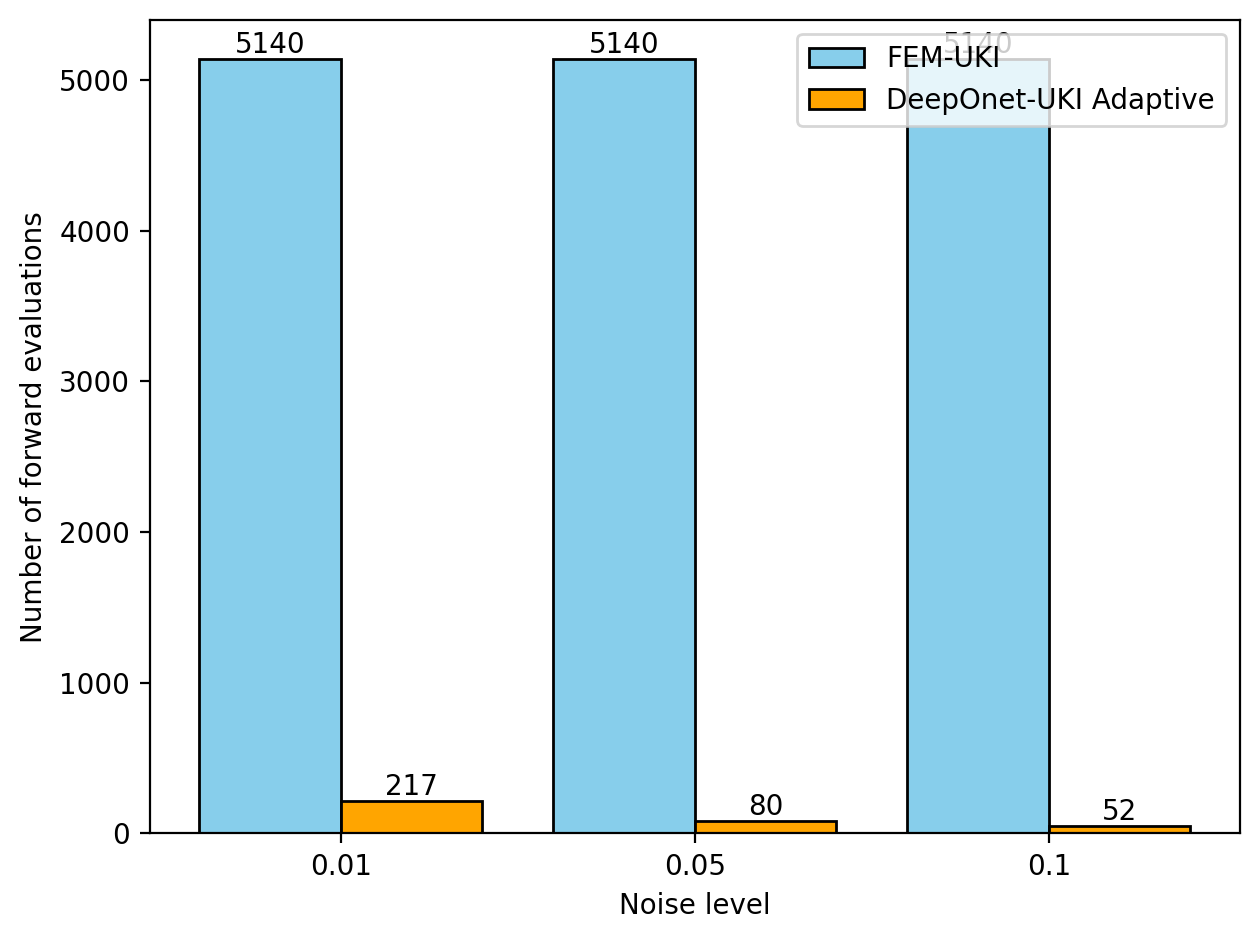}
    \put(40, 78){\footnotesize \bf IDD case}
    \end{overpic}
    \begin{overpic}[width = 0.406\textwidth]{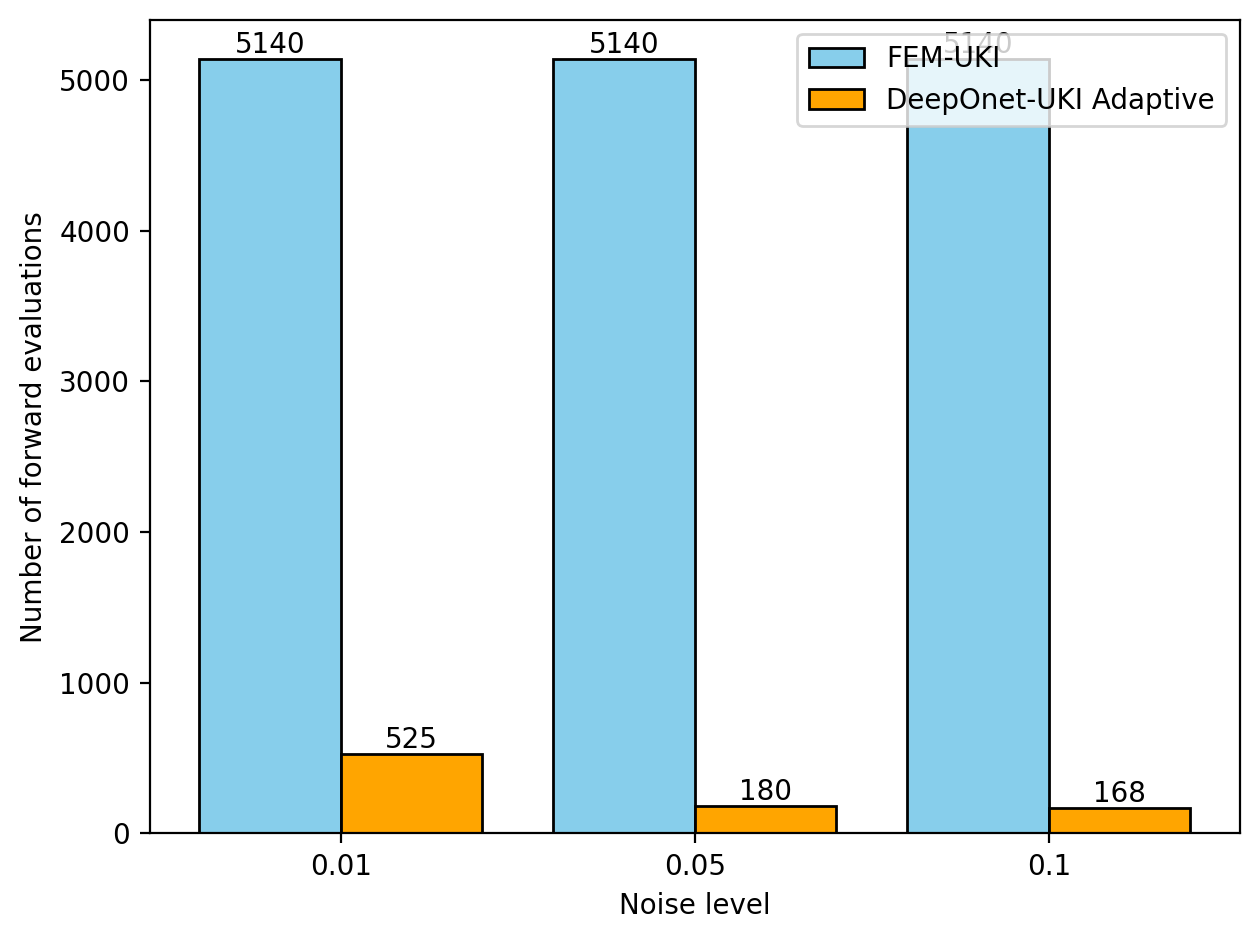}
    \put(40, 78){\footnotesize \bf OOD case}
    \end{overpic}
    \end{center}
    \caption{The mean total number of forward evaluations for FEM-UKI and DeepOnet-UKI- Adaptive respectively.  Left: IDD case and Right: OOD case. }
    \label{err_comparison}
\end{figure}

\begin{figure}[htbp]
    \centering 
    \begin{overpic}[width = 0.4\textwidth]{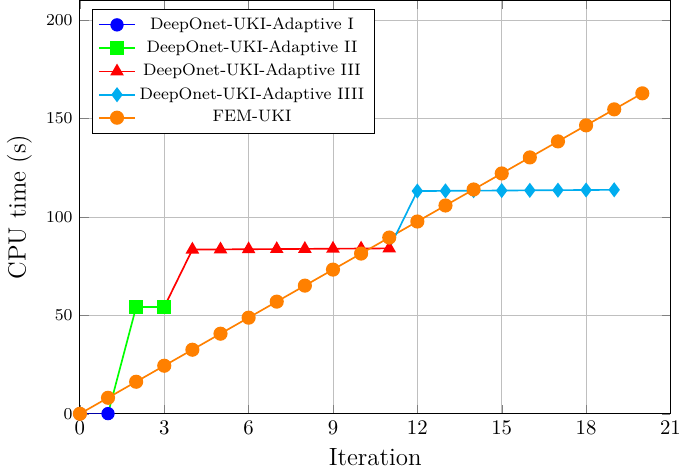}
        \put (40,70) {\footnotesize \bf IDD case}
    \end{overpic}
    \begin{overpic}[width = 0.4\textwidth]{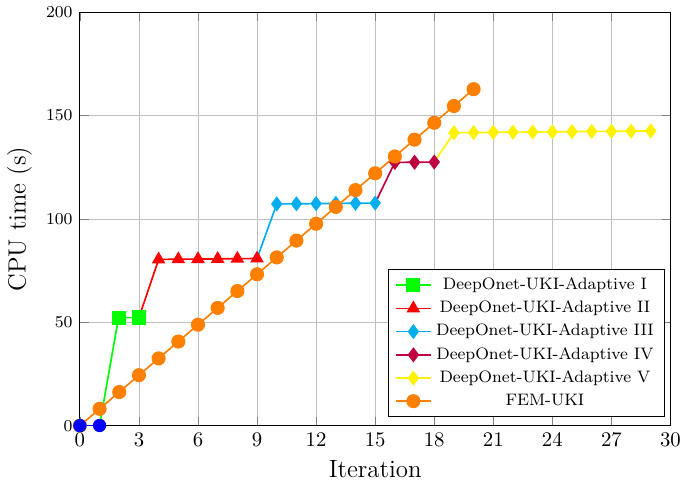}
        \put (40,70) {\footnotesize \bf OOD case}
    \end{overpic}

    \begin{overpic}[width = 0.41\textwidth]{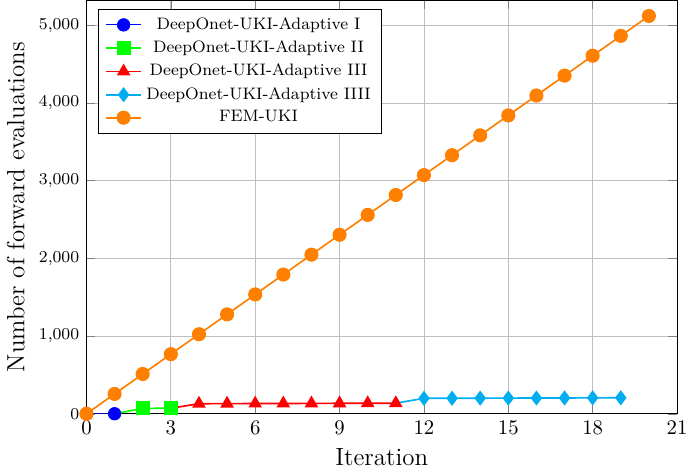}
    \end{overpic}
    \hspace{-0.2cm}
    \begin{overpic}[width = 0.41\textwidth]{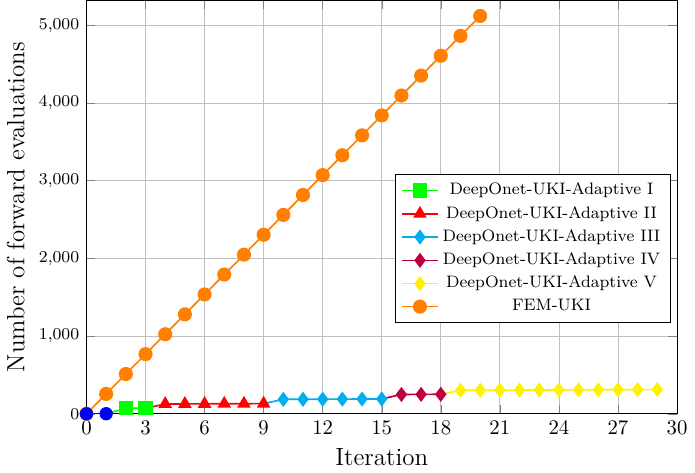}
    \end{overpic}
    \caption{Left: the average computational time and average number of forward evaluations each iteration for IDD case. Right: the average computational time and average number of forward evaluations each iteration for IDD case. Current noise level is 0.01.}
    \label{flow_computation}
\end{figure}

To test the effect of the number $Q$ of adaptive samples used in each refinement, we repeat the experiment ten times for each $Q\in[20,50,100,150]$.  The error box of $e_{\mathcal{I}-DeepOnet} - e_{\mathcal{I}-FEM}$ is plotted Fig.\ref{num_flow}. It is evident from the IDD data that the relative inversion error does not decrease significantly with increasing data. This suggests that a small set of adaptive samples-- roughly 50 -- can meet the requirements for accuracy and efficiency. On the other hand, the relative inversion error for OOD data steadily drops with increasing dataset size.  In order to  examine the effects of varying noise levels, we are going to perform the experiment with three different noise levels (0.01, 0.05, and 0.1) and then repeat it with ten different UKI initial values. The numerical results are also shown in Fig.\ref{num_flow}. We can clearly observe that the relative inversion error gradually drops as noise levels rise, suggesting that higher noise levels are less sensitive  to model errors. Consequently, our framework performs better in real-world applications with higher noise levels. 

To test the computational efficiency of our adaptive framework,  we plot the mean total online forward evaluations in Figure \ref{err_comparison}. Our method incurs a significantly lower cost compared to conventional numerical methods, even when adaptive refinement is applied. Specifically, our approach requires a maximum of only 50 samples for model refinement, whereas FEM-UKI demands 5,140 forward evaluations.  To further evaluate the computational efficiency of our adaptive framework, we present the average computational CPU times and the number of forward evaluations per valid iteration in Figure \ref{flow_computation}. The results show that, apart from the initial offline cost for generating data points and training, the iterative computational efficiency of our adaptive framework surpasses traditional FEM solvers for both IDD and OOD cases. Given that the online fine-tuning process is highly efficient, with retraining taking just a few seconds, the cost of these forward simulations becomes negligible. Moreover, even when considering the offline cost, the total number of forward evaluations required by our framework is substantially lower than that of FEM-UKI, indicating that our scheme offers superior efficiency, particularly when the underlying PDE is expensive to solve.

\subsection{Example 2: The heat source inversion problem}
 \label{Heat_source}
Consider the following  heat conduction problem in $\Omega$
 \begin{equation}
    \label{heat_equation}
    \begin{aligned}
    u_{t}(\mb{x}, t) - \Delta u(\mb{x}, t) &= f(\mb{x},t), &&\text{ in } \Omega\times[0,1], \\ 
    u(\mb{x}, 0) &= \textcolor{black}{u_0(\mb{x})},&&\text{ in } \Omega.
    \end{aligned}
 \end{equation}
The objective is to identify the heat source $f$ from noisy measurements. To illustrate our method more clearly, we divide this inversion task into two cases. In both cases, the surrogate is the DeepOnet model with same architectures. The FEM method is used to solve the forward problem on a $70\times70$ grid and the resulting differential equations are integrated using the implicit-Euler scheme.

\textbf{Case I:}  In this case, we consider a 2D heat source inversion problem which is adapted from\cite{marzouk2007stochastic}.  Take $f = \frac{s}{2\pi\sigma^2}\exp\left(-\frac{\left|\boldsymbol{\chi}-\mathbf{x}\right|^2}{2\sigma^2}\right)[1-H(t-T)]$ with zero initial condition and zero Neumann boundary condition, where $H$ is the heavy-side function. Take $s = 5, \sigma = 0.1, T = 0.05$. The inverse problem is to infer the location by giving the observations at $t = 0.05$ and $t = 0.15$. In this paper, we take the ground truth to be $\mathbf{x}_{true} = [0.2, 0.2]$. A uniform $3 \times 3$ sensor network is used to collect noisy point-wise observations from the PDE field solution. At each sensor location, two measurements are taken at $t = 0.05$ and $t = 0.15$, for a total of 18 measurements. 

To replace the forward model, DeepOnet is trained on $[0.5, 1]\times [0.5, 1]$ with 500 uniformly distributed samples. We use UKI with our adaptive scheme to run the experiment with initial value $[0.6, 0.6]$. The numerical results are shown in Fig.\ref{source}. The DeepOnet-UKI-Direct method can provide only a rough estimate. While with our adaptive refinement, the DeepOnet-UKI-Adaptive can nearly achieve the same accuracy as FEM-UKI. That is because the adaptive sampling can select import samples in the high density area of the approximate posterior and the surrogate is refined to reduce the local model error. This phenomenon can be further observed in Fig.\ref{path}, which plots the inversion trajectories and samples distribution for our adaptive method. We can clearly see that with refinements, the surrogate gradually become accurate in the high density area of the true posterior, which can modify the inversion trajectory and thus lead to more accurate results compared to DeepOnet-UKI-Direct.

\begin{figure}[htbp]
    \centering 
    \begin{overpic}[width = 0.3\textwidth]{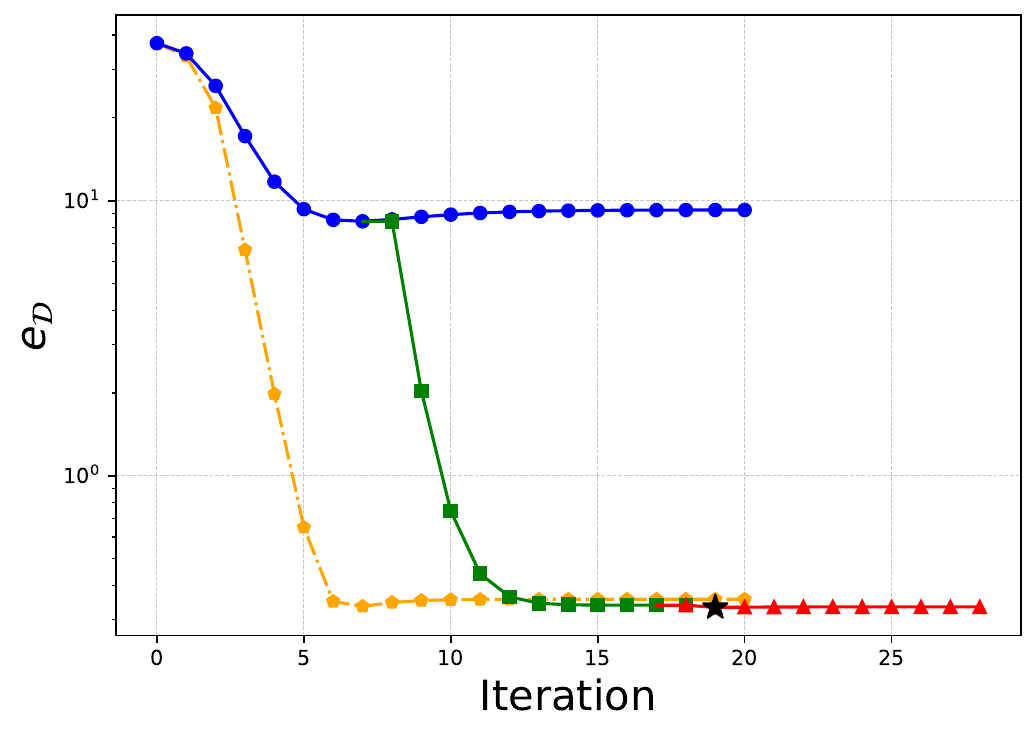}
    \put (35,72) {\scriptsize {\bf fitting error}}
    \end{overpic}
    \begin{overpic}[width = 0.3\textwidth]{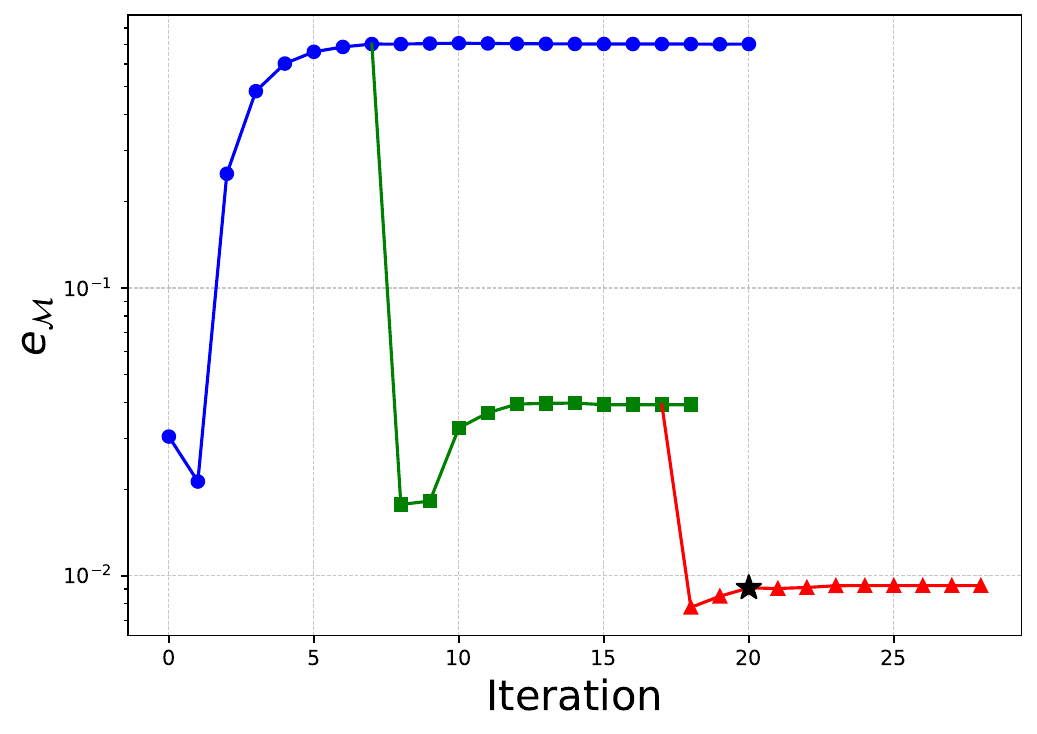}
    \put (35,72) {\scriptsize {\bf model error}}
    \end{overpic}
    \begin{overpic}[width = 0.3\textwidth]{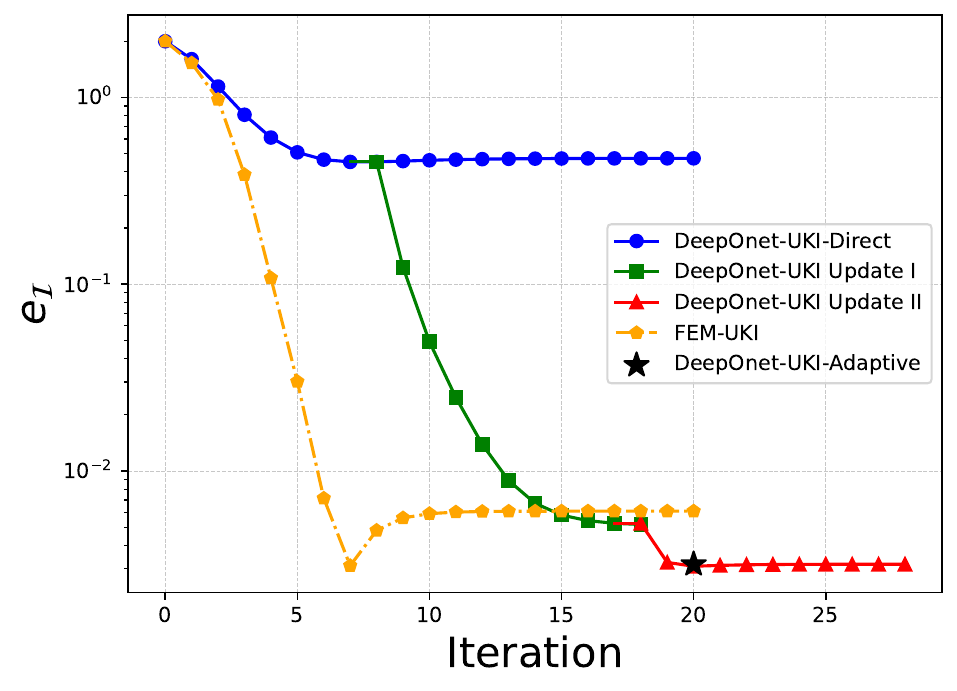}
    \put (20,72) {\scriptsize {\bf relative inversion error}}
    \end{overpic}
    \caption{The fitting error, model error and relative inversion error from left to right respectively.}
    \label{source}
    \end{figure}
    
    \begin{figure}[htbp]
        \centering 
        \begin{overpic}[width = 0.45\textwidth]{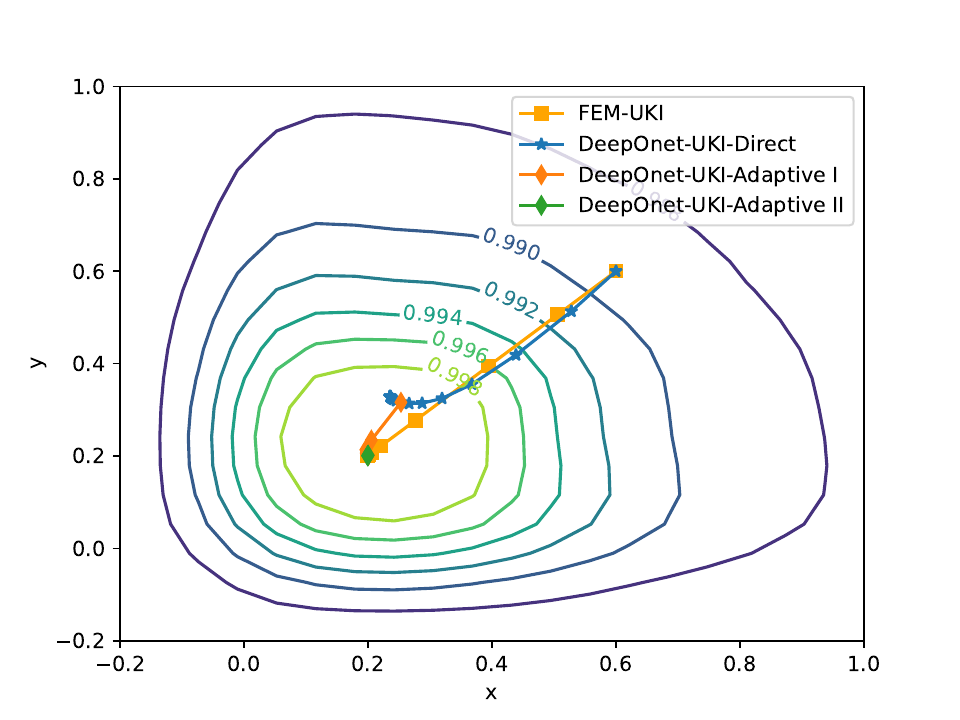}
        \end{overpic}
        \begin{overpic}[width = 0.45\textwidth]{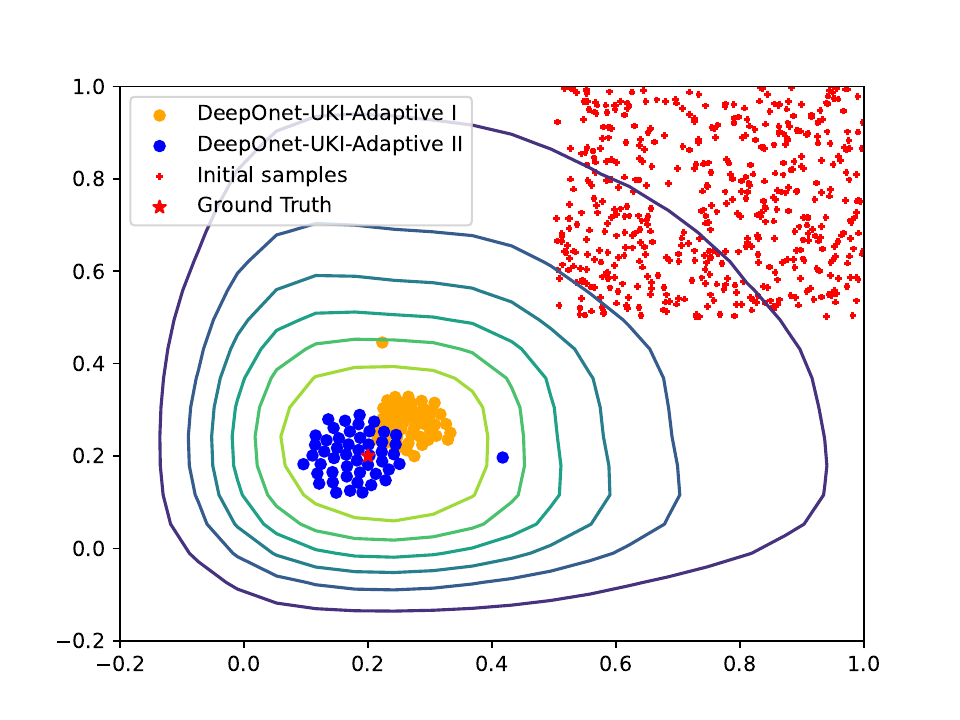}
        \end{overpic}
        \vspace{-0.3cm}
        \caption{The inversion trajectories of the three methods and the samples distribution of DeepOnet-UKI-Adaptive.}
        \label{path}
    \end{figure}

\textbf{Case II:} In this case, the heat source field is considered with the formula $e^{-t}\mb{m}(\mb{x})$ with zero Dirichlet boundary condition and initial condition $\textcolor{black}{u_0(x, y)} = 100\sin(x)\sin(y)$. Conversely, the inverse problem involves using noisy measurements of $u(\mb{x}, 1)$ to determine the true spatial source field $\mb{m}(\mb{x})$. We assume that the Gaussian random field defined in Eq.\eqref{KL} is the prior of $\mb{m}(\mb{x})$.

We assume that the ground truth $\mb{m}_{ref}(\mb{x})$ has an analytical solution in this example to increase the dimension of the problem, i.e.,
 \begin{equation}
    \mb{m}_{ref}(\mb{x}) = \sin(\pi x)\cos(\pi y).
 \end{equation}
 Using this specific solution, we generate the observations $y$ from the final temperature field $u(\mb{x}, 1)$ at 36 equidistant points in $\Omega$. Fig.\ref{heat_observations} displays the corresponding observations and the true spatial field $\mb{m}$.  In the inverse procedure, the KL expansion \eqref{KL} will be employed to approximate the true source field.  Specifically, to accomplish the inversion task, we will truncate the first 128 modes.

 \begin{figure}[t]
\centering 
\begin{overpic}[width=0.45\textwidth]{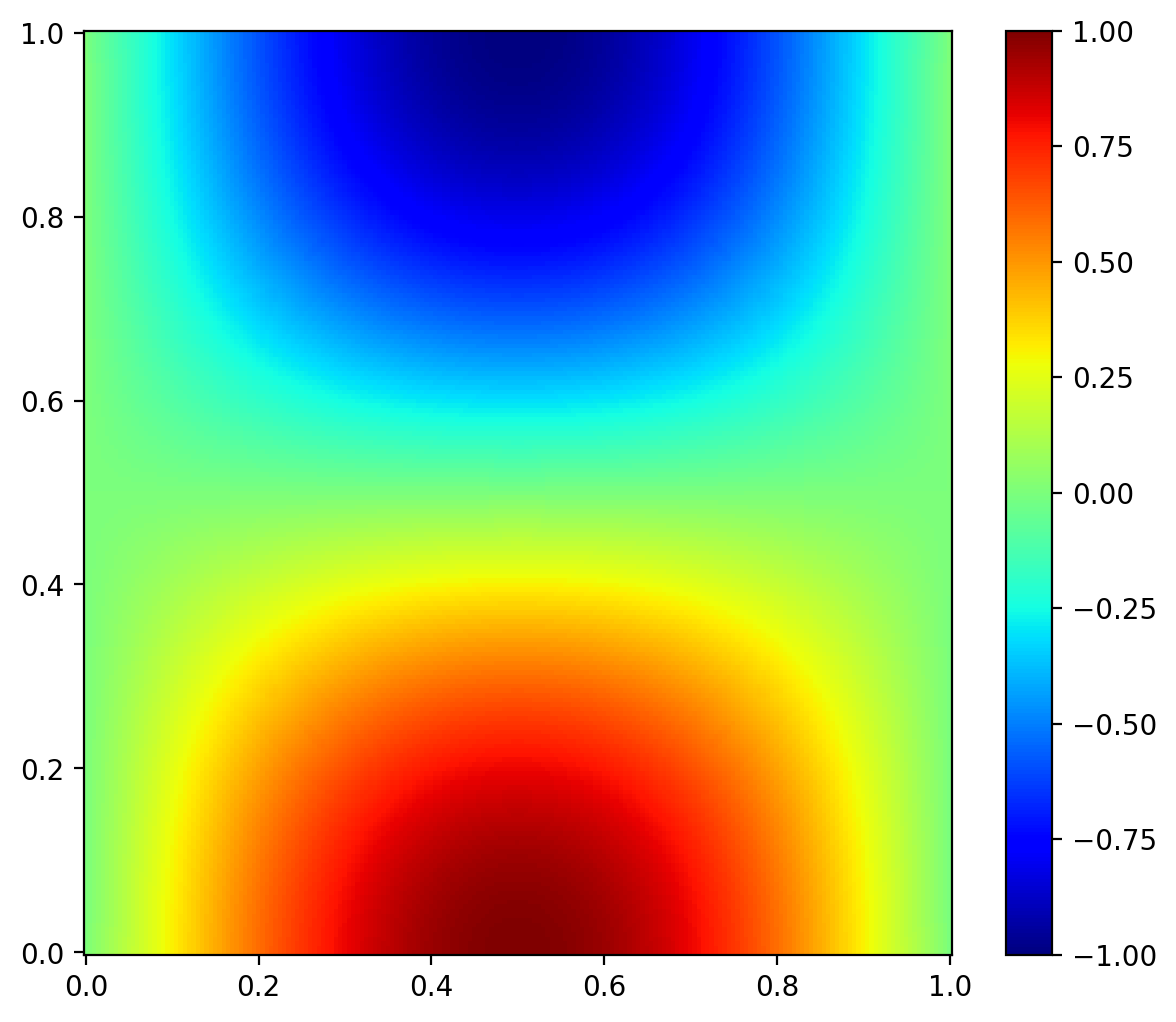}  
\end{overpic}
\begin{overpic}[width=0.47\textwidth]{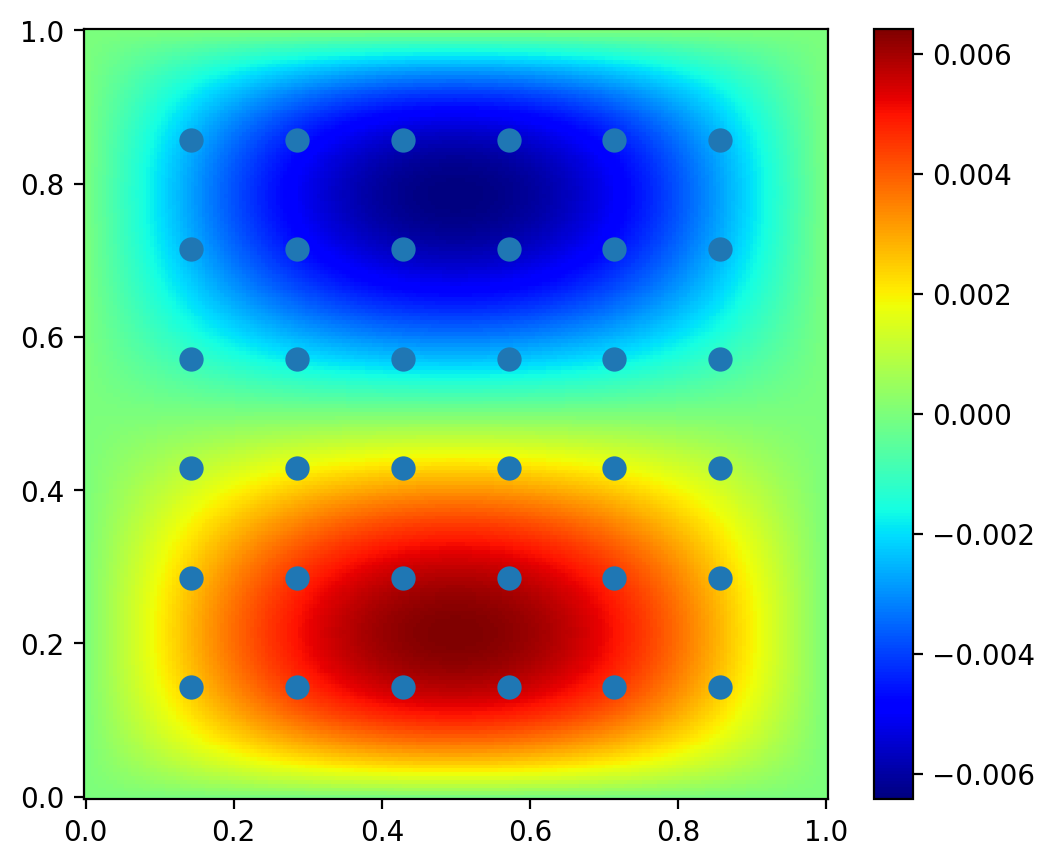}  
\end{overpic}
\caption{Example 2. Left: the true source field. Right: the temperature field $u(\mb{x}, 1)$ and the corresponding $36$ equidistant observations with noise level 0.01.}
\label{heat_observations}  
 \end{figure}
 
To test the effectiveness of our framework, we first run the experiment using the original pre-trained model directly running UKI, i.e. DeepOnet-UKI-Direct. We plot the local  model error and the relative inversion error  in the middle and right displays of Fig.\ref{loss_heat}, respectively. As expected, the local  model error will increase significantly, and the pre-trained model will eventually fail to predict the result. Because of the growing model error, the relative inversion error follows accordingly, increasing dramatically and providing a totally inaccurate final estimate.  Nonetheless, we may refine our model by constructing adaptive samples based on this estimate. The procedures are similar to Example 1. During the exploration stage, we will run UKI for 10 steps and select the {\it anchor point} with the smallest data fitting error as the initial value  for the next run of UKI.  To reduce model error, we will generate adaptive samples near the {\it anchor point} and refine the surrogate.  As shown in the middle of Fig.\ref{loss_heat}, the local model error significantly decreases after refinement. As a result, our method produces relative inversion errors that are significantly reduced after refinement and comparable to those of FEM-UKI.   Finally, after several refinements, the entire procedure is terminated based on the stop criteria, with the black star indicating the final point that we would accept for the inversion process using our DeepOnet-UKI-Adaptive algorithm. To demonstrate the effectiveness of our method, we plot the inversion results in Figs.\ref{kappa_heat} and \ref{forward_heat}. We can see that the final numerical results produced by DeepOnet-UKI-Adaptive and FEM-UKI are very similar and do not differ significantly. 
This suggests that our method can handle OOD data in closed form.

\begin{figure}[t]
    \begin{center}
        \begin{overpic}[width=0.3\textwidth]{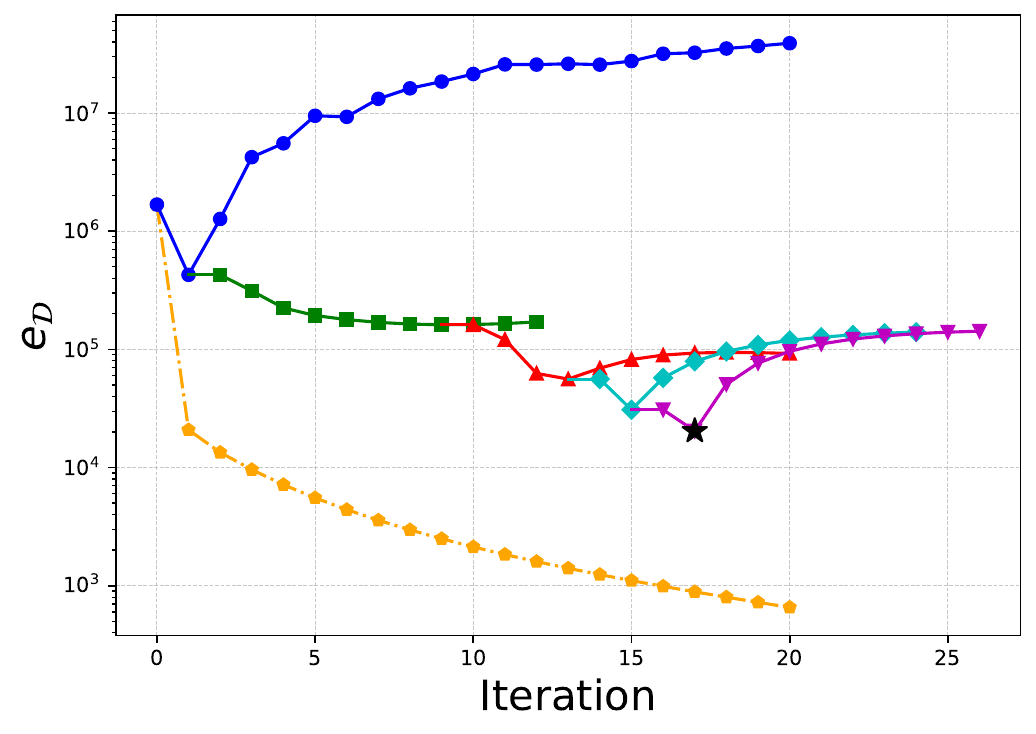}
        \put (35,72) {\scriptsize {\bf fitting error}}
        \end{overpic}
        \begin{overpic}[width=0.3\textwidth]{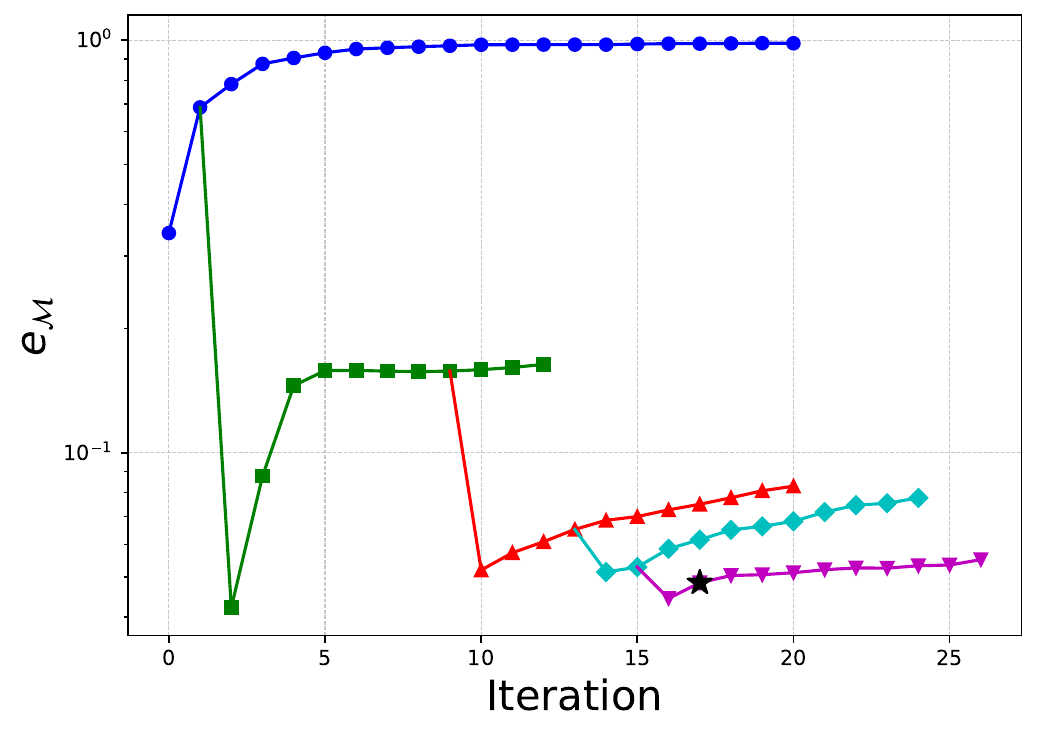}
                   \put (38,72) {\scriptsize {\bf model error}}
        \end{overpic}
        \begin{overpic}[width=0.3\textwidth]{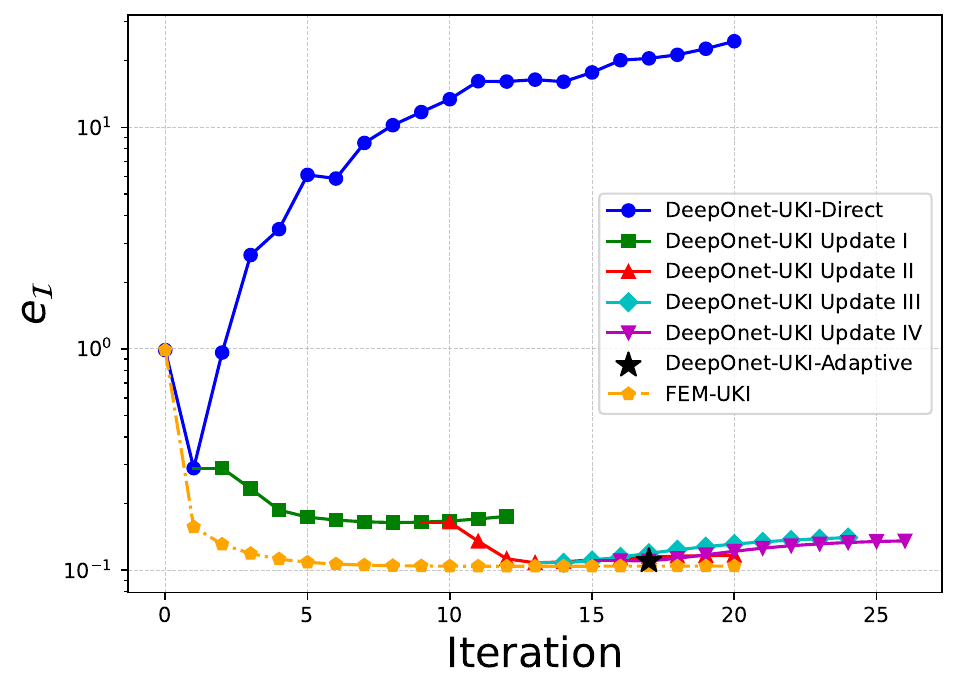}
        \put (22,72) {\scriptsize {\bf relative inversion error}}
        \end{overpic}
    \end{center}
    \caption{The inversion results obtained by the three methods for the heat source case. Left: the data fitting error. Middle: the model error. Right: the relative inversion errors.}
    \label{loss_heat}
\end{figure}

\begin{figure}[t]
    \begin{center}
        \begin{overpic}[width=0.3\textwidth]{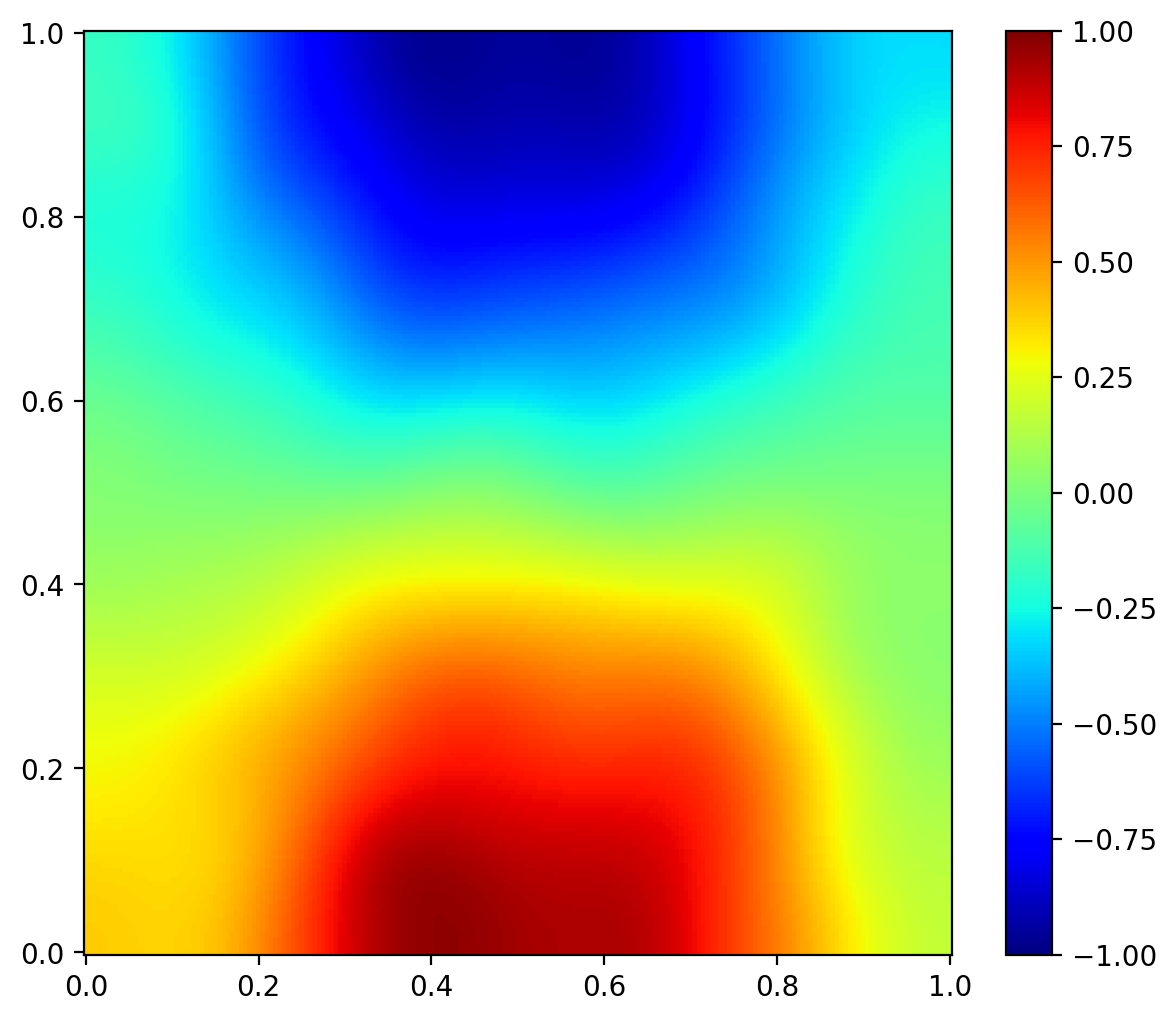}
        \put (30,87) {\footnotesize \bf FEM-UKI}
        \end{overpic}
        \begin{overpic}[width=0.3\textwidth]{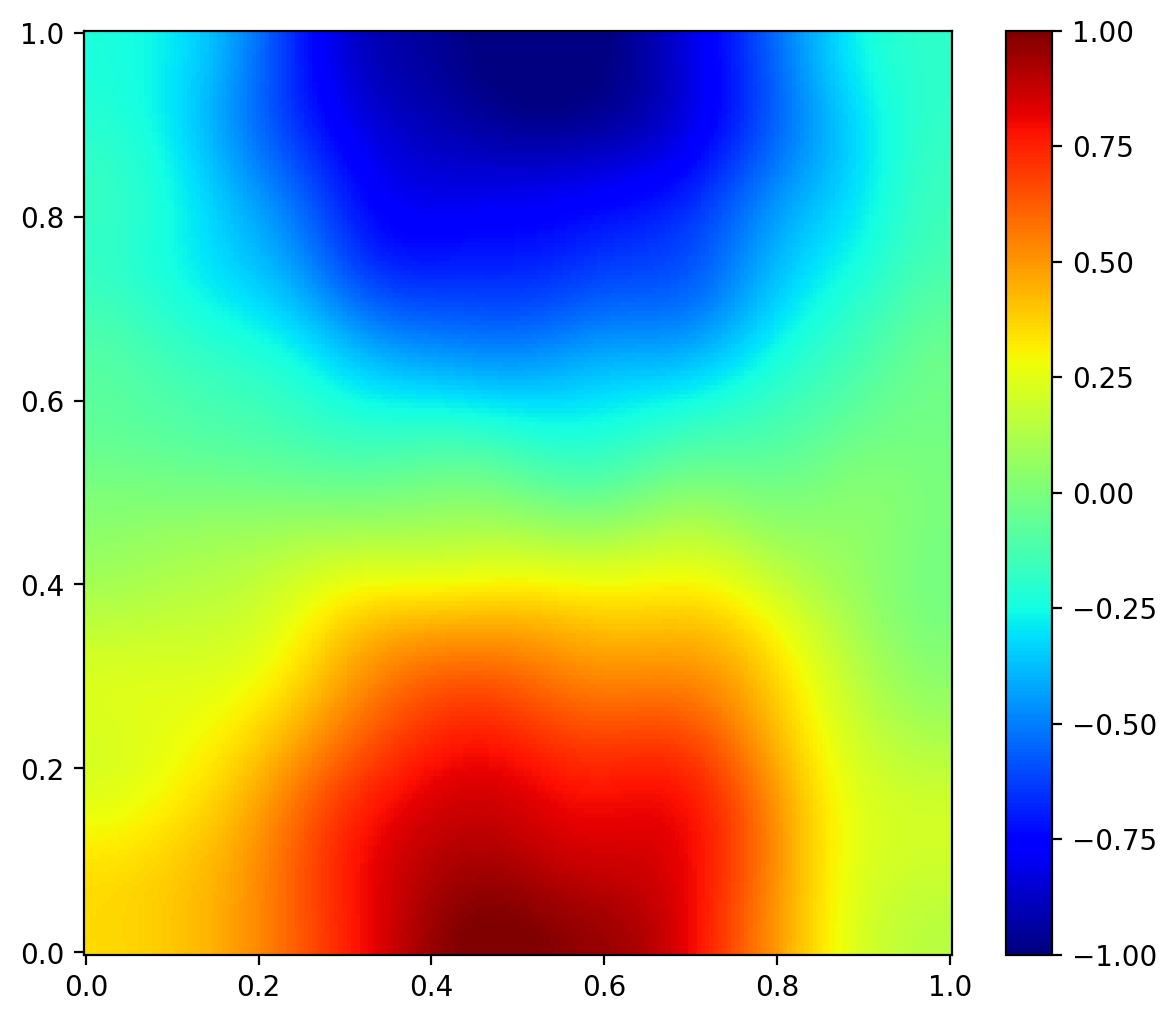}
            \put (8,87) {\footnotesize \bf DeepOnet-UKI-Adaptive}
        \end{overpic}
                \begin{overpic}[width=0.30\textwidth]{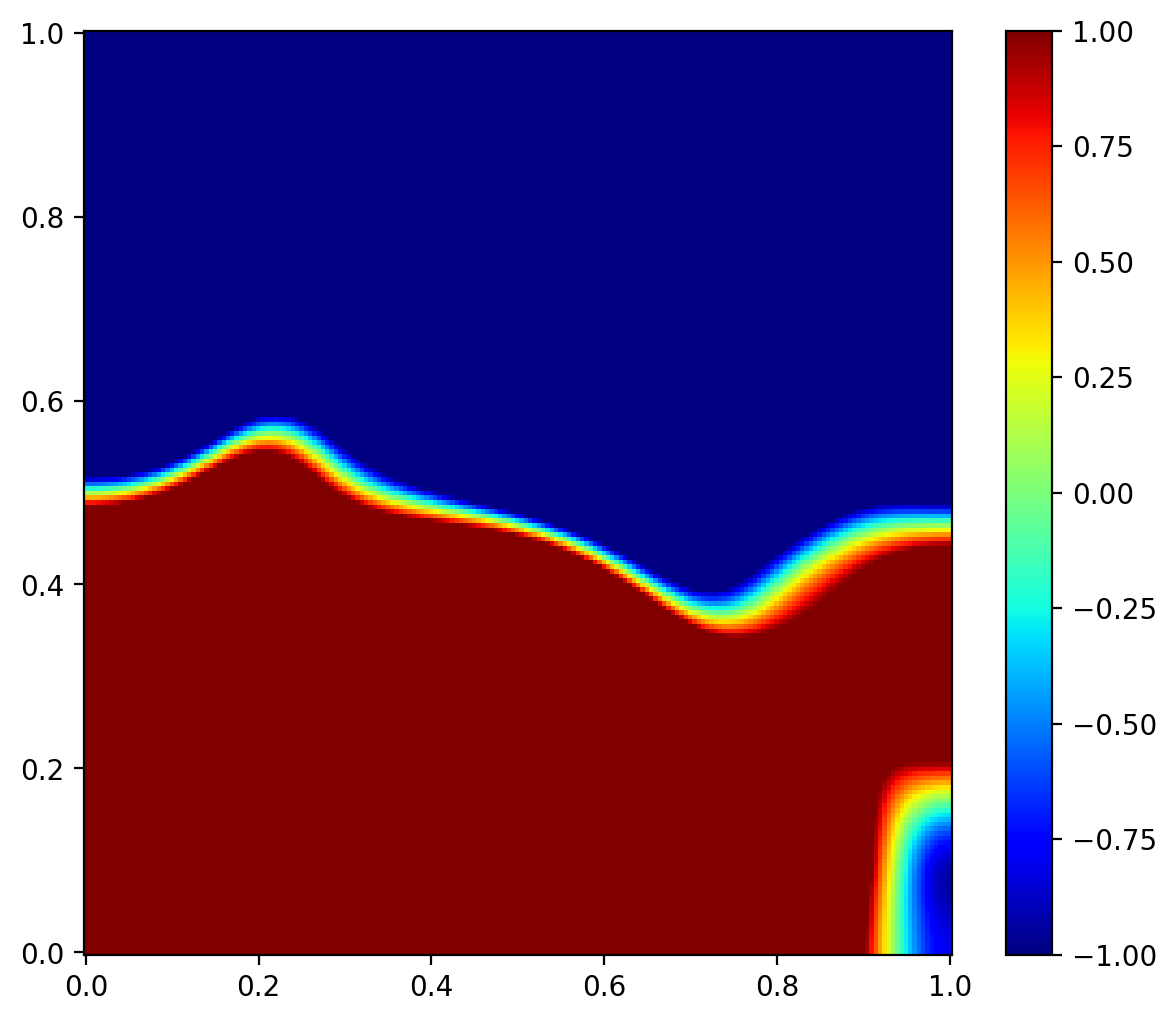}
            \put (10,87) {\footnotesize \bf DeepOnet-UKI-Direct}
        \end{overpic}
    \end{center}
    \begin{center}
        \begin{overpic}[width=0.3\textwidth]{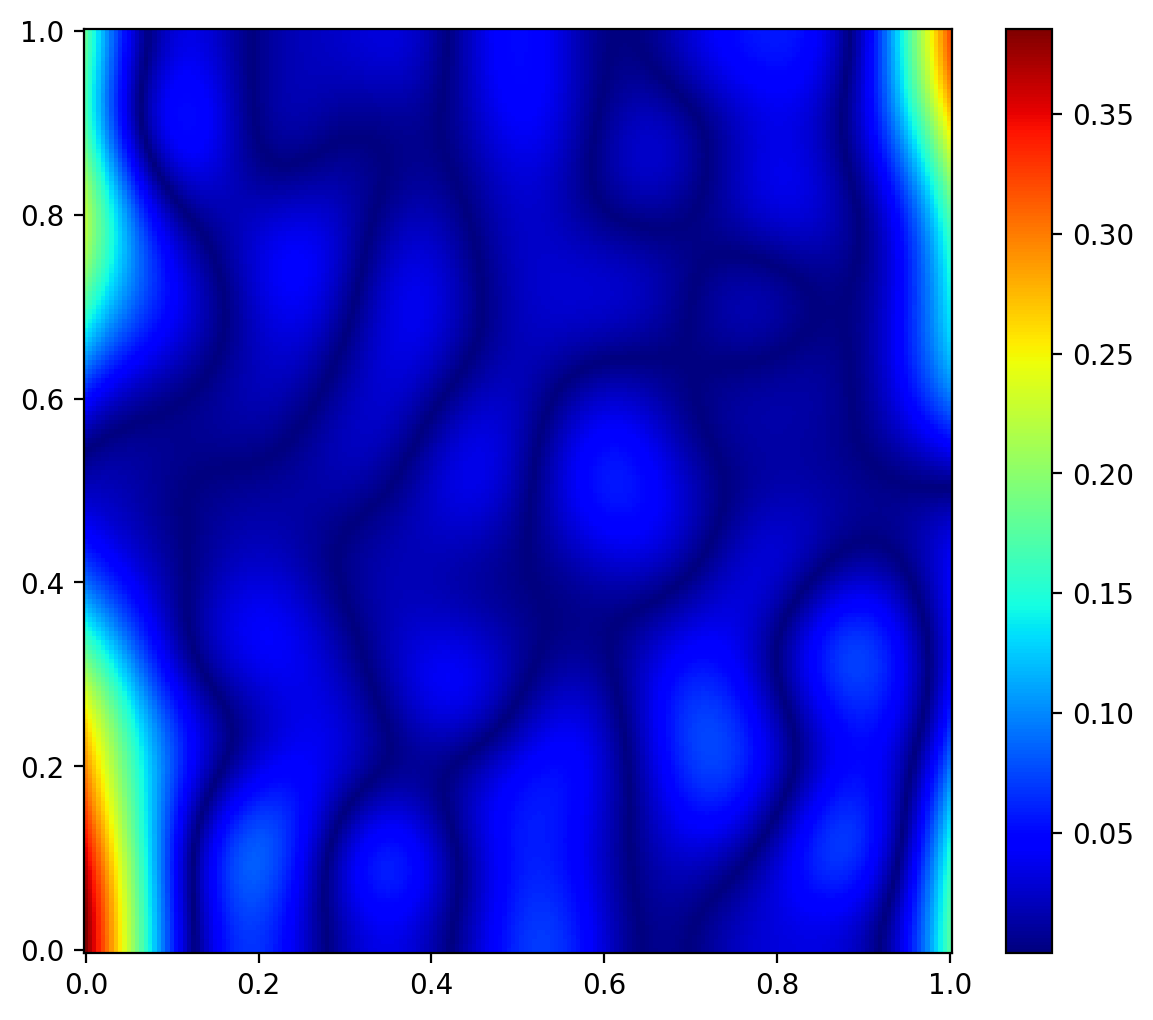}
        \end{overpic}
        \begin{overpic}[width=0.3\textwidth]{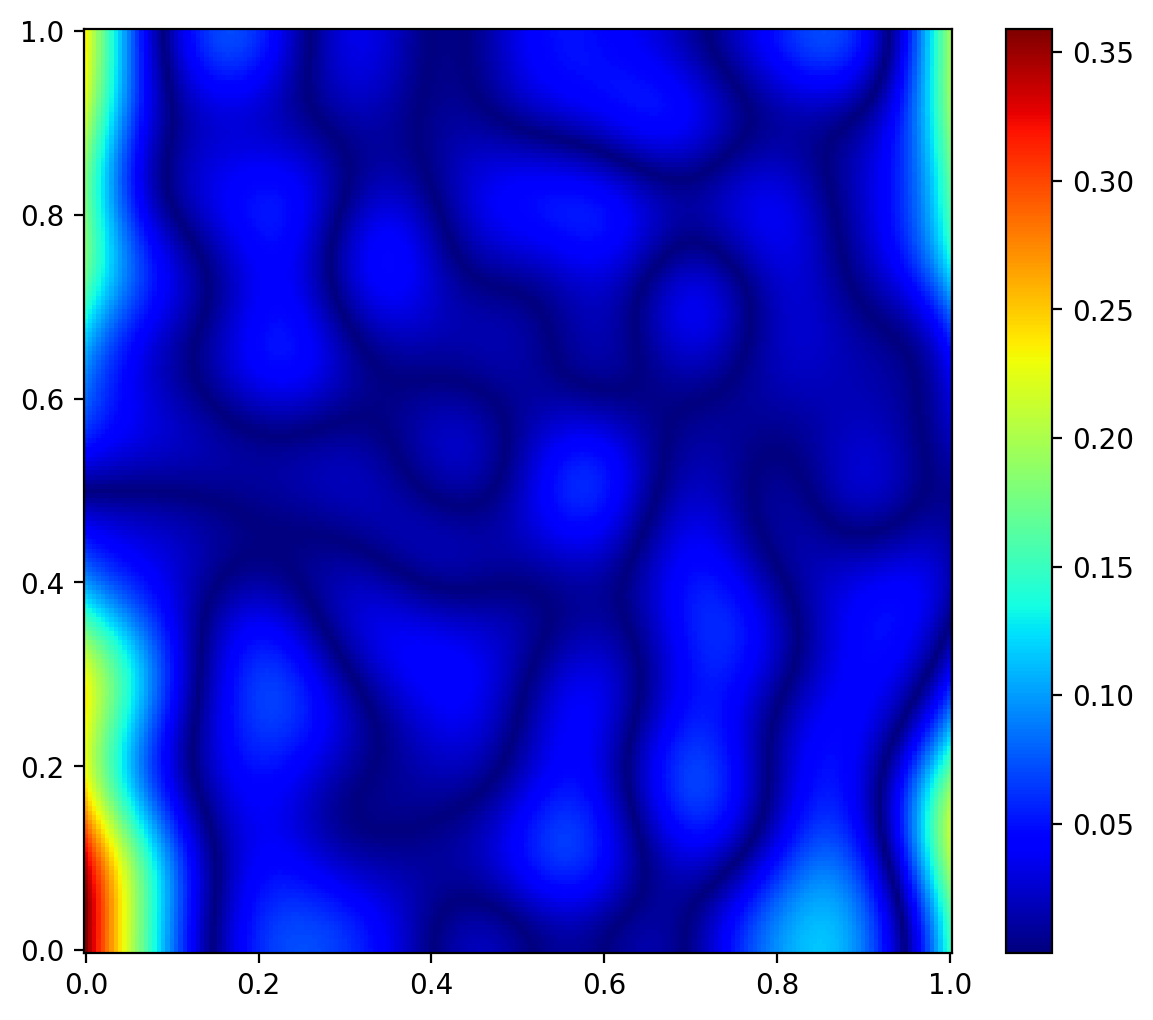}
        \end{overpic}
         \begin{overpic}[width=0.295\textwidth]{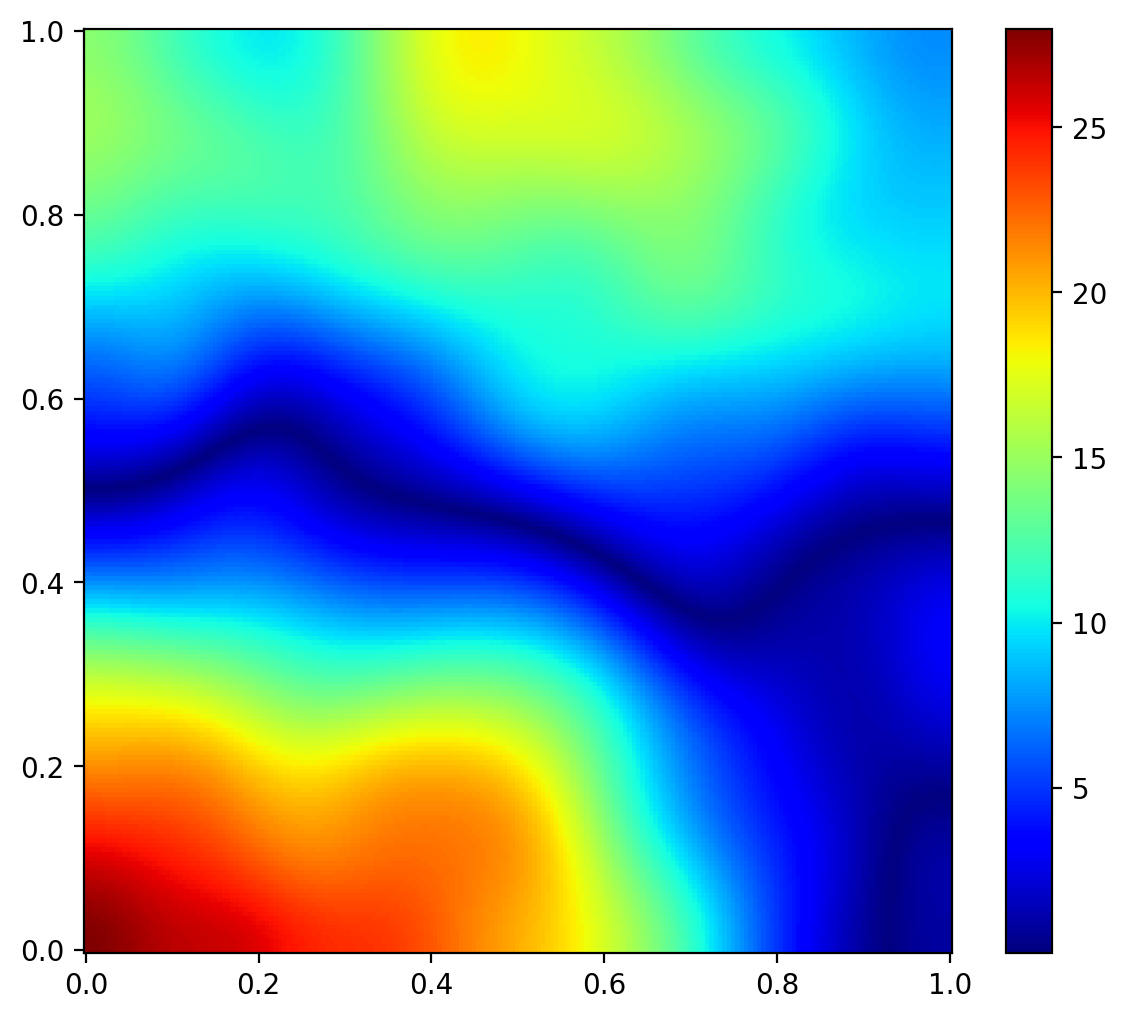}
        \end{overpic}
    \end{center}
    \caption{The inversion results for the heat source case. Above: the estimated source fields obtained by three different methods. Below: the absolute errors with respect to the true one.}
    \label{kappa_heat}
\end{figure}
\begin{figure}[htbp]
    \begin{center}
        \begin{overpic}[width=0.3\textwidth]{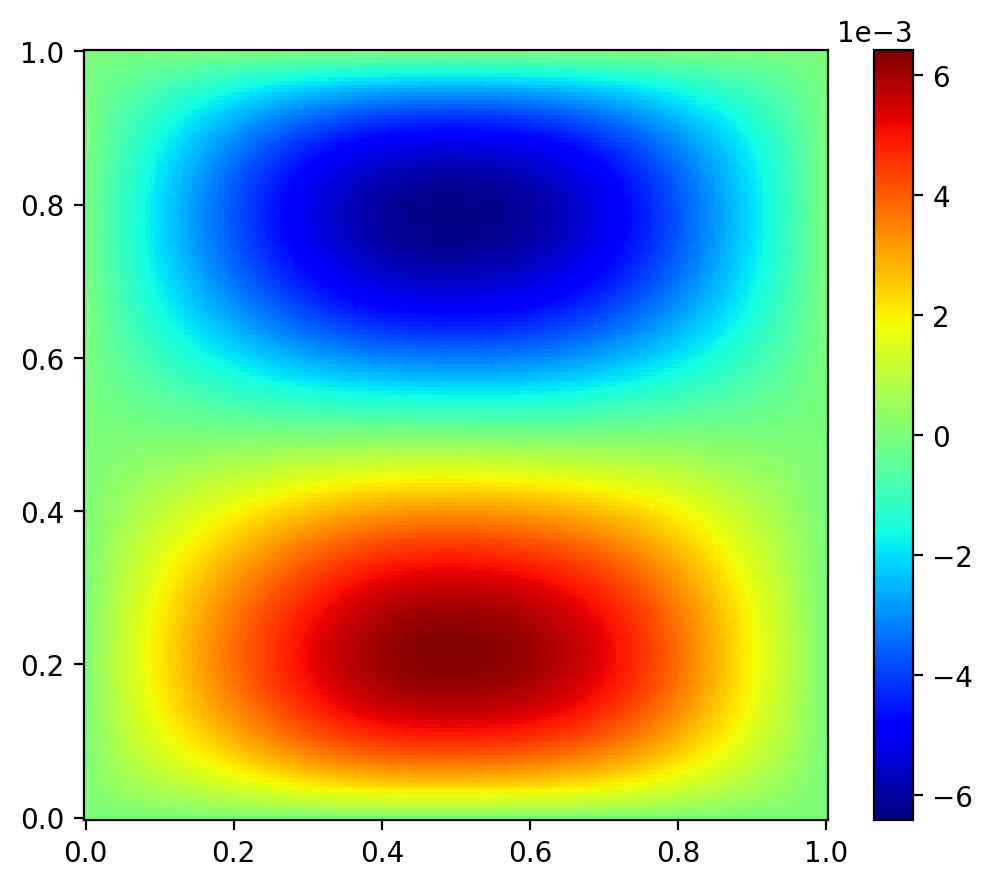}
        \put (30,88) {\footnotesize \bf FEM-UKI}
        \end{overpic}
        \begin{overpic}[width=0.3\textwidth]{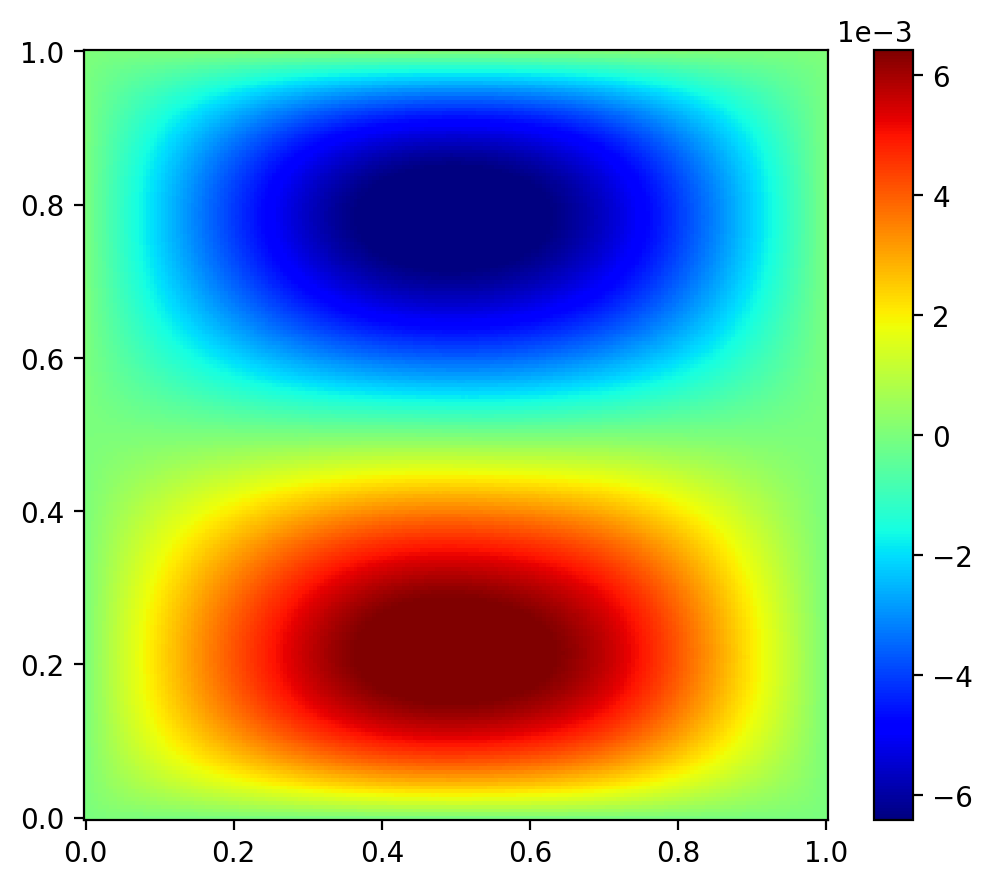}
            \put (8,88) {\footnotesize \bf DeepOnet-UKI-Adaptive}
        \end{overpic}
         \begin{overpic}[width=0.30\textwidth]{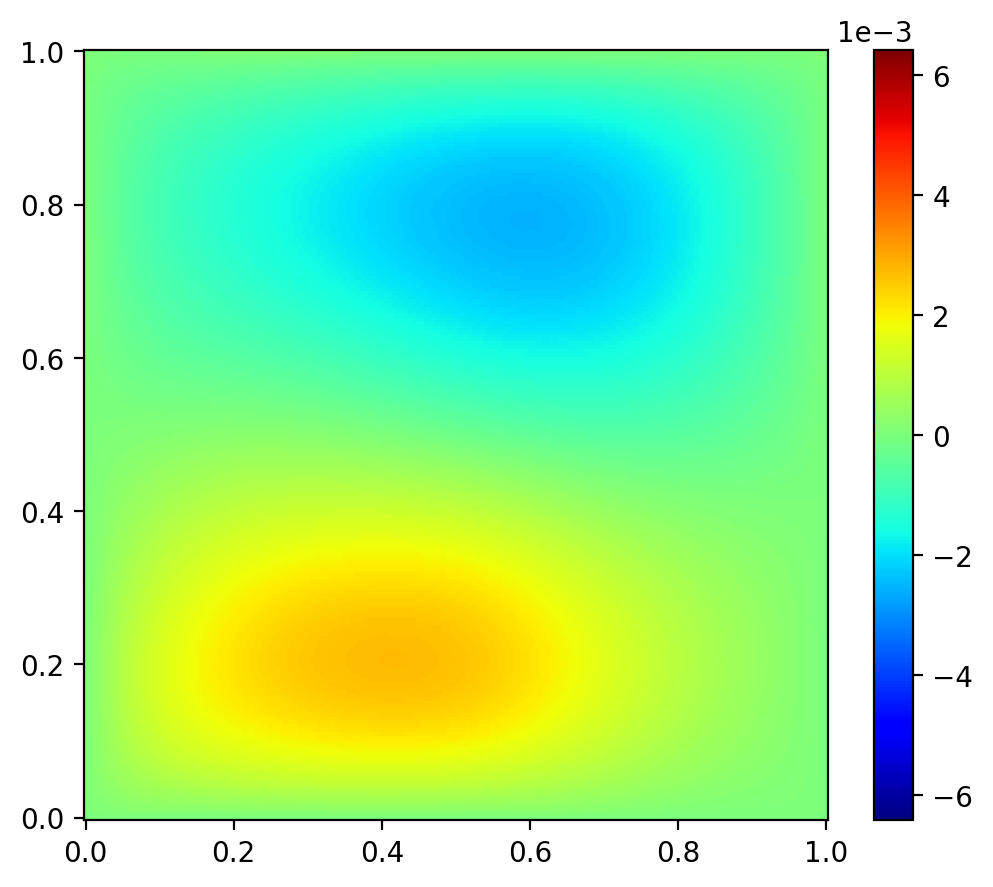}
            \put (10,88) {\footnotesize \bf DeepOnet-UKI-Direct}
        \end{overpic}
    \end{center}
    \begin{center}
        \begin{overpic}[width=0.30\textwidth]{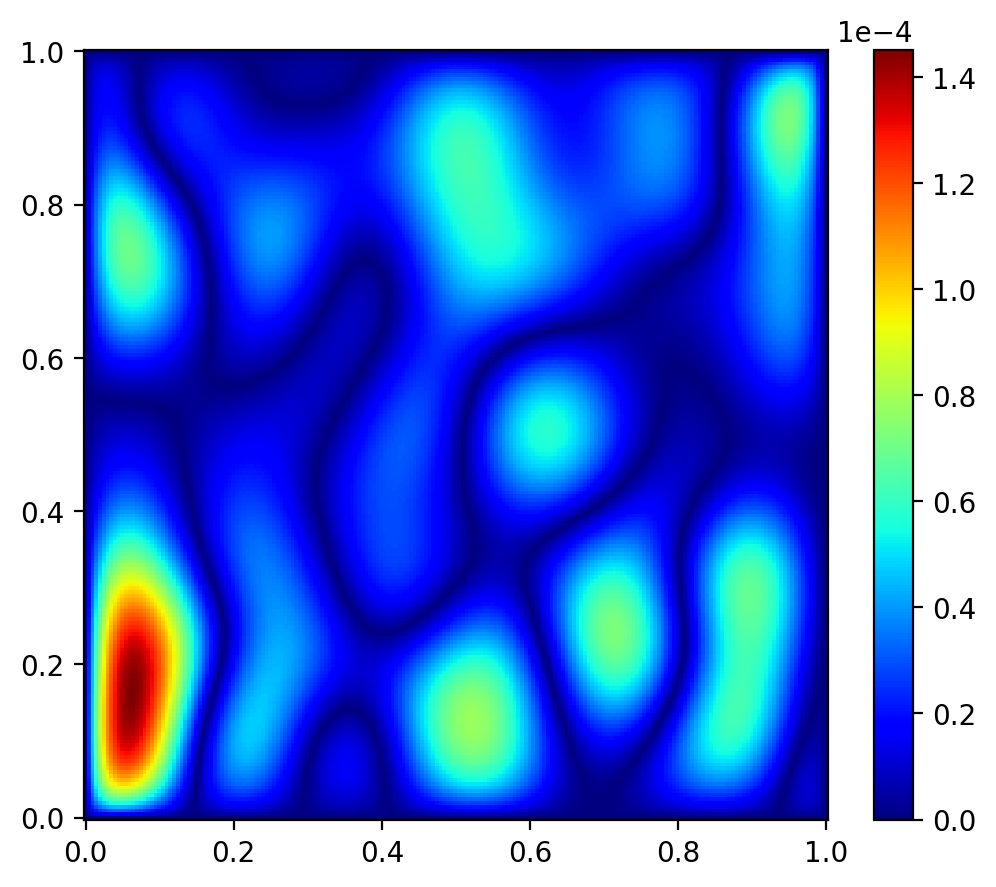}
        \end{overpic}
        \begin{overpic}[width=0.295\textwidth]{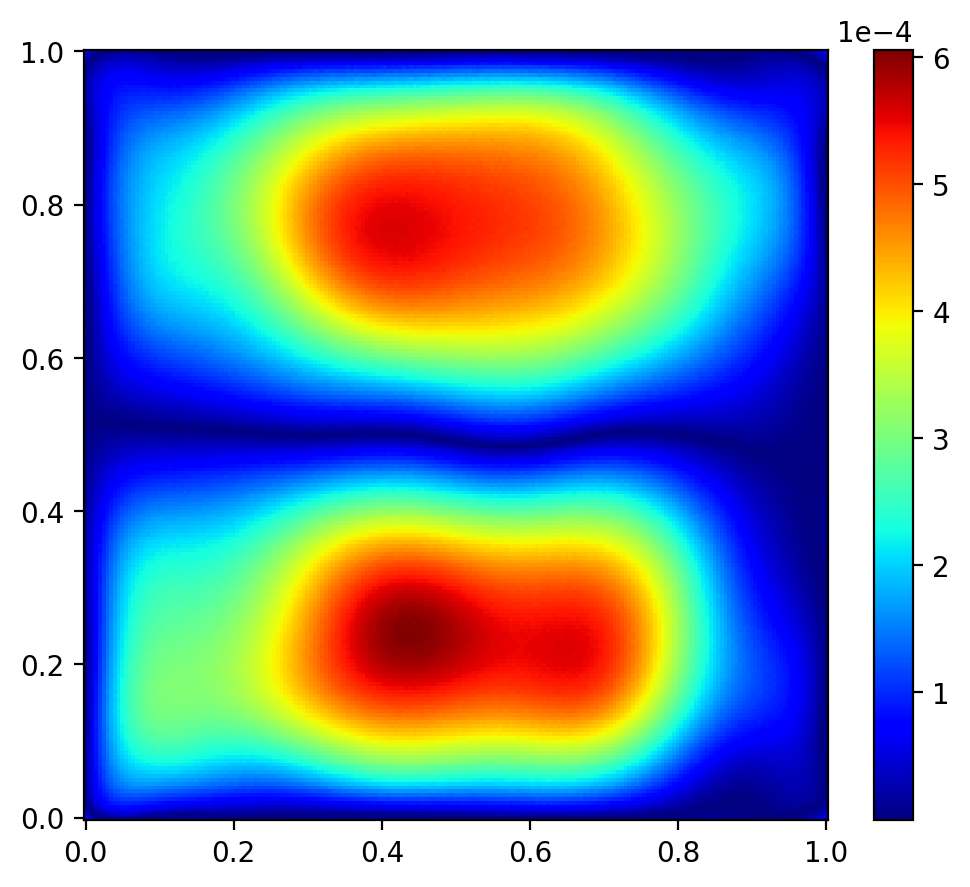}
        \end{overpic}
        \begin{overpic}[width=0.30\textwidth]{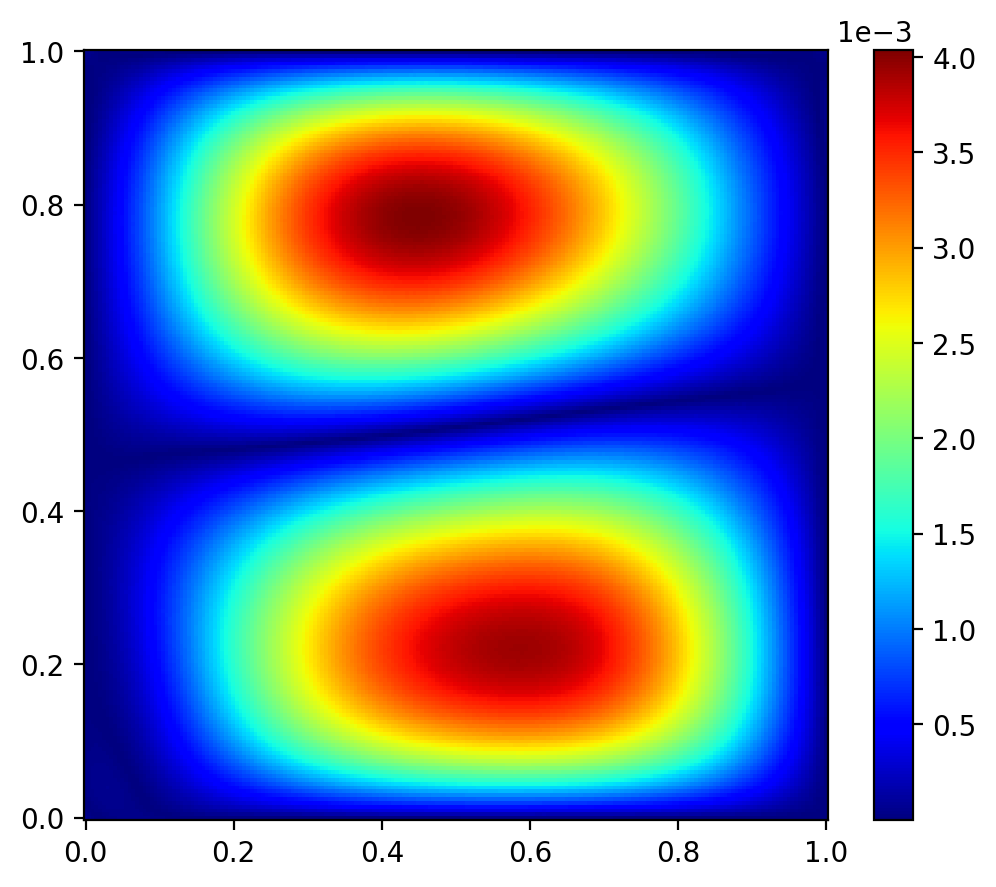}
        \end{overpic}
    \end{center}
    \caption{The inversion results for the heat source case. Above: the approximate states evaluated at the final estimated source fields obtained by different methods. Below: the absolute errors between the approximate state and the true state.}
    \label{forward_heat}
\end{figure}

\begin{figure}[htbp]
    \centering 
    \begin{overpic}[width = 0.325\textwidth]{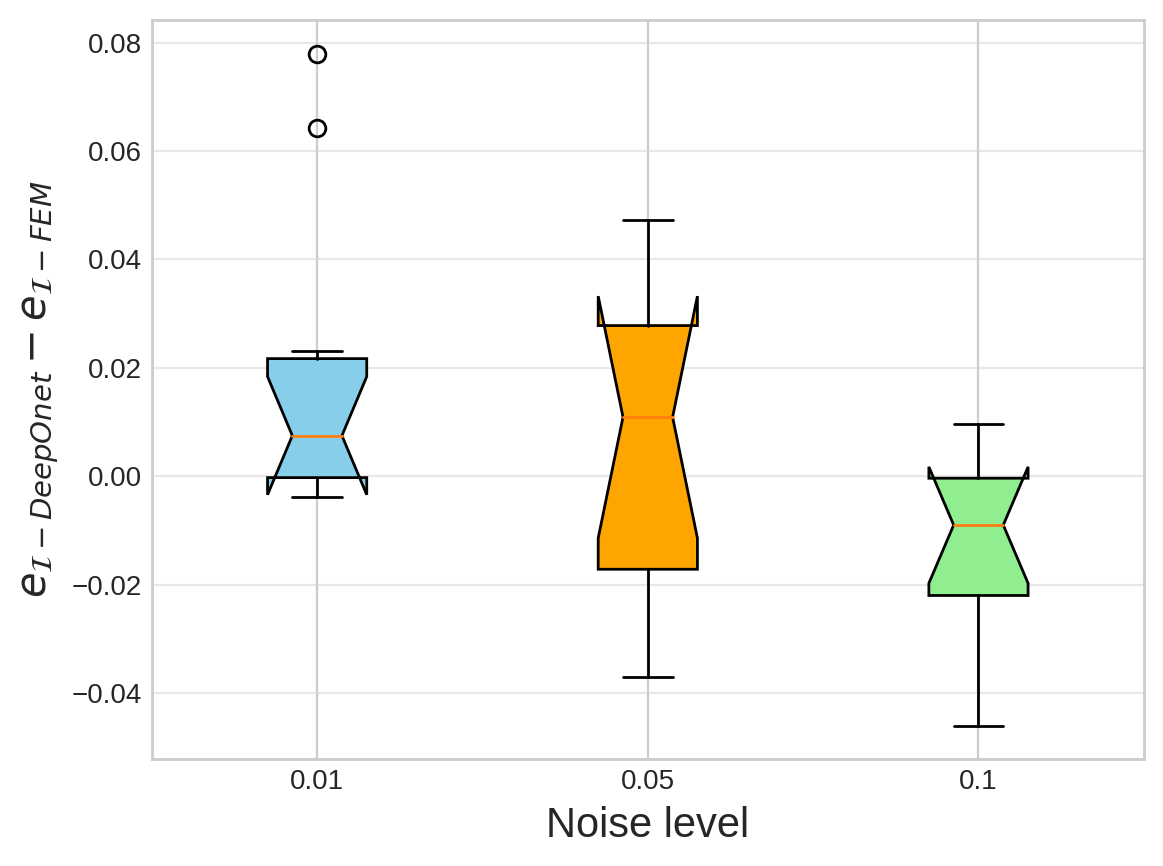}
    \end{overpic}
    \begin{overpic}[width = 0.32\textwidth]{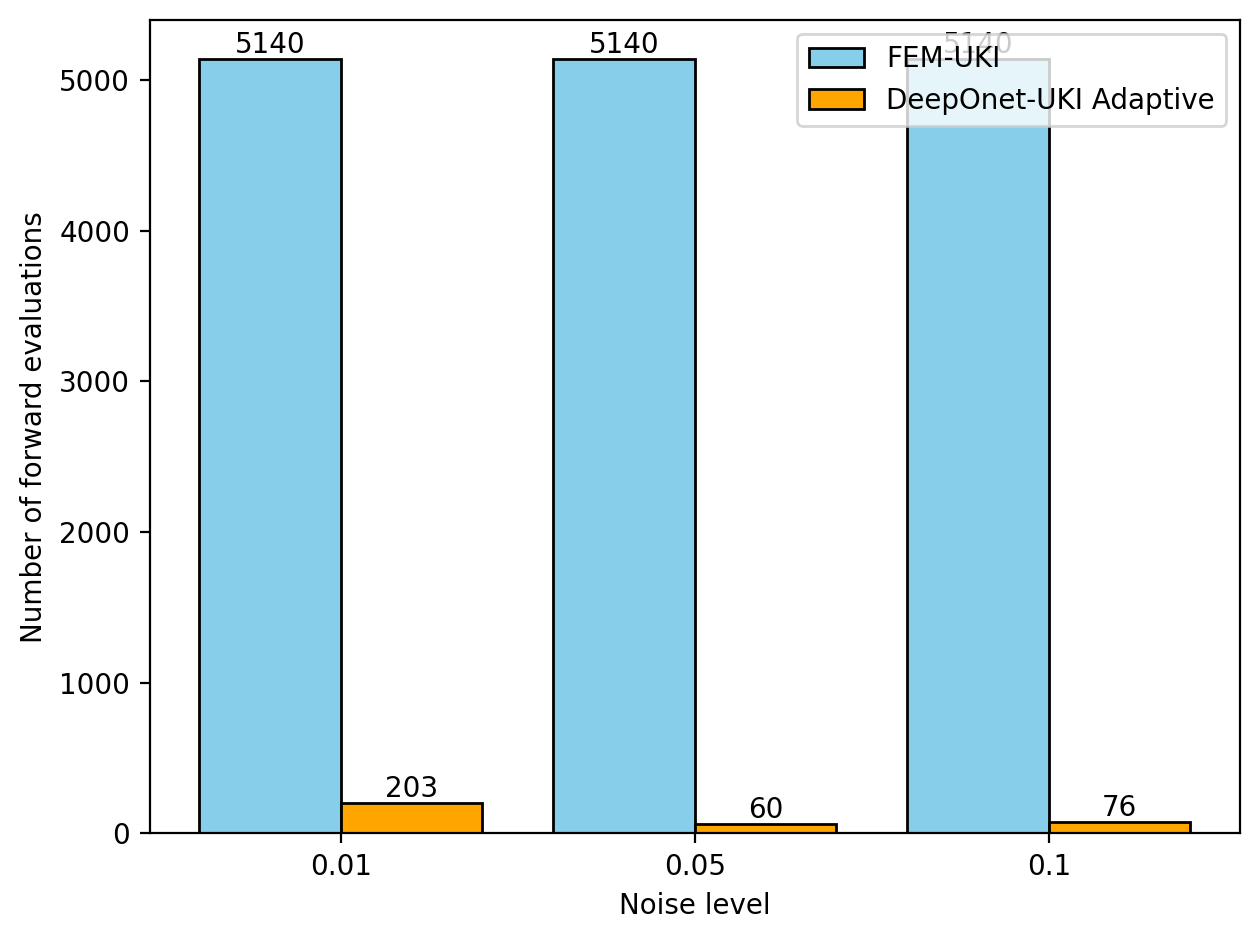}
    \end{overpic}
        \begin{overpic}[width = 0.34\textwidth]{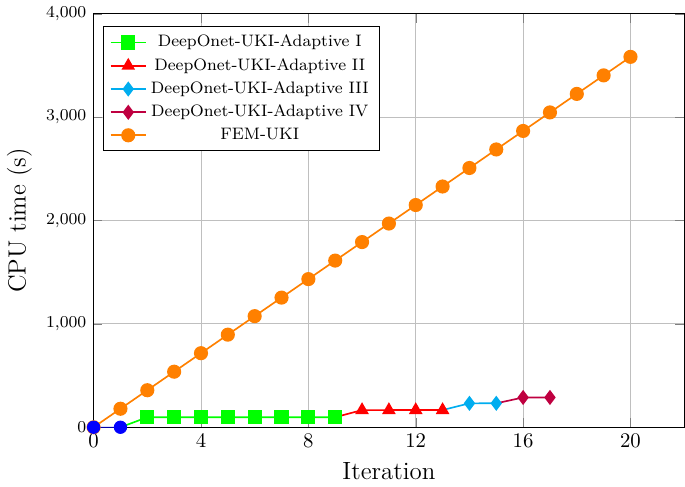}
    \end{overpic}
    \caption{The inversion results for the heat source case. Left: the error box plot of $e_{\mathcal{I}-DeepOnet} - e_{\mathcal{I}-FEM}$. Middle: the mean total number of forward evaluations. Right: the average computational time. }
    \label{err_comparison_heat}
\end{figure}


In order to provide additional evidence of the efficacy of our approach, we figure out to perform the experiment for UKI at varying noise levels. In addition, we will repeat the experiment ten times with different initial values for each noise level. Following that, we will compare the number of forward evaluations and the relative inversion errors for every approach. We plot the difference of the relative inversion errors $e_{\mathcal{I}-DeepOnet} - e_{\mathcal{I}-FEM}$ in the left display of Fig.\ref{err_comparison_heat}. It is clear that when noise levels increase, DeepOnet-UKI-Adaptive performs often better than FEM-UKI. This implies that our method can achieve higher accuracy than traditional solvers.   For the reasons mentioned below, the computational cost of the new method can also be extremely small.  First of all, it is much faster to fine-tune the original pre-trained surrogate model than it is to solve PDEs. Specially, we only need a maximum of 50  online forward evaluations in this example to retrain the network, which drastically lowers computational costs. We are able to clearly see that DeepOnet-UKI-Adaptive has a substantially smaller total number of forward evaluations than FEM-UKI, as the middle display of Fig.\ref{err_comparison_heat} illustrates. Secondly, the entire process is automatically stopped by applying the stop criterion.  We can start with the initial model that has been trained offline and fine-tune it multiple times for a given inversion task. As a result, our framework can achieve an accuracy level comparable to traditional FEM solvers, but at a significant reduction in computational cost for such problems. The CPU computation times  are plotted in  the right display of Fig.\ref{err_comparison_heat}.  It is clear that our adaptive framework greatly accelerates the inversion process, which is nearly 10 times faster than FEM-UKI with nearly same accuracy.

\subsection{Example 3: The reaction diffusion problem}
 \label{diffusion_reaaction}
 Here we consider the forward model as a parabolic PDE defined as 
 \begin{equation}
    \label{reaction_equation}
    \begin{aligned}
        u_t(\mb{x},t)-\kappa \Delta u(\mb{x},t)+\mathbf{v}(\mb{x}) \cdot \nabla u(\mb{x},t) & =0 & & \text { in } \Omega \times(0, 1), \\
        u(\cdot, 0) & =\mb{m} & & \text { in } \Omega, \\
        \kappa \nabla u \cdot \mathbf{n} & =0 & & \text { on } \partial \Omega \times(0, 1),
        \end{aligned}
 \end{equation}
 where $\kappa = \frac{1}{30}$ is the diffusion coefficient, and $\mb{v} := (\sin(\pi x)\cos(\pi y)$, $-\cos(\pi x)\sin(\pi y))^{T}$ is the velocity field. The forward problem is to find the concentration field $u(\mb{x}, t)$ defined by the initial field $\mb{m}(\mb{x})$. The inverse problem here is to find the true initial field $\mb{m}$ using noisy measurements of $u(\mb{x}, 1)$.  The forward problem is  discretized using FEM method on a $70\times 70$ grid, and  the resulting system of ordinary differential equations is integrated over time using a Crank-Nicolson scheme with uniform time step $\Delta t = 0.02$.

We only take into account the OOD data as Example 1 for the inverse problem.  In other words, we will attempt to inverse the first 128 KL modes using the ground truth $\mb{m}_{ref}(\mb{x})$, which is defined by \eqref{KL} with $\textcolor{black}{\zeta_{k}}\sim \mathcal{U}[-20,20],k=1,\cdots, 256$.  The exact solution and the corresponding  synthetic data  are displayed in Fig.\ref{observations_diffusion}.
 
 \begin{figure}[t]
\centering 
\begin{overpic}[width=0.39\textwidth]{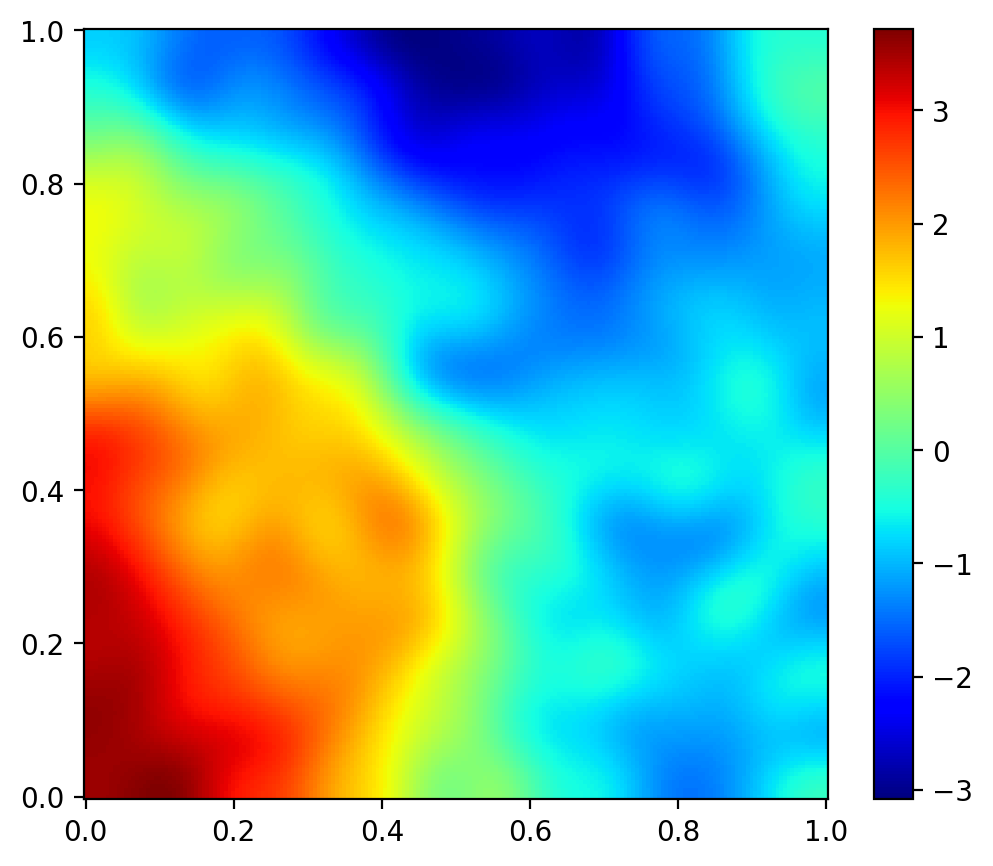}
\end{overpic}
\begin{overpic}[width=0.4\textwidth]{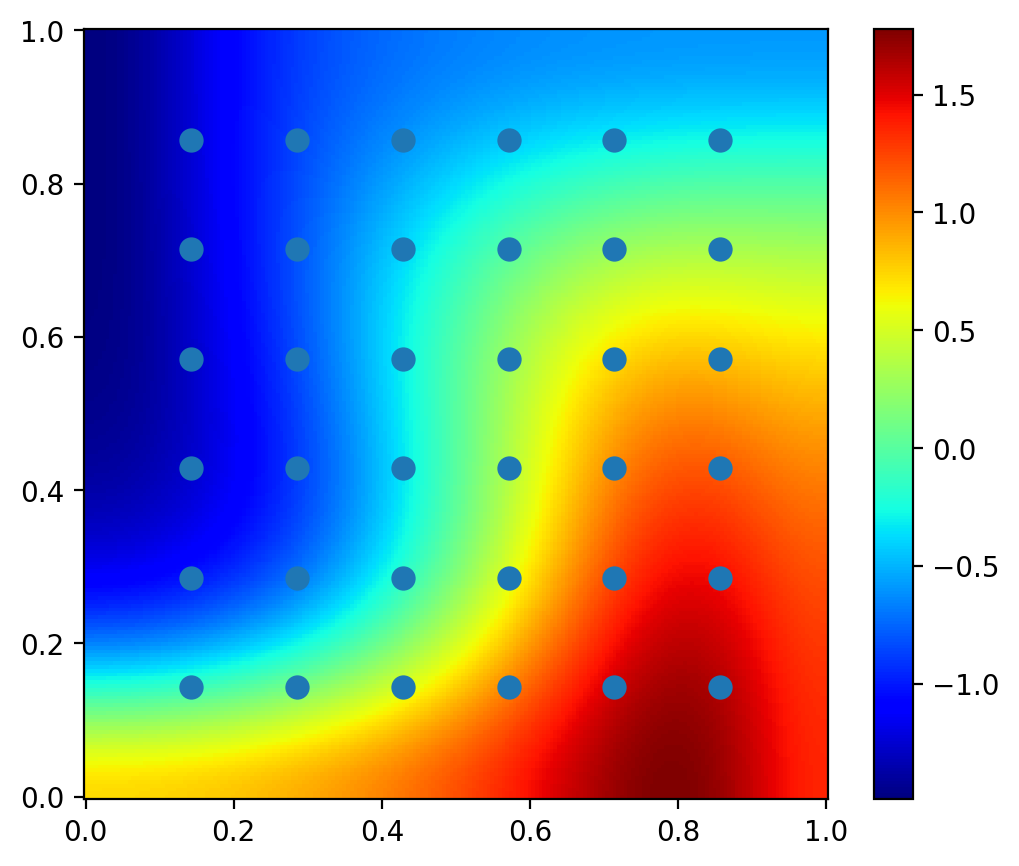}  
\end{overpic}
\caption{The reaction diffusion case. Left: the true initial field $m_{ref}(\mb{x})$. Right: the pressure field $u(\mb{x})$ and the corresponding $36$ equidistant observations with noise level 0.01.}
\label{observations_diffusion}  
 \end{figure}
 
The numerical results are shown in Fig.\ref{loss_diffusion}.  We can clearly see that refinement reduces the local model error significantly, and thus the inversion error will continue to decrease. This implies that our surrogate model can maintain its local accuracy during the inversion process by focusing on the region with the highest posterior probability. Finally, after six iterations of the initial model, the retraining was terminated using the stop criteria. Furthermore, as shown in the right display of Fig.\ref{loss_diffusion}, DeepOnet-UKI-Adaptive can achieve nearly the same order of accuracy as FEM-UKI.  This can be further confirmed by examining  at Figs.\ref{kappa_diffusion} and \ref{forward_diffusion}, which plot the final estimated initial fields and estimated states obtained by different methods.   In this case, the CPU time of evaluating the conventional FEM-UKI is more than 5149s. In contrast, for the  DeepOnet-UKI-Adaptive approach, the online CPU times is only 530s, meaning that the adaptive approach can provide accurate results, yet with less computational time.

\begin{figure}[t]
    \begin{center}
        \begin{overpic}[width=0.3\textwidth]{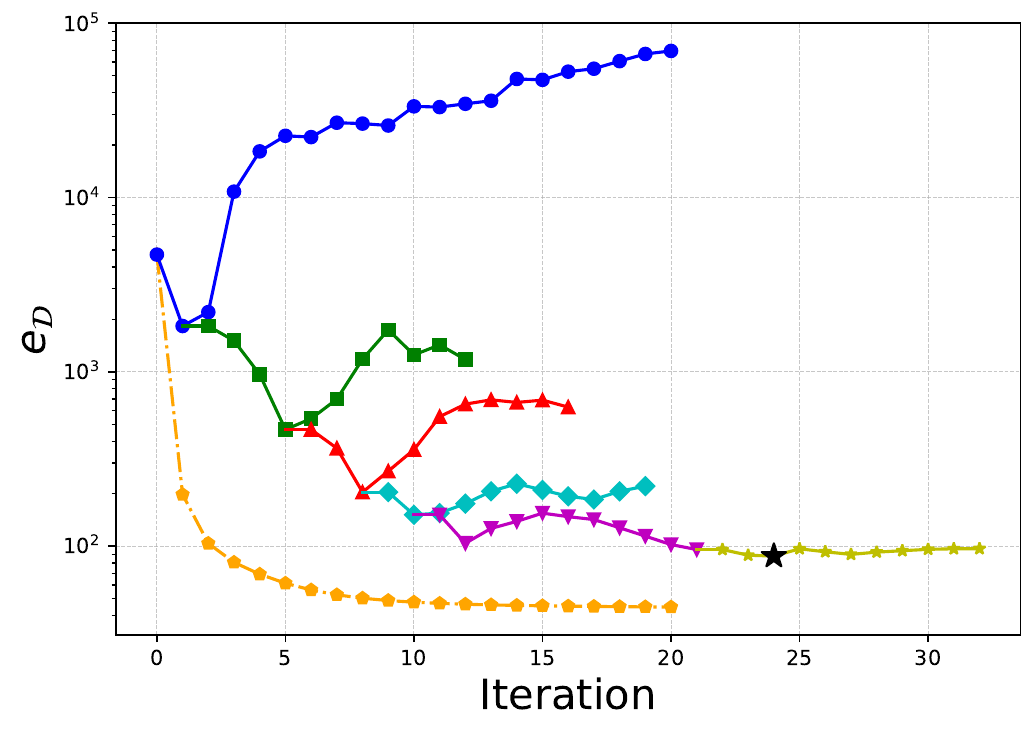}
         \put (35,72) {\scriptsize {\bf fitting error}}
        \end{overpic}
        \begin{overpic}[width=0.3\textwidth]{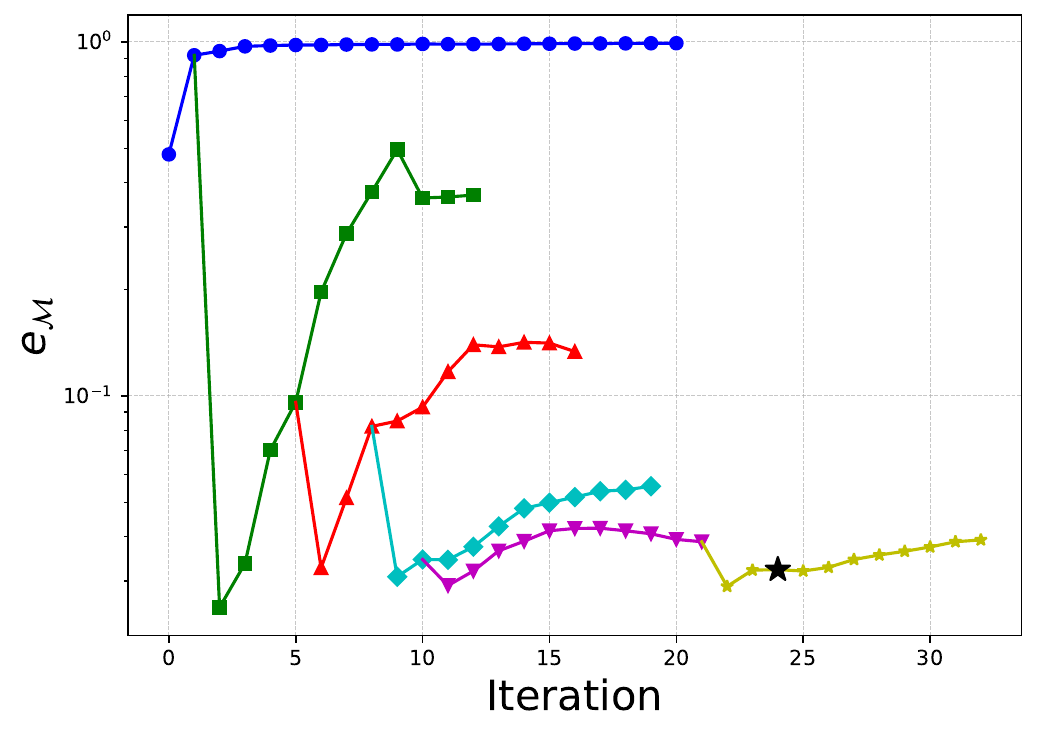}
         \put (38,72) {\scriptsize {\bf model error}}
        \end{overpic}
        \begin{overpic}[width=0.29\textwidth]{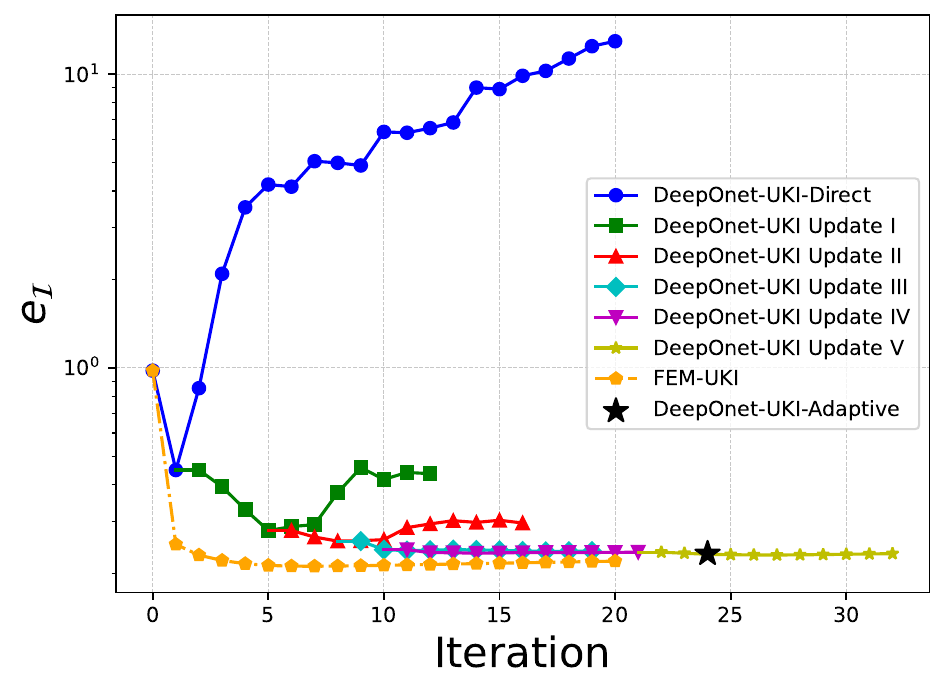}
                \put (22,74) {\scriptsize {\bf relative inversion error}}
        \end{overpic}
    \end{center}
    \caption{The inversion results obtained by three methods for the reaction diffusion case. Left: the data fitting error. Middle: the model error. Right: the relative inversion errors. }
    \label{loss_diffusion}
\end{figure}

\begin{figure}[t]
    \begin{center}
        \begin{overpic}[width=0.3\textwidth]{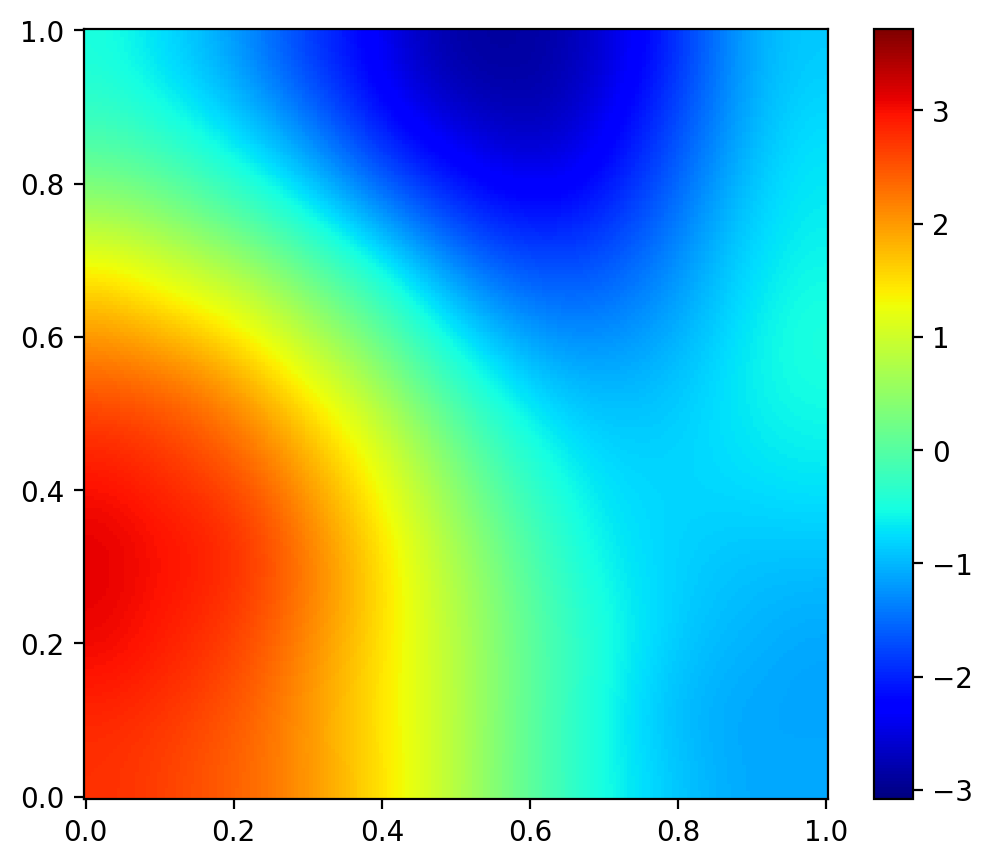}
        \put (30,87) {\footnotesize \bf FEM-UKI}
        \end{overpic}
        \begin{overpic}[width=0.3\textwidth]{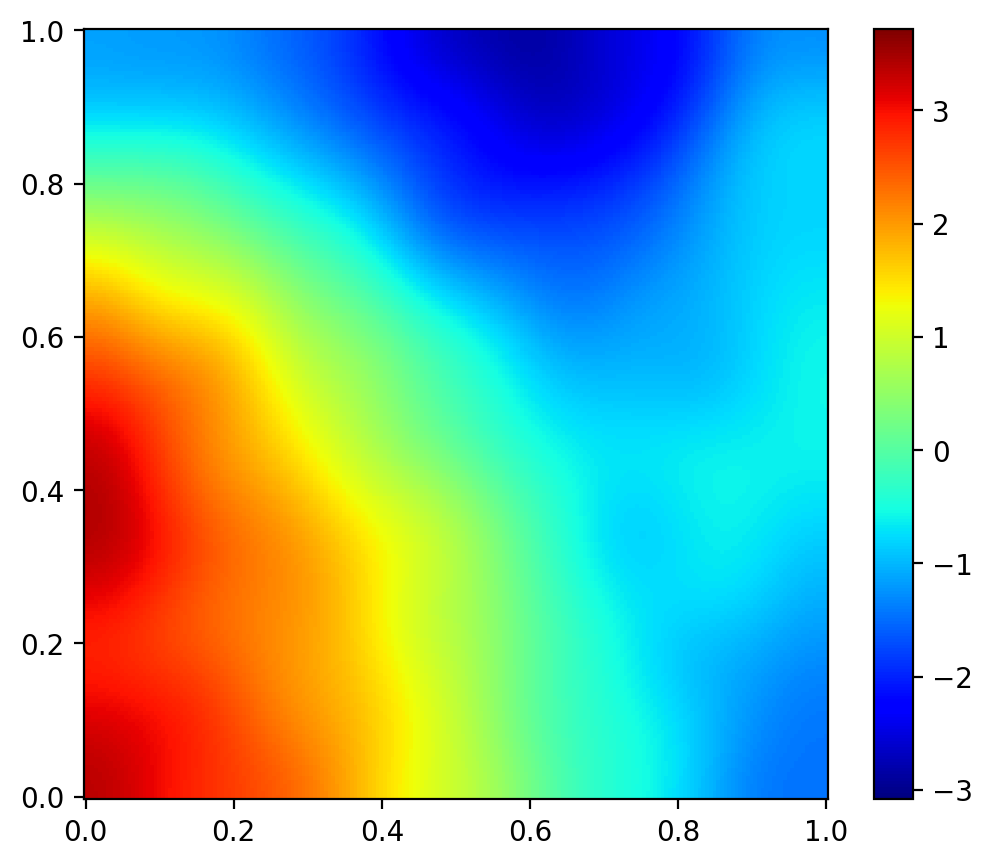}
            \put (5,87) {\footnotesize \bf DeepOnet-UKI-Adaptive}
        \end{overpic}
      \begin{overpic}[width=0.30\textwidth]{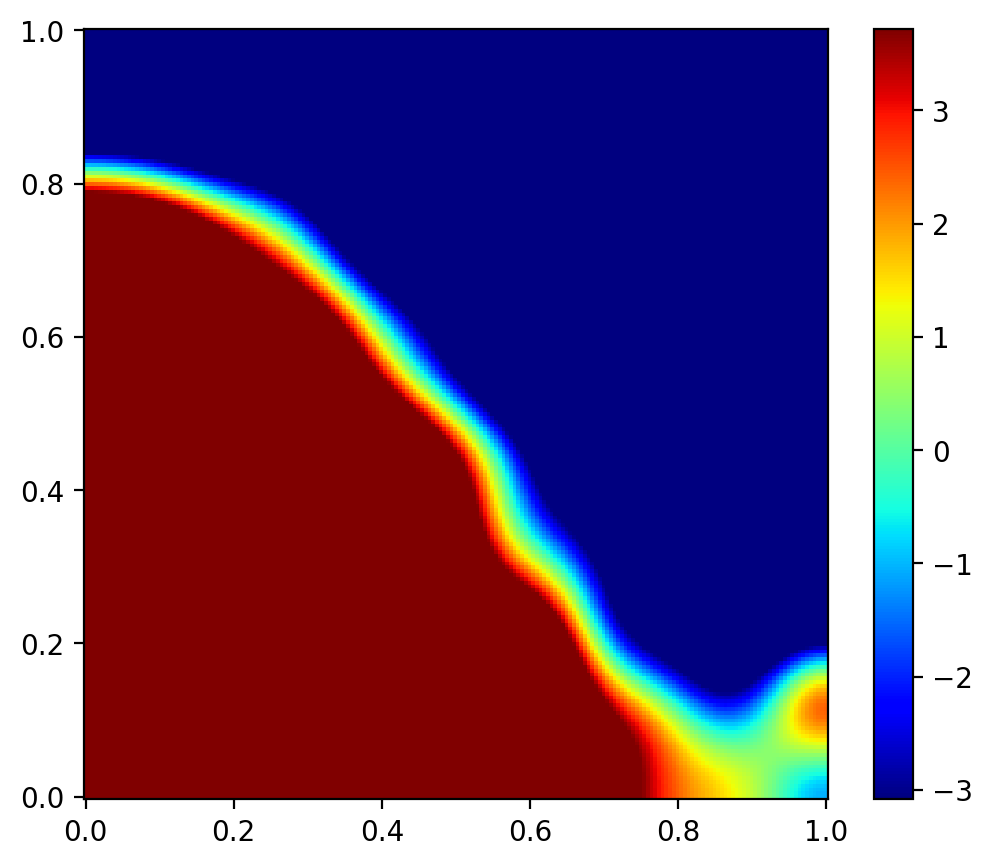}
            \put (10,87) {\footnotesize \bf DeepOnet-UKI-Direct}
        \end{overpic}
    \end{center}
    \begin{center}
        \begin{overpic}[width=0.3\textwidth]{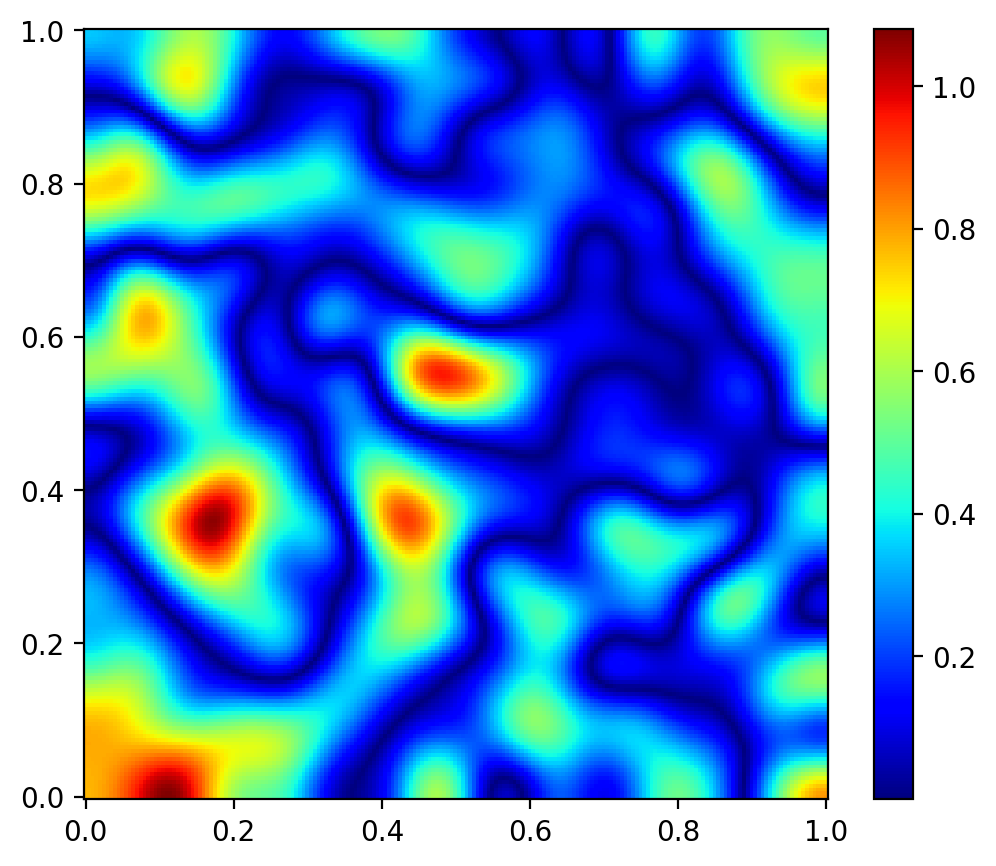}
        \end{overpic}
        \begin{overpic}[width=0.304\textwidth]{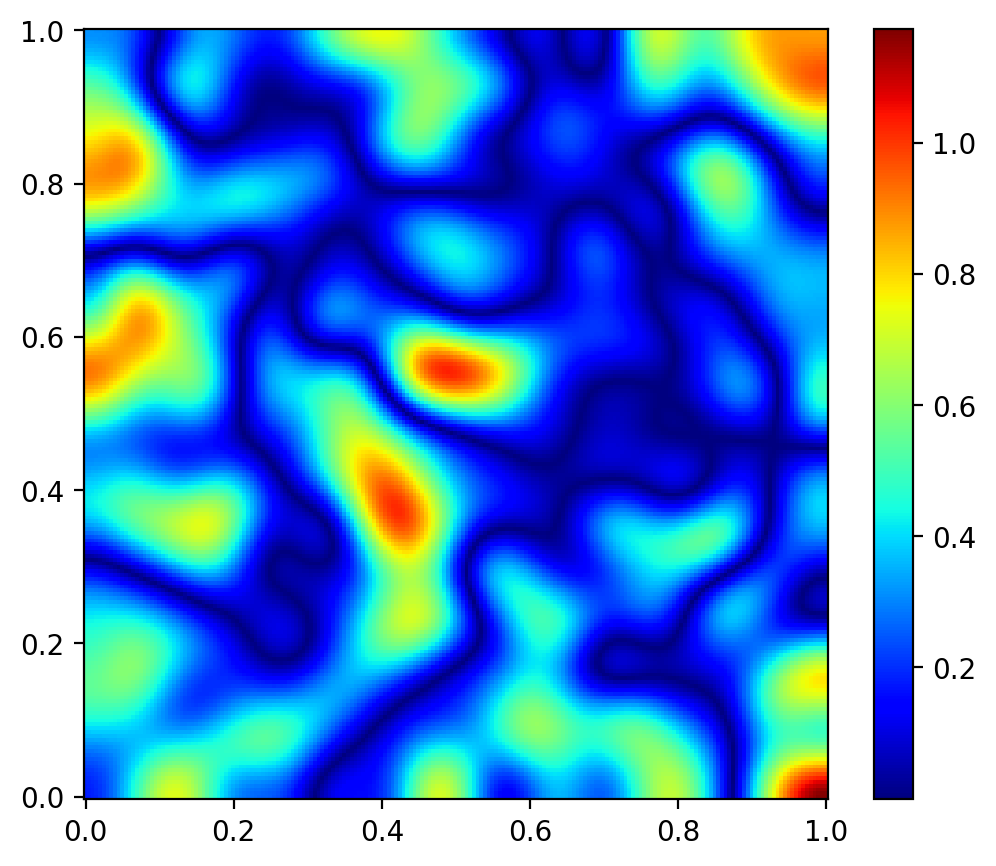}
        \end{overpic}
        \begin{overpic}[width=0.3\textwidth]{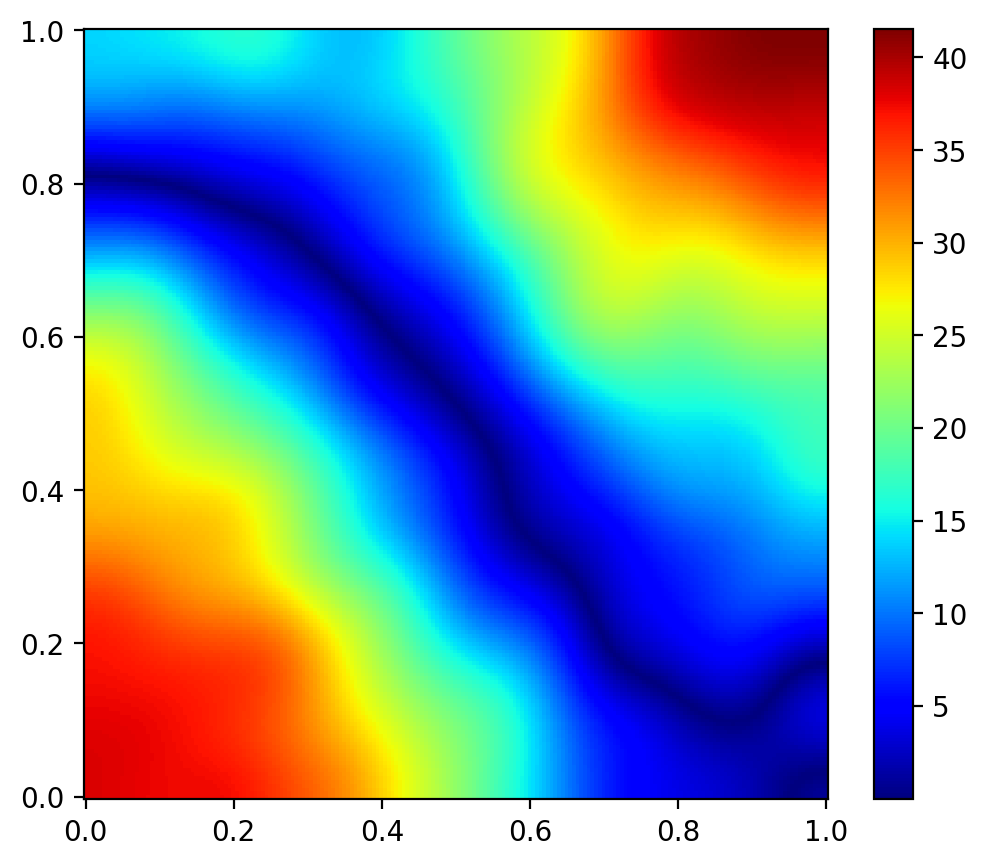}
        \end{overpic}
    \end{center}
    \caption{The inversion results for reaction diffusion case. Above: the estimated initial fields. Below: the absolute errors with respect to the true ones.}
    \label{kappa_diffusion}
\end{figure}
\begin{figure}[htbp]
    \begin{center}
        \begin{overpic}[width=0.3\textwidth]{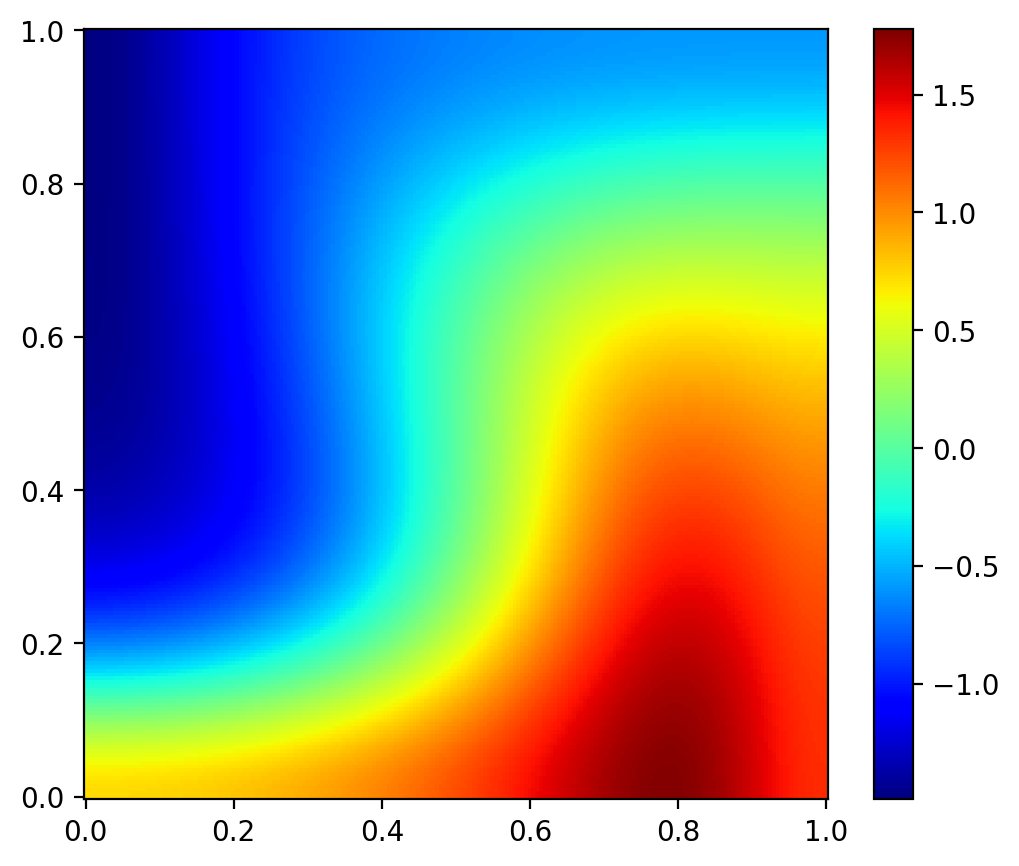}
            \put (30,85) {\footnotesize \bf FEM-UKI}
        \end{overpic}
        \begin{overpic}[width=0.3\textwidth]{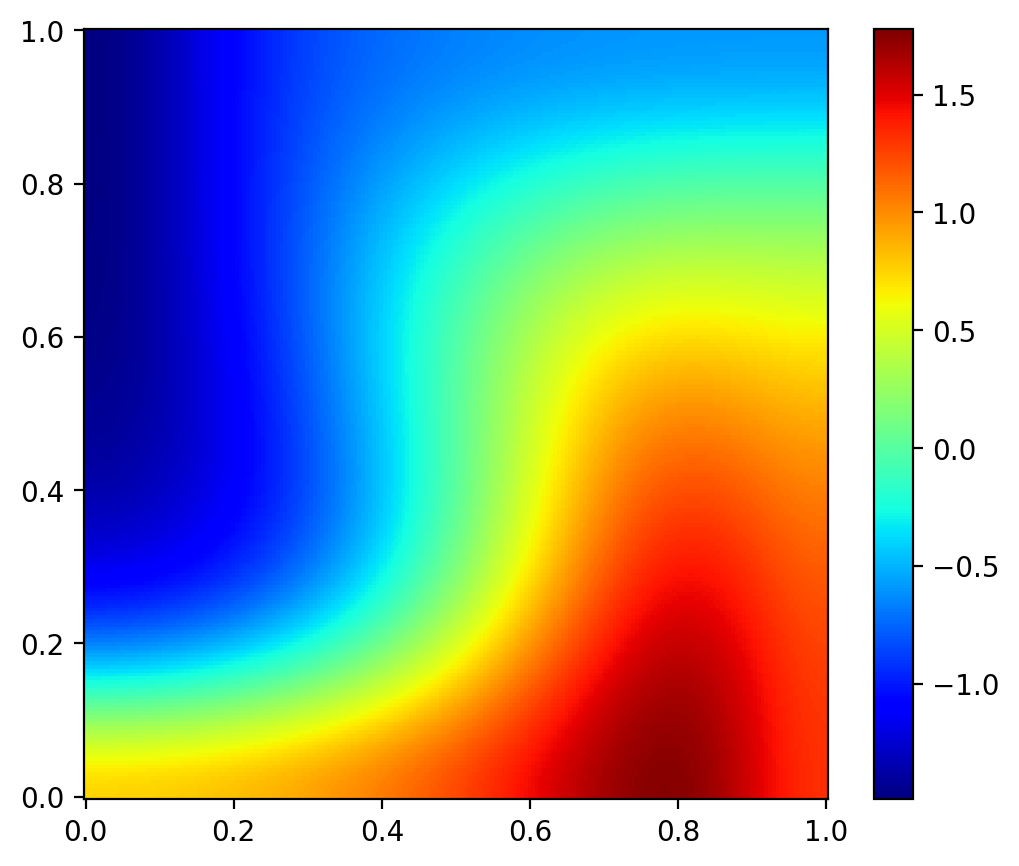}
            \put (5,85) {\footnotesize \bf DeepOnet-UKI-Adaptive}
        \end{overpic}
                \begin{overpic}[width=0.3\textwidth]{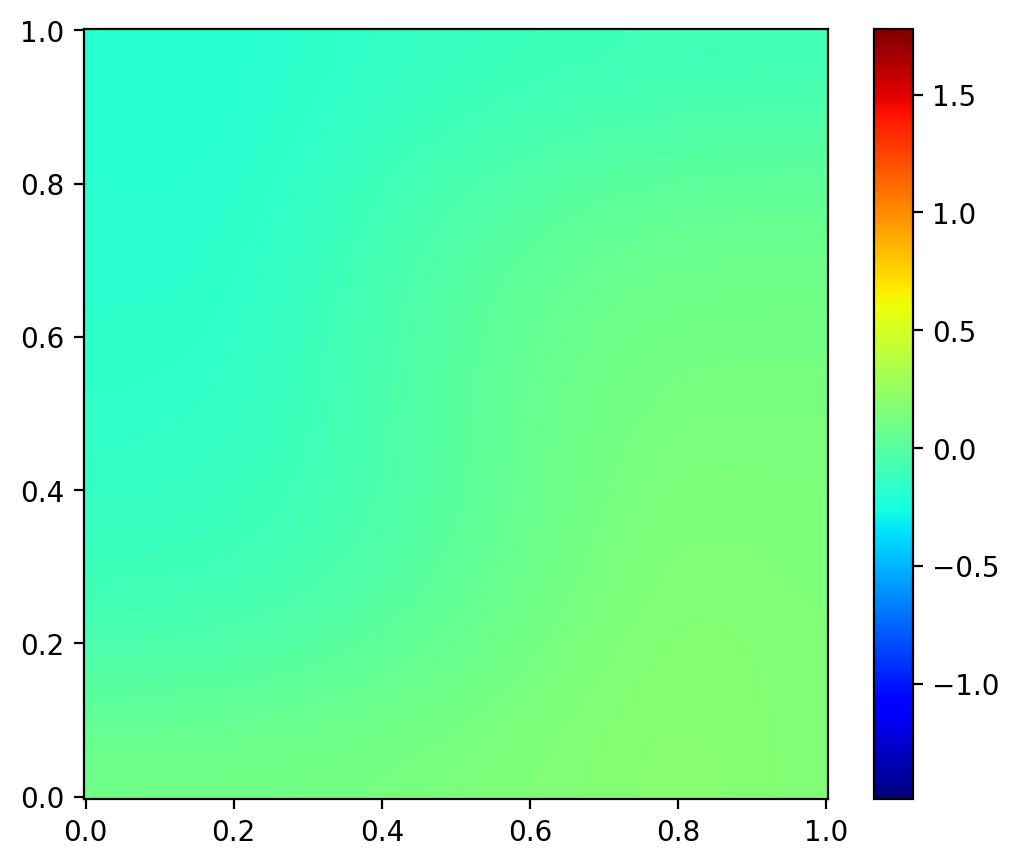}
            \put (10,85) {\footnotesize \bf DeepOnet-UKI-Direct}
        \end{overpic}
    \end{center}
    \begin{center}
        \begin{overpic}[width=0.3\textwidth]{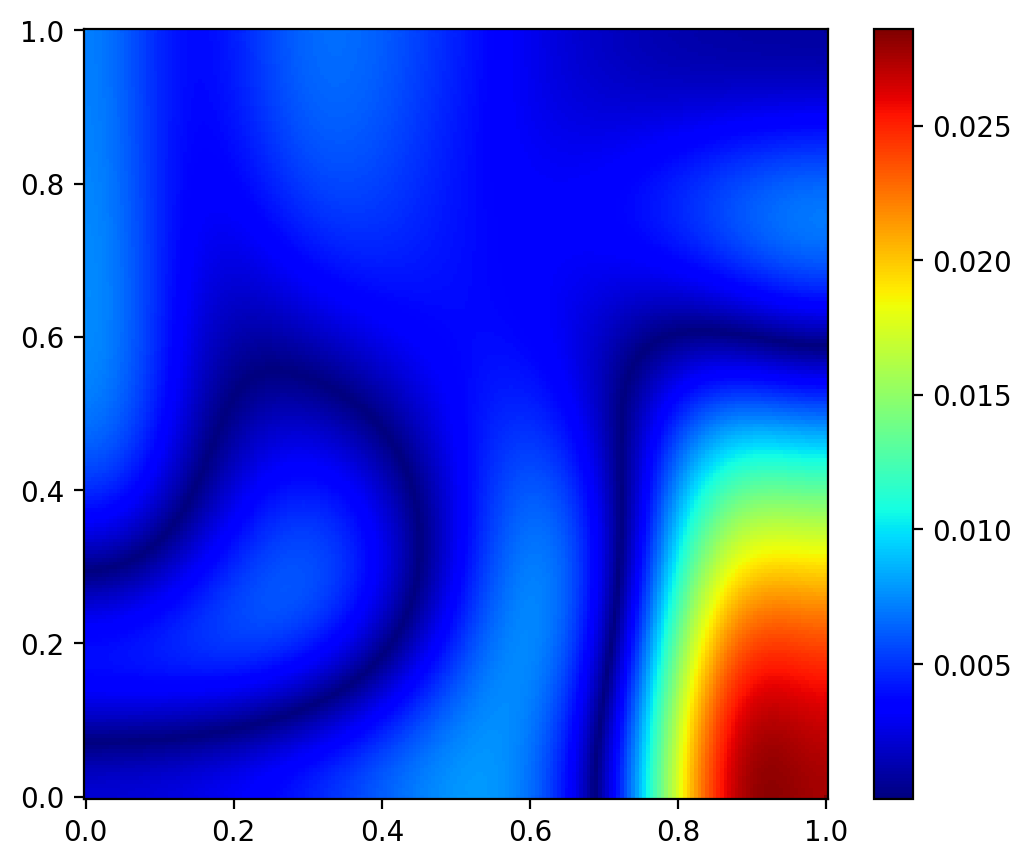}
        \end{overpic}
        \begin{overpic}[width=0.3\textwidth]{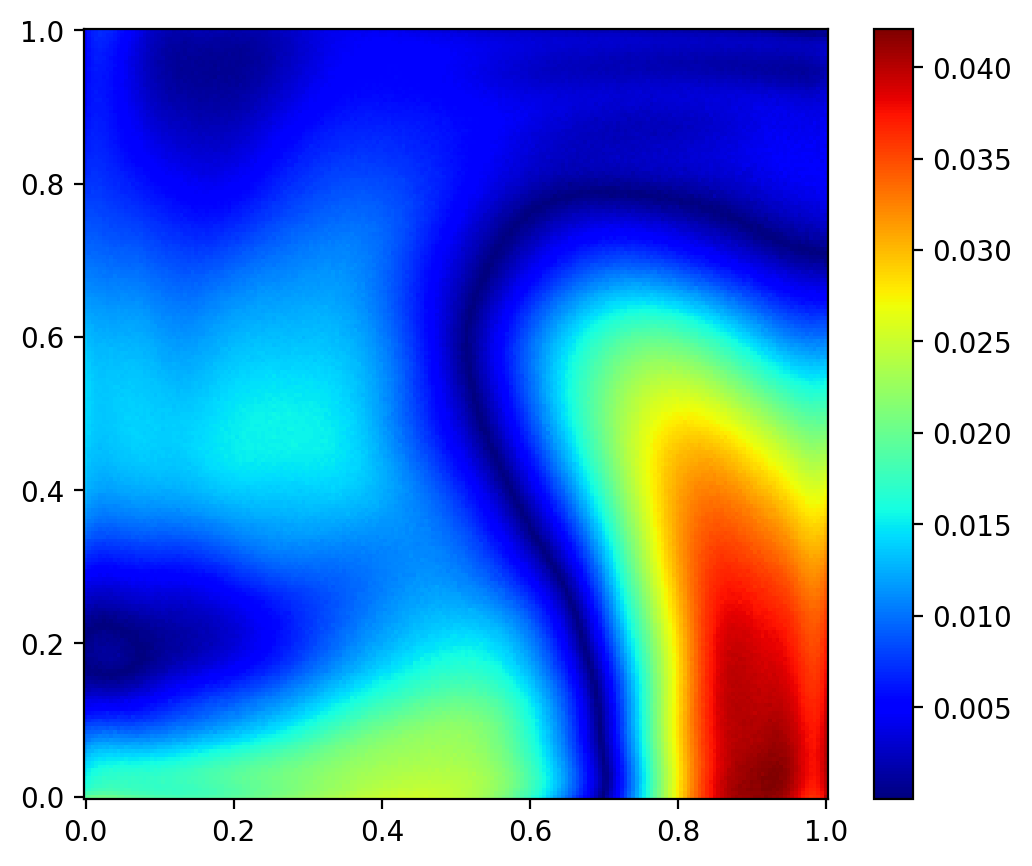}
        \end{overpic}
              \begin{overpic}[width=0.29\textwidth]{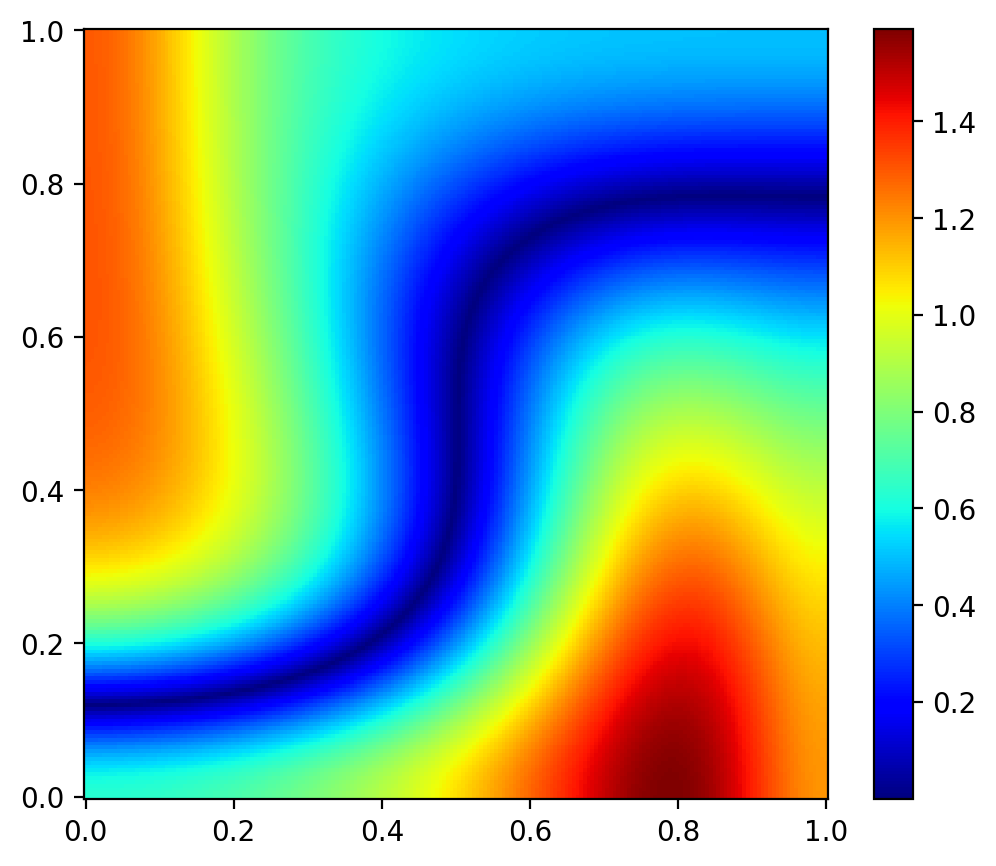}
        \end{overpic}
    \end{center}
    \caption{The inversion results for the reaction diffusion case. Above: the approximate states evaluated at the final estimated initial fields. Below: the absolute error between the approximate state and the true state.}
    \label{forward_diffusion}
\end{figure}

\begin{figure}[htbp]
    \centering 
    \begin{overpic}[width = 0.4\textwidth]{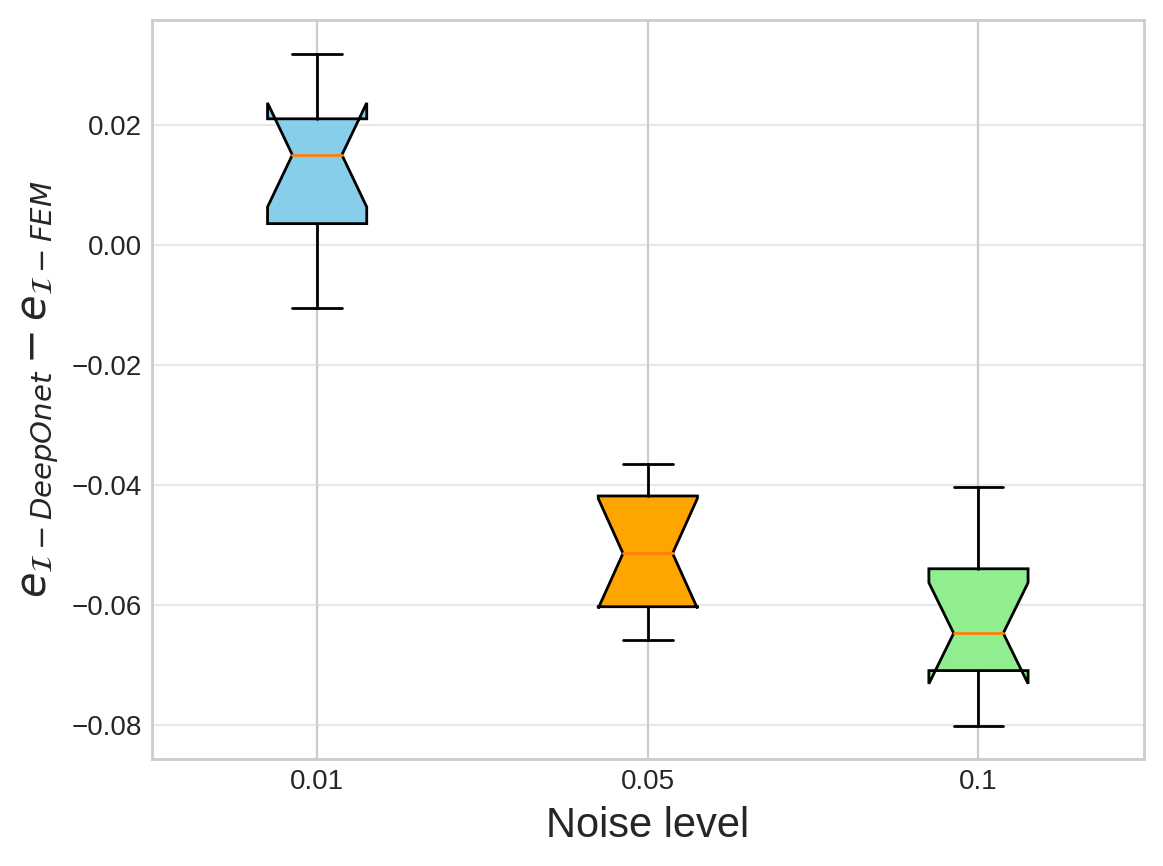}
    \end{overpic}
    \begin{overpic}[width = 0.41\textwidth]{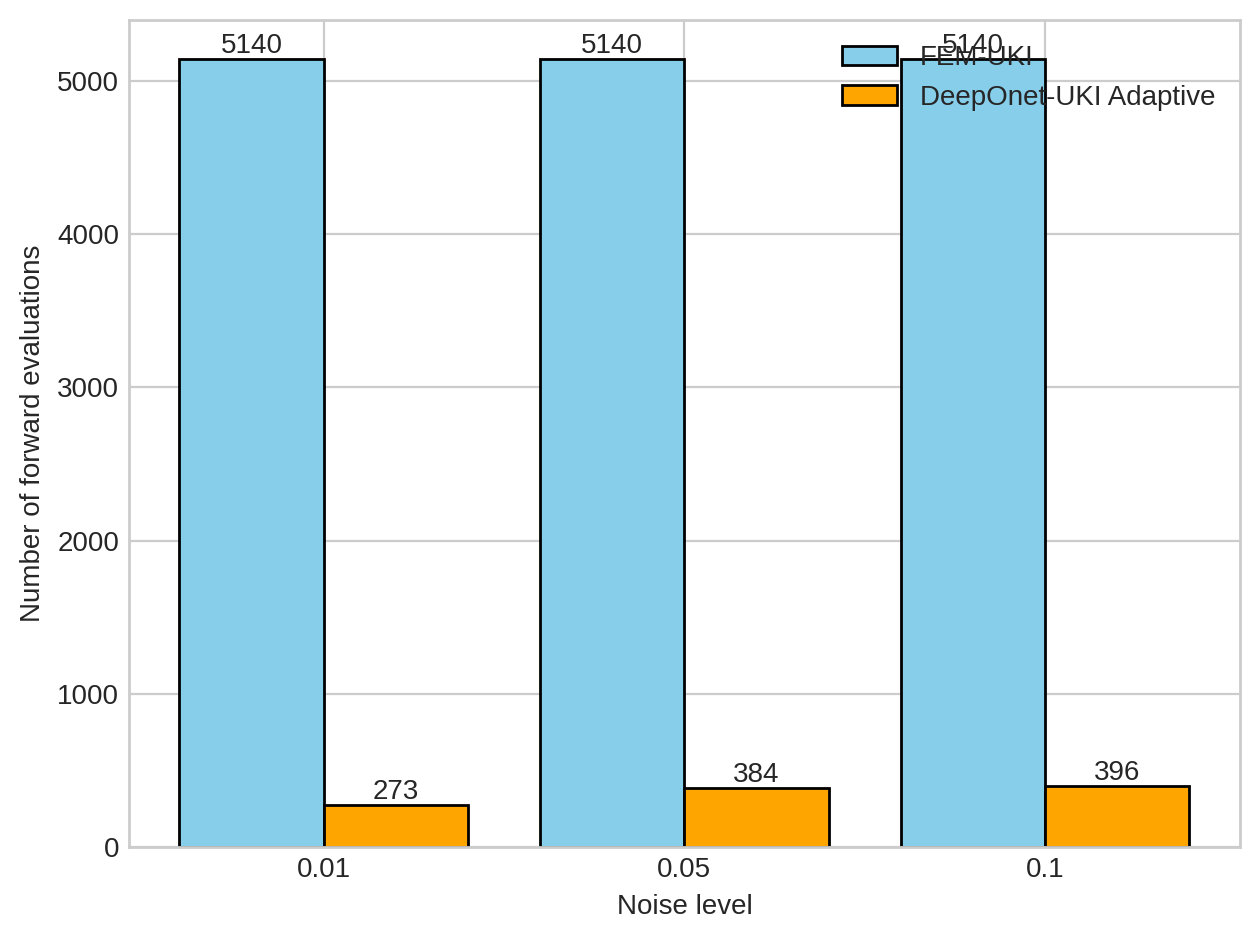}
    \end{overpic}
    \caption{Left: the error box plot of $e_{\mathcal{I}-DeepOnet} - e_{\mathcal{I}-FEM}$. Right: the mean total number of forward evaluations for FEM-UKI and DeepOnet-UKI-Adaptive.}
    \label{err_comparison_diffusion}
\end{figure}

We repeat the experiment with varying noise levels in order to thoroughly compare the performance of DeepOnet-UKI-Adaptive and FEM-UKI. We  repeat the experiment ten times for each noise level, varying the UKI initial values each time. The difference between the relative inversion errors $e_{\mathcal{I}-DeepOnet} - e_{\mathcal{I}-FEM}$ and the mean total number of forward evaluations is displayed in Fig \ref{err_comparison_diffusion}. It is evident that DeepOnet-UKI-Adaptive can even achieve smaller relative inversion errors than FEM-UKI when dealing with higher noise levels. 
Furthermore, our method has a very low computational cost. In comparison to DeepOnet-UKI-Adaptive, the total number of forward evaluations for FEM-UKI is at least ten times higher. That is to say, our approach can efficiently complete the inversion task with significantly lower computational cost once the initial model has been trained. This feature offers the possibility to handle real-time forecasts in  some data assimilation tasks. 

\section{Conclusion}
\label{conclusion}
We present an adaptive operator learning framework for iteratively reducing the model error in Bayesian inverse problems. In particular, the unscented Kalman inversion(UKI) is used to approximate the solution of inverse problems, and the DeepOnet is utilized to construct the surrogate. We suggest a greedy algorithm to choose the adaptive samples to retrain the approximate model. The performance of the proposed strategy has been illustrated by three numerical examples. It should be noted that our adaptive framework may be ineffective in certain situations due to the use of Kalman-based methods. On the one hand, for strictly non-Gaussian posterior distributions, such as those with multiple modes, the UKI fails to capture all of the modes because it approximates the posterior using a Gaussian distribution.  On the other hand, for some low- to moderate-dimensional problems, particularly when focusing on a specific task, the CES framework may be more suitable. Nonetheless, our strategy is intended to be adaptable to a broader range of scenarios. Using this framework, we can easily modify the posterior computation and surrogate methods to broaden the applicability of our method. While future work will address the potential drawbacks mentioned earlier.

\appendix
\section*{}
\textit{Proof of Theorem \ref{main_theorem}}:
We first consider the error estimate of the covariance matrix.  Using Eqs. \eqref{converge_a} $ and $ \eqref{surrogate_converge_a}, we have  
    \begin{equation}
        C_{\infty}^{-1} - \widehat{C}_{\infty}^{-1} = \underbrace{\mathcal{G}^{T}\Sigma_{\eta}^{-1}\mathcal{G} - \widehat{\mathcal{G}}^{T}\Sigma_{\eta}^{-1} \widehat{\mathcal{G}}}_{I_{1}}+\underbrace{(\alpha^{2}C_{\infty}+\Sigma_{\omega})^{-1}-(\alpha^{2}\widehat{C}_{\infty}+\Sigma_{\omega})^{-1}}_{I_{2}}.
    \end{equation}
    Note that the first part is proved in Eq.\eqref{divide}, i.e.,
    \begin{equation}
        \label{I_1}
        \|I_{1}\|_{2} \leq  2\epsilon H\|\Sigma_{\eta}^{-1}\|_{2}.
    \end{equation}
   We consider the second part. Let us assume that $\mathcal{B}$ represents the Banach spaces of matrices in $\mathbb{R}^{N_{m}}\times \mathbb{R}^{N_{m}}$. The operator norm in $\mathbb{R}^{N_{m}}$ is induced by the Euclidean norm. The Banach spaces of linear operators equipped with the operator norm are denoted by $\mathcal{L}:\mathcal{B}\rightarrow \mathcal{B}$.  If we define $f(X;\alpha):=(\alpha^{2}X^{-1} + \Sigma_{\omega})^{-1}$, then $I_{2} = f(C_{\infty}^{-1};\alpha) - f(\widehat{C}_{\infty}^{-1};\alpha)$. $Df(X;\alpha)$, the derivative of $f$, is defined by the direction $\Delta X\in \mathcal{B}$ as 
    \begin{equation}
        Df(X;\alpha)\Delta X = \alpha^{2}(\alpha^{2}\mathcal{I}+X\Sigma_{\omega})^{-1}\Delta X(\alpha^{2}\mathcal{I}+X\Sigma_{\omega})^{-1}.
    \end{equation}
    According to \cite{huang2022iterated}, $f$ is a contraction map in $\mathcal{B}$, such that we have 
    \begin{equation}
        \beta:= \sup_{X\in\mathcal{B}}\|Df(X;\alpha)\|_{2} <1.
    \end{equation}
    Therefore, we can use the \textit{Mean Value Theorem} in matrix functions to get that 
    \begin{equation}
        \label{I_2}
        \begin{split}
        \|I_{2}\|_{2} &= \|f(C_{\infty}^{-1};\alpha) - f(\widehat{C}_{\infty}^{-1};\alpha)\|_{2}\ \\
        &\leq \beta \|C_{\infty}^{-1} - \widehat{C}_{\infty}^{-1}\|_{2}.
        \end{split}
    \end{equation}
    Combining Eqs.\eqref{I_1} and \eqref{I_2} yields 
    \begin{equation}
        \|C_{\infty}^{-1} - \widehat{C}_{\infty}^{-1}\|_{2}\leq 2\epsilon H\|\Sigma_{\eta}^{-1}\|_{2} + \beta \|C_{\infty}^{-1} - \widehat{C}_{\infty}^{-1}\|_{2}.
    \end{equation}
Then we can have the error estimate of the covariance matrix 
    \begin{equation}
        \|C_{\infty}^{-1} - \widehat{C}_{\infty}^{-1}\|_{2} \leq \frac{2\epsilon H}{1-\beta}\|\Sigma_{\eta}^{-1}\|_{2}.
    \end{equation}
We now take into consideration the error estimate of the mean vector. Using Eqs. \eqref{converge_b} and \eqref{surrogate_converge_b}, we obtain
\begin{equation}
    \label{subtraction_cov}
    \begin{split}
    C_{\infty}^{-1}r_{\infty} - \widehat{C}_{\infty}^{-1}\widehat{r}_{\infty}  &= (\mathcal{G}^{T} - \widehat{\mathcal{G}}^{T})\Sigma_{\eta}^{-1}y+(\alpha^{2}C_{\infty}+\Sigma_{\omega})^{-1}\alpha r_{\infty} -(\alpha^{2}\widehat{C}_{\infty}+\Sigma_{\omega})^{-1}\alpha\widehat{r}_{\infty}\\ 
    & = (\mathcal{G}^{T} - \widehat{\mathcal{G}}^{T})\Sigma_{\eta}^{-1}y + \alpha r_{\infty} f(C_{\infty}^{-1};\alpha) - \alpha \widehat{r}_{\infty}f(\widehat{C}_{\infty}^{-1};\alpha).
    \end{split}
\end{equation}
Since 
\begin{equation}
    \label{divide_left}
    \begin{split}
    C_{\infty}^{-1}r_{\infty} - \widehat{C}_{\infty}^{-1}\widehat{r}_{\infty} &= C_{\infty}^{-1}r_{\infty} - C_{\infty}^{-1}\widehat{r}_{\infty} + C_{\infty}^{-1}\widehat{r}_{\infty} - \widehat{C}_{\infty}^{-1}\widehat{r}_{\infty}\\ 
    & = C_{\infty}^{-1}(r_{\infty} - \widehat{r}_{\infty}) + (C_{\infty}^{-1} - \widehat{C}_{\infty}^{-1})\widehat{r}_{\infty},
\end{split}
\end{equation}
and 
\begin{equation}
    \label{divide_right}
    \begin{split}
        \alpha r_{\infty} f(C_{\infty}^{-1};\alpha) - \alpha \widehat{r}_{\infty}f(\widehat{C}_{\infty}^{-1};\alpha)
        &= \alpha (r_{\infty}-\widehat{r}_{\infty}) f(C_{\infty}^{-1};\alpha)  \\
        & + \alpha \widehat{r}_{\infty} \left(f(C_{\infty}^{-1};\alpha)- f(\widehat{C}_{\infty}^{-1};\alpha) \right). 
    \end{split}
\end{equation}
We can obtain  
\begin{equation}
    \label{mean}
    \begin{split}
   \left(C_{\infty}^{-1} - \alpha f(C_{\infty}^{-1};\alpha)\right)(r_{\infty} - \widehat{r}_{\infty}) &= \underbrace{(\mathcal{G}^{T} - \widehat{\mathcal{G}}^{T})\Sigma_{\eta}^{-1}y}_{I_{3}} - \underbrace{(C_{\infty}^{-1} - \widehat{C}_{\infty}^{-1})\widehat{r}_{\infty}}_{I_{4}} \\
    &+ \underbrace{\alpha \widehat{r}_{\infty} (f(C_{\infty}^{-1};\alpha)- f(\widehat{C}_{\infty}^{-1};\alpha) )}_{I_{5}}.
    \end{split}
\end{equation}
For the first part $I_3$, we have 
\begin{equation}
    \label{mean_I_1}
    \|I_{3}\|_{2} \leq \|\mathcal{G}^{T} - \widehat{\mathcal{G}}^{T}\|_{2}\|\Sigma_{\eta}^{-1}y\|_{2} \leq \epsilon \|\Sigma_{\eta}^{-1}y\|_{2}.
\end{equation}
And then the second part, 
\begin{equation}
    \label{mean_I_2}
    \|I_{4}\|_{2} \leq \|C_{\infty}^{-1} - \widehat{C}_{\infty}^{-1}\|_{2}\|\widehat{r}_{\infty}\|_{2}.
\end{equation}
For the last part, according to Eq.\eqref{I_2} we have 
\begin{equation}
    \label{mean_I_3}
    \begin{split}
    \|I_{5}\|_{2} &\leq \alpha\|\widehat{r}_{\infty}\|_{2}\|f(C_{\infty}^{-1};\alpha)- f(\widehat{C}_{\infty}^{-1};\alpha)\|_{2}\\ 
    &\leq \alpha\beta \|\widehat{r}_{\infty}\|_{2}\|C_{\infty}^{-1} - \widehat{C}_{\infty}^{-1}\|_{2}.
    \end{split}
\end{equation}
Moreover, from Eq.\eqref{converge_a} we have 
\begin{equation}
    \label{mean_I_4}
    C_{\infty}^{-1} - \alpha f(X_{\infty}^{-1};\alpha) = (1-\alpha)(\alpha^{2}C_{\infty} + \Sigma_{\omega})^{-1} + \mathcal{G}^{T}\Sigma_{\eta}^{-1}\mathcal{G}\succ 0.
\end{equation}
Combining Eqs.\eqref{mean}-\eqref{mean_I_4}, we have 
\begin{equation}
    \label{mean_error}
    \begin{split}
    \|r_{\infty} - \widehat{r}_{\infty}\|_{2} &\leq \|\left((1-\alpha)(\alpha^{2}C_{\infty} + \Sigma_{\omega})^{-1} + \mathcal{G}^{T}\Sigma_{\eta}^{-1}\mathcal{G}\right)^{-1}\|_{2}
        \|I_{3} - I_{4} + I_{5}\|_{2}\\ 
    & \leq \|(\mathcal{G}^{T}\Sigma_{\eta}^{-1}\mathcal{G})^{-1}\|_{2}\left(
        \epsilon \|\Sigma_{\eta}^{-1}y\|_{2} + (1 + \alpha\beta) \|\widehat{r}_{\infty}\|_{2}\|C_{\infty}^{-1} - \widehat{C}_{\infty}^{-1}\|_{2}
    \right)\\ 
    & \leq \|(\mathcal{G}^{T}\Sigma_{\eta}^{-1}\mathcal{G})^{-1}\|_{2}\left(
        \|\Sigma_{\eta}^{-1}y\|_{2} +\frac{2(1 + \alpha\beta)  H}{1-\beta} \|\widehat{r}_{\infty}\|_{2}\|\Sigma_{\eta}^{-1}\|_{2}
    \right) \epsilon.
    \end{split}
\end{equation}
Note that by Eq.\eqref{surrogate_converge}, we have 
\begin{equation}
    \begin{split}
    \left(\widehat{C}_{\infty}^{-1} - \alpha(\alpha^{2}\widehat{C}_{\infty}+\Sigma_{\omega})^{-1}\right)\widehat{r}_{\infty}&=\left(\widehat{\mathcal{G}}^{T}\Sigma_{\eta}^{-1}\widehat{\mathcal{G}} + (1-\alpha)(\alpha^{2}\widehat{C}_{\infty}+\Sigma_{\omega})^{-1}\right)\widehat{r}_{\infty}\\
    &=\widehat{\mathcal{G}}^{T}\Sigma_{\eta}^{-1}y.
    \end{split}
\end{equation}
Afterwards, we can get the bound of $\widehat{r}_{\infty}$ as  
\begin{equation}
    \label{r_bound}
    \begin{split}
    \|\widehat{r}_{\infty}\|_{2}&\leq \left\|\left(\widehat{\mathcal{G}}^{T}\Sigma_{\eta}^{-1}\widehat{\mathcal{G}} + (1-\alpha)(\alpha^{2}\widehat{C}_{\infty}+\Sigma_{\omega})^{-1}\right)^{-1}\right\|_{2}\|\widehat{\mathcal{G}}^{T}\Sigma_{\eta}^{-1}y\|_{2}\\ 
    &  \leq \left\|\left(\widehat{\mathcal{G}}^{T}\Sigma_{\eta}^{-1}\widehat{\mathcal{G}}\right)^{-1} \right\|_{2}\|\widehat{\mathcal{G}}^{T}\Sigma_{\eta}^{-1}y\|_{2}\\
    &\leq H\left\|\left(\widehat{\mathcal{G}}^{T}\Sigma_{\eta}^{-1}\widehat{\mathcal{G}}\right)^{-1} \right\|_{2}\|\Sigma_{\eta}^{-1}y\|_{2}.
    \end{split}
\end{equation}
Combining Assumption \ref{Assumption3} and Eqs.\eqref{mean_error} and \eqref{r_bound}, we can get 
\begin{equation}
    \begin{split}
    \|r_{\infty} - \widehat{r}_{\infty}\|_{2}&\leq \|(\mathcal{G}^{T}\Sigma_{\eta}^{-1}\mathcal{G})^{-1}\|_{2}\left(
        \|\Sigma_{\eta}^{-1}y\|_{2} +\frac{2(1 + \alpha\beta)  H}{1-\beta} \|\widehat{r}_{\infty}\|_{2}\|\Sigma_{\eta}^{-1}\|_{2}\right)\epsilon \\ 
        & \leq \|(\mathcal{G}^{T}\Sigma_{\eta}^{-1}\mathcal{G})^{-1}\|_{2}\|\Sigma_{\eta}^{-1}y\|_{2}\left(1 + \frac{2(1 + \alpha\beta)  H^{2}}{1-\beta}\left\|\left(\widehat{\mathcal{G}}^{T}\Sigma_{\eta}^{-1}\widehat{\mathcal{G}}\right)^{-1} \right\|_{2}\|\Sigma_{\eta}^{-1}\|_{2}\right)\epsilon \\ 
        & \leq \frac{K_{1}\|\Sigma_{\eta}^{-1}y\|_{2}}{\|\mathcal{G}^{T}\Sigma_{\eta}^{-1}\mathcal{G}\|_{2}}\left(1 + \frac{2(1 + \alpha\beta) K_{2} H^{2}}{1-\beta}\frac{\|\Sigma_{\eta}^{-1}\|_{2}}{\left\|\widehat{\mathcal{G}}^{T}\Sigma_{\eta}^{-1}\widehat{\mathcal{G}} \right\|_{2}}\right)\epsilon\\ 
        &\leq \frac{K_{1}H_{\eta}H_{y}}{C_{1}}\left(1 + \frac{2(1 + \alpha\beta) K_{2}H_{\eta} H^{2}}{(1-\beta)C_{2}}\right)\epsilon,
    \end{split} 
\end{equation}
where $K_{1}, K_{2}$ is the upper bound of the condition number of $\mathcal{G}^{T}\Sigma_{\eta}^{-1}\mathcal{G}, \widehat{\mathcal{G}}^{T}\Sigma_{\eta}^{-1}\widehat{\mathcal{G}}$ respectively and $H_{\eta}:=\|\Sigma_{\eta}^{-1}\|_{2}, H_{y} = \|y\|_{2}$.
\end{proof}

\section*{Acknowledgment}
The authors would like to thank anonymous referees for their many insightful and constructive comments and suggestions that improved the organization and quality of our paper essentially.

\end{document}